\documentclass[a4paper,11pt]{article}

\usepackage{amsthm, amsmath, amsfonts, amssymb, mathrsfs, xfrac, algorithm}
\usepackage[colorlinks]{hyperref}
    
\newtheorem{theorem}{Theorem}[section]
   
\newtheorem{remark}[theorem]{Remark}
\newtheorem{proposition}{Proposition}
\newtheorem{lemma}[theorem]{Lemma}
\newtheorem{problem}[theorem]{Problem}

\usepackage{bm}
\newcommand{\vect}[1]{\boldsymbol{\mathbf{#1}}} 
\renewcommand{\div}[1]{\operatorname{div}\left(#1\right)}
\newcommand{\curl}{\nabla \times } 
\newcommand{\ferj}[1]{{\color{black}{#1}}}
\newcommand{\fergy}[1]{{\color{black}{#1}}}

\newcommand{\correction}[1]{{\color{black}{#1}}}
\DeclareMathOperator*{\argmin}{arg\,min}

\newcommand{\un}{u_{N}}
\newcommand{\ud}{u_{D}}
\newcommand{\uo}{\un} 
\newcommand{\uok}{\un^{k}} 
\newcommand{\uokp}{\un^{k+1}} 
\newcommand{\dotun}{\dot{u}_{N}}
\newcommand{\dotud}{\dot{u}_{D}}

\newcommand{\unp}{u^{\prime}_{N}}
\newcommand{\udp}{u^{\prime}_{D}}
\newcommand{\JD}{J_{D}}
\newcommand{\JN}{J_{N}}
\newcommand{\pn}{p_{N}}
\newcommand{\pd}{p_{D}}
\newcommand{\HGamma}{\ferj{V_{\Gamma}(\Omega)}}
\newcommand{\HSigma}{\ferj{V_{\Sigma}(\Omega)}}

\newcommand{\interval}{I}
\newcommand{\constant}{c}
\newcommand{\dd}{\delta}
\newcommand{\bb}{\vect{b}}

\newcommand{\sfTheta}{\vect{\Theta}}

\newcommand{\VV}{\vect{\theta}}
\newcommand{\Vn}{\theta_n} 
\newcommand{\nn}{\vect{n}}

\newcommand{\dn}[1]{\partial_{\nn}{#1}}

\newcommand{\intO}[1]{\int_{\Omega}{#1}{\, {d} x}}
\newcommand{\intdO}[1]{\int_{\partial\Omega}{#1}{\, {d} s}}
\newcommand{\intOt}[1]{\int_{\Omega_{t}}{#1}{\, {d} x_{t}}} 

\newcommand{\intS}[1]{\int_{\Sigma}{#1}{\, {d} s}}  
\newcommand{\intG}[1]{\int_{\Gamma}{#1}{\, {d} s}}

\newcommand{\norm}[1]{\left\|{#1}\right\|} 
\newcommand{\abs}[1]{\left|{#1}\right|}   
\begin{document} 
\title{Numerical solution by shape optimization method to an inverse shape problem in multi-dimensional advection-diffusion problem with space dependent coefficients}
\author{Elmehdi Cherrat$^{\ast}$, Lekbir Afraites$^{\ast}$ and Julius Fergy T. Rabago$^{\dagger}$}

\date{%
	{\footnotesize
        $^\ast$Mathematics and Interactions Teams (EMI)\\%
        Faculty of Sciences and Techniques\\%
        Sultan Moulay Slimane University\\%
        Beni Mellal, Morocco\\\vspace{-2pt}
       	\texttt{${}^{1}$cherrat.elmehdi@gmail.com, \, ${}^{2}$l.afraites@usms.ma}\\[2ex]
	$^{\dagger}$Faculty of Mathematics and Physics\\%
	 Institute of Science and Engineering\\%
         Kanazawa University, Kanazawa 920-1192, Japan\\\vspace{-2pt}
        \texttt{jftrabago@gmail.com}}\\[2ex]           
    \today
}

\maketitle

\begin{abstract}
This work focuses on numerically solving a shape identification problem related to advection-diffusion processes with space-dependent coefficients using shape optimization techniques. 
Two boundary-type cost functionals are considered, and their corresponding variations with respect to shapes are derived using the adjoint method, employing the chain rule approach. 
This involves firstly utilizing the material derivative of the state system and secondly using its shape derivative. 
Subsequently, \fergy{ an alternating direction  method of multipliers (ADMM)} combined with the Sobolev-gradient-descent algorithm is applied to stably solve the shape reconstruction problem. 
Numerical experiments in two and three dimensions are conducted to demonstrate the feasibility of the methods.

\medskip

\end{abstract}

\section{Introduction}
\label{sec:Introduction} 
Understanding the transport of quantities like heat, mass, or momentum is crucial for predicting and controlling various natural and engineered systems. 
In advection-diffusion problems, transport primarily occurs through two mechanisms: advection, driven by fluid flow, and diffusion, caused by random molecular motion. 
For instance, pollutants in rivers are carried downstream by water flow (advection) but also spread out due to diffusion. 
The behavior of advection-diffusion systems varies depending on whether advection or diffusion dominates. 
In advection-dominated regimes, rapid transport occurs with sharp gradients, while diffusion-dominated regimes lead to slower, smoother dispersion. 
This paper examines a shape identification problem for advection-diffusion problems with space-dependent advection coefficient through shape optimization methods.

Various shape identification problems have been previously investigated in several domains, including inverse obstacle scattering problems \cite{Hettlich1998,HettlichRundell1998}, inverse conduction scattering problems \cite{KressRundell2001}, static and time-dependent inverse boundary problems involving perfectly conducting or insulating inclusions \cite{AfraitesDambrineKateb2007,ChapkoKressYoon1998,ChapkoKressYoon1999,HarbrechtTausch2011,HarbrechtTausch2013,YanMa2006}, shape detection in convection-diffusion problems \cite{YanHouGao2017} and unsteady advection-diffusion problems \cite{YanSuJing2014}, inverse geometric source problems \cite{AfraitesMasnaouiNachaoui2022}, boundary shape reconstructions with Robin conditions \cite{AfraitesRabago2025,Fang2022,FangZeng2009,RabagoAzegami2019a}, obstacle reconstruction in Stokes fluid flow \cite{CaubetDambrineKateb2013,CaubetDambrineKatebTimimoun2013,RabagoAfraitesNotsu2025,YanMa2008}, and more.

In this paper, we aim to identify the shape and location of an object $\omega \subset \mathbb{R}^{d}$, where $d \in \{2,3\}$, within a larger domain $D \subset \mathbb{R}^{d}$ using a pair of known datasets $(f,g)$ observed at the accessible boundary of $D$. 
Denoting by $\Omega$ the connected domain $D \setminus \overline{\omega}$, we investigate here a multi-dimensional advection-diffusion equation within the context of shape inverse problems. 
The equation describes, for example, the transport of contaminants in surface water.

Mathematically, given a simply connected domain $D \subset \mathbb{R}^{d}$ and assuming the existence of an unknown simply connected inclusion $\omega$, as well as functions $f$ and $g$ defined over the boundary $\partial D$, we are primarily interested in the shape optimization reformulation of the following overdetermined advection-diffusion problem:
\begin{equation}\label{eq:overdetermined}
\left\{
\begin{array}{rcll}
    -\div{\sigma\nabla{u}}+\bb \cdot \nabla{u} & = & 0 & \text{in }\;\; \Omega,\\
    u & = & f & \text{on }\;\Sigma,\\
    {\sigma}\dn{u} & = & g & \text{on }\;\Sigma,\\
    u & = & 0 & \text{on }\;\Gamma,
\end{array}
\right.
\end{equation}
where $\Sigma := \partial{D}$, $\Gamma := \partial{\omega}$, $\sigma:=\sigma(x)$, $x \in D$, denotes the diffusion coefficient, $u:=u(x)$, $x\in \Omega$, is the concentration of the contaminant, $\bb:=\bb(x)$, $x \in D $, the velocity of the fluid flow, and $\dn{u} $ is the outward normal derivative of $u$ on $\Sigma$.
\fergy{We assume, in the general setup,} that the coefficients $\sigma$ and $\bb = (b_{1}, \ldots, b_{d})^{\top}$ satisfy the following conditions:
%
%
%
%
\begin{equation}\label{eq:assumptions}
\left\{\
\begin{aligned}
	&\text{{$\sigma \in W^{1,\infty}(D)^{d \times d}$} and there exists $\sigma_{0}>0$ (ellipticity constant) such that for all $\xi \in \mathbb{R}^{d}$,}\\
	&\qquad\qquad \text{$\sum_{i,j = 1}^{2} \sigma_{ij} \xi_{i} \xi_{j} \geqslant \sigma_{0} \|\xi\|_{d}^{2}$, \qquad almost everywhere in $D$;}\\
	&\text{and there is a constant $b_{0} > 0$ such that, for $i\in \{1,\ldots,d\}$, $b_{0} < b_{i} \in {W}^{1,\infty}(D)$.}
\end{aligned}
\right.\tag{A}
\end{equation}
\fergy{To simplify the analysis and avoid cumbersome notation, we assume, unless otherwise specified, that $\sigma$ is a space-dependent scalar function (i.e., $\sigma \in W^{1,\infty}(D)^{1\times 1}$).}

%
%
%
%

\fergy{We assume, unless otherwise stated, that $f \in H^{3/2}(\Sigma)$, $f \not\equiv 0$.
Also, we let $g \in H^{1/2}(\Sigma)$ be an admissible boundary measurement corresponding to $f$.
This means that $g$ belongs to the image of the Dirichlet-to-Neumann map $\Lambda_{\Sigma}: f \in H^{3/2}(\Sigma) \mapsto \dn{\ud} \in H^{1/2}(\Sigma)$, where $\ud$ solves boundary value problem 
%
%
\begin{equation}\label{eq:state_ud}
\left\{
\begin{array}{rcll}
    -\div{\sigma\nabla \ud}+\bb \cdot \nabla \ud & = & 0 & \text{in $\Omega$},\\
    \ud & = & 0 & \text{on $\Gamma$},\\
    \ud & = & f & \text{on $\Sigma$}.
    \end{array}
\right.
\end{equation} 

If $g \in H^{1/2}(\Sigma)$ is given instead, then we let $f \in H^{3/2}(\Sigma)$ be an admissible boundary measurement by taking $f$ as the image of the Neumann-to-Dirichlet map $\Lambda_{\Sigma}^{-1}: g \in H^{1/2}(\Sigma) \mapsto \un=: f \in H^{3/2}(\Sigma)$, where $\un$ solves the partial differential equation (PDE) system
%
%
\begin{equation}\label{eq:state_un}
\left\{
\begin{array}{rcll}
    -\div{\sigma\nabla \un}+\bb \cdot \nabla \un & = & 0 & \text{in }\;\; \Omega,\\
    \un & = & 0 & \text{on }\;\Gamma,\\
    \sigma\dn{\un} & = & g & \text{on }\;\Sigma.
\end{array}
\right.
\end{equation} 
}

Our main objective then is to examine the inverse geometry problem that reads as follows:  
\[
	\text{Given $D$, $f$, and $g$, find}~~ \omega~~ \text{such that}~~u(D\setminus\overline{\omega}) ~~\text{satisfies}~~ \eqref{eq:overdetermined}.
\]
This problem has been studied extensively in the literature. Yan et al. \cite{YanHouGao2017} addressed shape identification in convection-diffusion problems using the adjoint method, while Yan et al. \cite{YanSuJing2014} applied the domain derivative method to unsteady advection-diffusion problems. 
Fernandez et al. \cite{Fernandezetal2021} used the topological derivative method for pollution source reconstruction governed by a steady-state convection-diffusion equation.

In this work, we reformulate the inverse problem as an optimal control problem where the shape is the unknown. 
We propose two least-squares misfit functionals: the \textit{tracking-the-Dirichlet-data} \eqref{eq:Dirichlet_cost_function} and \textit{tracking-the-Neumann-data} \eqref{eq:Neumann_cost_function}.
We rigorously analyze the optimization problem, derive the material derivative of the state, and use it to compute the shape gradient of these functionals. 
By introducing adjoint systems, we express the shape gradient without requiring state derivatives. 
For numerical experiments, we apply the Alternating Direction Method of Multipliers (ADMM), following the approach in \cite{RabagoHadriAfraitesHendyZaky2024}. 
This method effectively handles challenges such as noise and concave regions on unknown interior boundaries.
Our goal is to enhance shape optimization techniques by incorporating an auxiliary variable into the cost functionals \eqref{eq:Dirichlet_cost_function} and \eqref{eq:Neumann_cost_function}, which are iteratively minimized using ADMM.


The rest of the paper is organized as follows.
In Section \ref{sec:problem_setting}, we delve deeper into the problem configuration and discuss the reformulation of the associated shape optimization. 
Section \ref{sec:shape_derivatives} provides a rigorous demonstration of the existence of the material derivative of the state variables, along with detailing the equation verified by the shape derivative of the state variables. 
Following this, the latter part of the section characterizes the shape gradients of the considered functionals, first through the material derivative and then using the shape derivative of the state. 
Subsequently, in Section \ref{sec:numerical_experiments}, we present an algorithm based on gradient methods, solving an elliptic problem to determine the steepest descent direction in the $H^1$ space. 
This section presents a comprehensive series of numerical experiments aimed at exploring various shapes across two and three spatial dimensions. 
It includes both the results from conventional shape optimization methods and ADMM, and offers a comparative analysis of these methods with a particular focus on their accuracy in three-dimensional cases.
Finally, Section \ref{sec:conclusion} concludes the paper briefly, summarizing key findings and highlighting the major implications of this study.

\section{The problem setting}
\label{sec:problem_setting}
\subsection{The main problem}
Let us now be more precise with the important assumptions of the study. 
We let $D$ be an open, non-empty simply connected bounded set in $\mathbb{R}^{d}$, $d \in \{2,3\}$, of class ${C}^{1,1}$. 
We fix a real number $\delta > 0$ and define $\mathcal{A}$ as the collection of all ${C}^{1,1}$ open, non-empty sets $\omega$ strictly contained in $D$ and that are of distance $\delta$ from $\Sigma=\partial{D}$ such that $\Omega=D \setminus \overline{\omega}$ is connected; i.e., we define the following set
\begin{equation}
    \mathcal{A} := \{ \omega \Subset D \mid \text{$\omega \in {C}^{1,1}$ is an open set, $d(x,\partial D) > \delta, \ \forall x \in \omega$, and $D\setminus \overline{\omega}$ is connected}\}.
\end{equation}
Hereinafter, we say $\Omega$ is an \textit{admissible domain} if $\Omega=D \setminus \overline{\omega}$, for some $\omega \in \mathcal{A}$.
\ferj{In this case, for the sake of notation, we write $\Omega \in \mathcal{A}^{\circ}$.}

We tacitly assume here that we can find $\omega^{\ast}$ in $\mathcal{A}$ such that \eqref{eq:overdetermined} has a solution.
In other words, we assume that there is $\omega^{\ast} \in \mathcal{A}$ such that the surface measurement $g$ (or $f$, if $g$ is given instead) is obtained without error.
\fergy{Therefore, we propose the following more precise formulation of the inverse geometry problem:}
\begin{equation}\label{eq:inverse_problem}
	\text{Given $D$, $f$, and $g$, find $\omega \in \mathcal{A}$ such that $u(D\setminus\overline{\omega})$ solves \eqref{eq:overdetermined}.}
\end{equation}
The regularity assumptions imposed on the data $f$ and $g$ are more than we can actually expect.
\fergy{In reality, we can only assume that $\Sigma$ is Lipschitz, with $f \in H^{1/2}(\Sigma)$ and $g \in H^{-1/2}(\Sigma)$, in order to obtain $H^{1}$ regular state solutions. 
Higher regularity of the states can be achieved by imposing additional smoothness on the boundaries and the data (see, e.g., \cite[Thm.~2.4.2.5, p.~124, Sec.~2, p.~84, and p.~128]{Grisvard1985}). 
Such a result can be derived through a local regularity argument similar to the proof in \cite[Thm.~29]{BacaniPeichl2013} (see also \cite{BadraCaubetDambrine2011,Caubet2013} for similar results in the context of fluids).
In this work, we adopt the stated regularity assumptions to streamline many of the proofs.}

\fergy{Before we finish this subsection, we comment that a key theoretical aspect in inverse problems is identifiability, which refers to the uniqueness of the solution given the observed data. 
In the context of the present study, which focuses on the recovery of an unknown inclusion in a domain governed by an advection-diffusion equation, an explicit identifiability result is not yet available in the literature. 
However, such a result is crucial to justify the well-posedness of the inverse formulation. 
While classical methods, such as unique continuation and Carleman estimates, establish identifiability in simpler elliptic settings, their extension to advection-diffusion systems with geometric complexity presents significant challenges. 
Recent advances by Cao et al. \cite{CaoDiaoLiuZou2022} have shown that, under precise geometric conditions, a single far-field measurement is sufficient to uniquely determine both the shape and the boundary impedance of polyhedral obstacles in inverse scattering problems. 
These findings suggest that, with appropriate geometric and analytical assumptions, a similar identifiability result could be derived for the inverse advection-diffusion problem considered here. 
This opens up an important avenue for future theoretical work, potentially providing a solid foundation for the proposed numerical identification framework.}
\subsection{Shape optimization reformulations}\label{subsec:formulations} 
To solve \eqref{eq:inverse_problem}, we reformulate it into shape optimization setting and apply the concept of shape calculus to solve the resulting optimization problem numerically. 
To this end, we consider two such reformulations of \eqref{eq:inverse_problem} by choosing one of the boundary conditions on the unknown boundary to obtain a well-posed state equation, and then track the remaining boundary data in $L^{2}$ sense over $\Sigma$.
More precisely, we consider the following minimization problems
\fergy{
\begin{align}
	\omega^{\ast} \in \operatorname{argmin}_{\omega \in \mathcal{A}}  \JD(D \setminus \overline{\omega}), & \quad \text{where }  \JD(D \setminus \overline{\omega}) = \JD(\Omega) := \frac12 \intS{(\un - f)^2},	\label{eq:Dirichlet_cost_function}\\
	\omega^{\ast} \in \operatorname{argmin}_{\omega \in \mathcal{A}}  \JN(D \setminus \overline{\omega}), & \quad \text{where } \JN(D \setminus \overline{\omega}) := \JN(\Omega) =\frac12 \intS{({\sigma}\dn{\ud}  - g)^2} \label{eq:Neumann_cost_function}
\end{align} 
}
where $\ud := \ud(\Omega)$ and $\un := \un(\Omega)$ are the solutions to \fergy{the PDE systems \eqref{eq:state_ud} and \eqref{eq:state_un}, respectively.}
In \eqref{eq:Dirichlet_cost_function} and \eqref{eq:Neumann_cost_function}, the infimum is always taken over the set of admissible domains $\mathcal{A}$. 
We refer to $\un$ and $\ud$ as state variables or the simply states.

To ensure $\JN$ is well-defined, the state variable $\ud$ needs to be \fergy{at least $H^{2}$ regular. In this case, assuming \eqref{eq:assumptions} holds, and that $g \in H^{1/2}(\Sigma)$ and $\Omega \in C^{1,1}$ are sufficient.
Such claim can be proved by arguing as in the proof of \cite[Thm.~29]{BacaniPeichl2013}.}  
This means that in numerical experiments where the state variables might lack high regularity, using $\JN$ might not be practical.
\begin{remark}\label{eq:equivalence_of_classical_formulations}
    The optimization problems \eqref{eq:Dirichlet_cost_function} and \eqref{eq:Neumann_cost_function} are only equivalent to \eqref{eq:overdetermined} if we have a perfect match of boundary data on the known boundary, namely, $u = f$ and $\partial_n u = g$ on $\Sigma = \partial \omega$. 
    Indeed, if $\omega^{\ast} \in \mathcal{A}$ solves \eqref{eq:inverse_problem}, then $J_{i}(\Omega^{\ast}) = J_{i}(D\setminus{\overline{\omega}}^{\ast})= 0$, for $i\in \{D,N\}$, and it holds that
    \begin{equation}\label{eq:argmin}
    \omega^{\ast} \in \operatorname{argmin}_{\omega \in \mathcal{A}}J_{i}(\Omega), \quad \text{for } i\in \{D,N\}.
    \end{equation} 
    Conversely, if $\omega^{\ast} \in \mathcal{A}$ solves \eqref{eq:argmin} with $J_{i}(\Omega^{\ast}) = 0$, for $i\in \{D,N\}$, then it is a solution of \eqref{eq:inverse_problem}.
\end{remark}
In the rest of the paper, the subscript $i$ of $J_{i}$ is always understood to be either $D$ or $N$.
\subsection{Weak forms of the state systems}\label{subsec:weak_formulations} 
\correction{Throughout the paper, we let $c > 0$ be a generic constant; that is, it may take different values at different places.}\ferj{ 
Let us briefly discuss the respective variational formulation of \eqref{eq:state_un} and \eqref{eq:state_ud}.
To this end, we denote:
\[
	\HGamma:=\{\varphi \in H^{1}(\Omega) \mid \text{$\varphi = 0$ on $\Gamma$}\},
\]
which is equipped with the norm 
\[
	\norm{\varphi}_{\HGamma}^{2} = \abs{\varphi}_{H^1(\Omega)} = \norm{\nabla \varphi}_{L^{2}(\Omega)} = \intO{\abs{\nabla \varphi}^{2}},
\]
and we introduce the following bilinear form:
\begin{equation}\label{eq:bilinear_form_a}
	a({\varphi},{\psi}) = \intO{\sigma\nabla \varphi \cdot \nabla \psi}+\intO{(\bb \cdot \nabla \varphi) \psi},  
	\quad \text{where}\ {\varphi},{\psi}\in \HGamma.
\end{equation}
}
\color{black}

The weak form of \eqref{eq:state_un} reads as follows:
\fergy{
	\begin{equation}\label{eq:state_weak_form_un}
		\text{Find ${\un} \in \HGamma$\ \ such that\ \  $a(\un,\psi) = \intS{g \psi}$, \ for all $\psi \in \HGamma$}.
	\end{equation}
}	
In a Lipschitz domain $\Omega$ and with $g \in H^{-1/2}(\Sigma)$, given the conditions in \eqref{eq:assumptions} regarding the coefficients, the Lax-Milgram lemma establishes the existence of a unique weak solution $\un \in \HGamma$ for \eqref{eq:state_weak_form_un}. 
To ensure the coercivity of $a$ in $\HGamma$, it suffices to assume 
\[
	\lvert \bb \rvert_{\infty} < \correction{c}\sigma_{0}, \qquad \text{where $\lvert \bb \rvert_{\infty} = \sup \left\{  \abs{b_{i}} \mid 1 \leqslant i \leqslant d  \right\}$}, 
\]
\correction{for some constant $c > 0$.\footnote{The constant $c>0$ is, in fact, the Poincar\'{e} constant, which enables us to control the full $H^1$ norm by the $H^1$ seminorm. It appears that this point was inadvertently overlooked in \cite{CherratAfraitesRabago2025}.}}
This condition guarantees $a(\varphi,\varphi) \geqslant \left( \sigma_{0} - \correction{c}\lvert \bb \rvert_{\infty} \right) \lVert \varphi \rVert_{\HGamma}^{2} = c \lVert \varphi \rVert_{\HGamma}^{2}$, for some real number $c > 0$ \cite{CherratAfraitesRabago2025}.
%

Similarly, we may write the weak form of \eqref{eq:state_ud} as follows:
\ferj{
	\begin{equation}\label{eq:state_weak_form_ud}
		\text{Find ${\ud} \in \HGamma$, \ \ $\ud = f$ on $\Sigma$, \ \ such that\ \  $a(\ud,\psi) = 0$, \ for all $\psi \in \HGamma$}.
	\end{equation}
	}

For a Lipschitz domain $\Omega$ and $f \in H^{1/2}(\Sigma)$, assuming the conditions in \eqref{eq:assumptions} on the coefficients, the Lax-Milgram lemma again guarantees a unique weak solution $\un \in \HGamma$ for \eqref{eq:state_weak_form_un}, contingent upon the condition $\lvert \bb \rvert_{\infty} < \correction{c}\sigma_{0}$ \correction{(for some constant $c>0$)} for the well-posedness of \eqref{eq:state_weak_form_ud}.

\section{Shape Derivatives}
\label{sec:shape_derivatives}
\fergy{
To numerically solve \eqref{eq:Dirichlet_cost_function} and \eqref{eq:Neumann_cost_function}, we require the structure of $J_{i}$, $i\in \{D,N\}$, to devise a gradient-based iterative scheme for concrete problem-solving. 
These expressions will be derived in this section using the concept of \textit{shape calculus}, specifically through the notion of the \textit{velocity} or \textit{speed} method; refer to \cite{DelfourZolesio2011, HenrotPierre2018, MuratSimon1976, Simon1980, SokolowskiZolesio1992} for more details.
}
\subsection{Some elements of shape calculus}
In this section, we briefly introduce key concepts from shape calculus, focusing on material and shape derivatives of functions and functionals, and fix some notations.

Let $t \geqslant 0$. We define $T_{t} : D \longmapsto D$ as the map given by
\begin{equation}\label{eq:poi}
	T_{t} = T_{t}[{\VV}] = id + t \VV,
\end{equation}
where $\VV$ is a $t$-independent deformation field belonging to the admissible space
\begin{equation}
\label{eq:space_for_V}
	\sfTheta:=\{\VV = (\theta_{1}, \ldots, \theta_{d})^{\top} \in {{C}}^{1,1}(\overline{D})^{d} \mid \operatorname{supp}{\VV} \subset {\overline{D}}_{\delta}\},
\end{equation}
where $\{x \in {D} \mid d(x,\partial{D}) > \delta/2\} \subset {D}_{\delta} \subset \{x \in {D} \mid d(x,\partial{D}) > \delta/3\}$.
Clearly, $T_{0} = id$, and it can be shown that, \fergy{for sufficiently small $\varepsilon>0$}, $t \in [0,\varepsilon]$, $T_{t}$ is a diffeomorphism of $\mathbb{R}^{d}$.
Throughout the paper, the subscript `$t$' indicates that the associated object is defined on a domain dependent on time $t$. 
For example, $u_{t}$ represents the solution of \eqref{eq:state_ud} with $\Omega$ replaced by $\Omega_{t}=T_{t}(\Omega)[\VV]$.

We set 
\fergy{
\[
	A_{t}  := \dd_{t}(DT_{t}^{-1})(DT_{t})^{-\top},
	\qquad
	B_{t} := \dd_{t} |(DT_{t})^{-\top} \nn|, \qquad \text{and}\qquad  C_{t} := \dd_{t} (DT_{t})^{-\top},
\]
}
and assume that $\varepsilon > 0$ is sufficiently small such that for all $t \in \interval := [0,\varepsilon]$, $\dd_{t} := \det \,DT_{t} > 0$, and we can find pair of constants $\constant_{1}$, $\constant_{2}$  ($0 < \constant_{1} < \constant_{2}$), $\constant_{3}$, $\constant_{4}$ ($0 < \constant_{3} < \constant_{4}$), and $\constant_{5}$, $\constant_{6}$ ($0 < \constant_{5} < \constant_{6}$) such that (cf. \cite[Chap.~10, Sec.~2.4, Eq.~(2.32) and Eq.~(2.33), p.~526]{DelfourZolesio2011})
\begin{equation}\label{eq:bounds}
	\constant_{1} |\xi|^2 \leqslant A_{t} \xi \cdot \xi \leqslant \constant_{2} |\xi|^2,\quad
	\constant_{3} \leqslant \dd_{t} \leqslant \constant_{4},
	\quad \text{and}\quad	
	\constant_{5} \leqslant \sup_{t\in \interval} \abs{C_{t}}_{\infty} \leqslant \constant_{6}, 
	\quad \text{for all $\xi \in \mathbb{R}^{d}$.}
\end{equation}
\fergy{Note that we can choose $m_1 = \min\{\constant_1, \constant_3, \constant_5\}$ and $m_2 = \max\{\constant_2, \constant_4, \constant_6\}$ as the respective lower and upper bounds, ensuring that all the inequality conditions in \eqref{eq:bounds} are satisfied.}

\ferj{For $t \in \interval$ and $\VV \in \sfTheta$}, we see that the following regularities hold:\footnote{Here, $\varepsilon > 0$ is made smaller if necessary.}
\ferj{
\begin{equation}\label{eq:regular_maps}
\left\{\ 
\begin{aligned}
	[t \mapsto \dd_{t}] &\in {{C}}^1(\interval,{{C}^{0,1}}(\overline{D})),\qquad 
	&[t \mapsto A_{t} ] \in {{C}}^{1}(\interval,{{C}^{0,1}}(\overline{D})^{d \times d}),\\	
	[t \mapsto B_{t}] &\in {{C}}(\interval,{{C}}(\partial\Omega)),\qquad
	&[t \mapsto C_{t}] \in {{C}}^1(\interval,{{C}^{0,1}}(\overline{D})^{d \times d}).
\end{aligned}
\right.
\end{equation} 
}
Lastly, mention that we have the following derivatives for the maps given previously:
\begin{equation}\label{eq:derivative_of_regular_maps}
\left\{
\begin{aligned}
	\frac{d}{dt}\dd_{t} \big|_{t=0} 
		&= \lim_{t \searrow 0} \frac{\dd_{t} - 1}{t} = \div{\VV} =: \dd,\\
	\quad \frac{d}{dt}A_{t} \big|_{t=0} 
		& = \lim_{t \searrow 0} \frac{A_{t} - \vect{I}}{t} 
		= \dd\vect{I} -  D\VV - (D\VV)^\top =: A,\\
	\quad\frac{d}{dt}B_{t} \big|_{t=0} 
		& =  \lim_{t \searrow 0} \frac{B_{t} - 1}{t} 
		= {\operatorname{div}}_{\tau} \VV
		= \dd \big|_{\Gamma} - (D\VV\nn)\cdot\nn,\\
	\quad\frac{d}{dt}C_{t} \big|_{t=0} 
		& =  \lim_{t \searrow 0} \frac{C_{t} - 1}{t} 
		= \dd\vect{I} - (D\VV)^\top =: C,
\end{aligned}
\right.
\end{equation}
where ${\operatorname{div}}_{\tau} \VV$ denotes the tangential divergence of the vector $\VV$.
The proofs of the above results \eqref{eq:regular_maps} and \eqref{eq:derivative_of_regular_maps} are provided in \cite[Chap.~2.15, pp.~75--76, Chap.~2.18--2.19, pp.~79--85]{SokolowskiZolesio1992}.

Without further notice, the pseudo-time parameter $t$ is always assume sufficient small so that the required regularities for the (perturbed) domain is preserved and the regularity of the mappings in \eqref{eq:regular_maps} hold true.

\sloppy We say that the function $u(\Omega)$ has a \textit{material} derivative $\dot{u} = \dot{u}(\Omega)[\VV] $ and a \textit{shape} derivative $u^{\prime}= u^{\prime}(\Omega)[\VV] $ at $0$ in the direction of the vector field $\VV$ if the limits 
\[
	\dot{u} = \lim_{t \searrow 0} \frac{u^{t}(\Omega) - u(\Omega) }{t} 
	\qquad\text{and}\qquad
	u^{\prime} = \lim_{t \searrow 0} \frac{u(\Omega_{t}) - u(\Omega)}{t}
\]
exist, respectively, where $u^{t}(x) := (u(\Omega_{t}) \circ T_{t})(x) = u(\Omega_{t})(T_{t}(x))$.
Notice here that $u^{t}$ is defined on the fixed domain $\Omega$.
For sufficiently smooth $\Omega$, $u$, and $\VV$, these derivatives are related by $u^{\prime} = \dot{u} - (\nabla{u} \cdot \VV)$ \cite{DelfourZolesio2011,SokolowskiZolesio1992}. 
\ferj{Meanwhile, we say that a shape functional $j : \mathcal{A}^{\circ} \to \mathbb{R}$ has a directional Eulerian derivative at $\Omega \in \mathcal{A}^{\circ}$ in the direction of $\VV \in \sfTheta$} if the limit
\[
	\lim_{t \searrow0} \frac{j(\Omega_{t}) - j(\Omega)}{t} =: d j(\Omega)[\VV]
\]
exists \cite[Eq.~(3.6), p.~172]{DelfourZolesio2011}. 
If the map $\VV \mapsto d j(\Omega)[\VV]$ is linear and continuous \ferj{for all $\VV \in \sfTheta$}, then $j$ is \textit{shape differentiable} at $\Omega$, and the map is referred to as the \textit{shape gradient} of $j$.\\

%
%
%
\subsection{Material and shape derivatives of the state variables} 
The main objective of this section is to present a variational equation satisfied by the material derivatives of the states $\un$ and $\ud$. 
Subsequently, we deduce the shape derivative of the state corresponding to $\un$ and provide a demonstration solely for the case of $\un,$ applying the same approach for the case of $\ud$.
Throughout the rest of the paper, $\VV \in \sfTheta$, unless otherwise specified.
%
%
%
%
%
%
\begin{theorem}\label{theo:material_un}
The Neumann solution $\un \in \HGamma$ of \eqref{eq:state_un} has a derivative $\dotun \in \HGamma$ that satisfies
\begin{equation}\label{eq:material_un}
	a(\dotun,  \psi) = l(\un;\psi), \quad \forall\psi \in \HGamma,
\end{equation}
where
\begin{equation}\label{eq:right_el}
\begin{aligned}
	l(\un;\psi)
		& = -\intO{\left( \sigma A\nabla \un \cdot  \nabla \psi + C^{\top}\bb \cdot \nabla \un \psi \right)}\\
		&\qquad - \intO{ \left[  (\nabla \sigma \cdot \VV)\nabla \un  \cdot \nabla \psi + D\bb  \VV \cdot\nabla \un \psi \right]}\\
		& =: l_{0}(\un;\psi) + l_{1}(\un;\psi).
\end{aligned}
\end{equation}
\end{theorem}
We comment here that, because of the regularity assumptions $\sigma \in W^{1,\infty}(D)$, $\bb \in W^{1,\infty}(D)^{d}$, $\VV \in \sfTheta$, and $\un \in \HGamma$, the bounds for $A$ and $C$ given in \eqref{eq:bounds}, and under the condition that $\abs{\bb}_{\infty} < \sigma_{0}$, the existence of unique weak solution $\dotun \in \HGamma$ of \eqref{eq:material_un} is a consequence of the Lax-Milgram theorem.
We omit the proof detail for economy of space.

\color{black}

Meanwhile, to support our assertion in Theorem~\ref{theo:material_un}, \fergy{ we have formulated several lemmas whose proofs are outlined in Appendix} \ref{appx:differentiability_of_the_state} along with their respective proofs.

%
%
%
The next theorem presents the shape derivative of the state using the expression for the material derivative described in the previous Theorem~\ref{theo:material_un}.
\color{black}
\begin{theorem}\label{theo:form_un}
Let $\Omega \in {C}^{2,1}$ be an admissible domain, $\VV \in \sfTheta \cap {{C}}^{2,1}(\overline{D})^{d}$, and the state $\un \in H^3(\Omega) \cap \HGamma$ be sufficiently smooth.  
Then, $\un$ is shape differentiable, and its shape derivative satisfies the system:
\begin{equation}\label{eq:unp}
\left\{
\begin{array}{rcll}
    -\div{\sigma\nabla \unp}+\bb \cdot \nabla \unp & = & 0 & \text{in }\;\; \Omega,\\
    \unp & = & -\dn{\un}\Vn & \text{on }\;\Gamma,\\
    {\sigma}\dn{\unp} & = & 0 & \text{on }\;\Sigma.
\end{array}
\right.
\end{equation}
\end{theorem}

To provide evidence for this Theorem~\ref{theo:form_un}, we will need to prove the result of Lemma~\ref{lem:curl_identity} below using the form of the curl operator $\curl$ on $\mathbb{R}^{d}$. 
We have the curl of the cross product identity
\color{black}
\begin{equation}\label{eq:prod_vect}
\begin{aligned}
    \curl(\vect{\varphi} \times \vect{\psi})
    	=\vect{\varphi}\div {\vect{\psi}}
   	 - \vect{\psi}\div{ \vect{\varphi}} 
	+ (\vect{\psi} \cdot \nabla)\vect{\varphi}
    	 - (\vect{\varphi} \cdot \nabla)\vect{\psi},
\end{aligned}
\end{equation}
for any differentiable $\mathbb{R}^{d}$-valued functions $\vect{\varphi}$ and $\vect{\psi}$.

For this purpose, we embed $\varphi$ and $\psi$ into $\mathbb{R}^{d}$ by appending zero as the third coordinate when $d=2$. However, when $d=3$, no adjustment is needed
\color{black}
and the curl of $\varphi$ in $\mathbb{R}^3$ is given by:
\begin{equation}
\nabla \times \mathbf{\varphi} = 
\left(
\frac{\partial \varphi_3}{\partial x_2} - \frac{\partial \varphi_2}{\partial x_3}, 
\frac{\partial \varphi_1}{\partial x_3} - \frac{\partial \varphi_3}{\partial x_1}, 
\frac{\partial \varphi_2}{\partial x_1} - \frac{\partial \varphi_1}{\partial x_2}
\right).
\end{equation}
\color{black}

\begin{lemma}\label{lem:curl_identity}
\fergy{
For $(u,v) \in \left[ H^{2}(\Omega) \cap \HGamma\right]^2  $ and $\VV \in \sfTheta$, we have}
\begin{equation}
    \begin{aligned}\label{eq:lemma_curl_{1}}
    	\intO{\curl (\sigma \nabla{u} \times \VV)\cdot \nabla{v}}=0,
    \end{aligned}
\end{equation}
and it holds that
\begin{equation}\label{eq:lemma_curl_{2}}
    \begin{aligned}
         &-\intO{\dd \sigma \nabla{u} \cdot \nabla{v}}
         +\intO{\sigma D\VV \nabla{u} \cdot \nabla{v} }\\
         &\qquad =
         -\intO{\div{\sigma \nabla{u}}\VV \cdot \nabla{v}}
         +\intO{\sigma \nabla^{2}{u} \VV \cdot \nabla{v}}
         +\intO{(\nabla\sigma \cdot\VV)(\nabla{u}\cdot \nabla{v})},
     \end{aligned}
\end{equation} 
where $\nabla^{2}{u}$ denotes the Hessian (matrix) of $u$.
\end{lemma}
\begin{proof}
Since $\div {\curl \varphi}=0 $ for all $\varphi \in \HGamma $, then by using integration by parts, we obtain
\[
\begin{aligned}
\intO{\curl(\sigma \nabla{u} \times \VV)\cdot \nabla{v}} &=-\intO{\div{\curl(\sigma \nabla{u} \times \VV)}v}
+\intdO{ \left[ \curl (\sigma \nabla{u} \times \VV) \right] \cdot {v}{\nn}} \\
&=\intG{ \left[  \curl (\sigma \nabla{u} \times \VV) \right] \cdot {v}{\nn}} 
+\intS{ \left[  \curl (\sigma \nabla{u} \times \VV) \right] \cdot {v}{\nn}} \\
&=0,
\end{aligned}
\]
%
%
%
%
\fergy{ because $v=0$ on $\Sigma$ and $\VV=0$ on 
$\Gamma$, applying \eqref{eq:prod_vect}, we obtain the following result:
\[
\curl(\sigma \nabla{u} \times \VV)\cdot \nabla{v}
	= \dd \sigma \nabla{u} \cdot \nabla{v} 
	- \div{\sigma \nabla{u}}\VV \cdot \nabla{v}
	+ \nabla (\sigma \nabla{u}) \VV \cdot \nabla{v}
 	- \sigma D\VV \nabla{u} \cdot \nabla{v}.
\]}
By integrating over $\Omega$ and knowing that 
 $\displaystyle \intO{\curl(\sigma \nabla{u} \times \VV)\cdot \nabla{v}}=0 $, we obtain
\[
 \intO{ \left( \dd \sigma \nabla{u} \cdot \nabla{v}
	- \div{\sigma \nabla{u}} \VV \cdot \nabla{v}
	+ \nabla (\sigma \nabla{u}) \VV \cdot \nabla{v}
	- \sigma D\VV \nabla{u} \cdot \nabla{v} \right )} = 0.
\]
Using the formula (see Appendix \ref{appx:proofs_of_key_identities} for a justification)
\begin{equation}\label{eq:grad_sig_gradu}
    \nabla (\sigma \nabla{u}) \VV \cdot \nabla{v} 
    = \sigma \nabla^{2} u \VV \cdot \nabla{v} + (\nabla \sigma \cdot \VV)(\nabla{u}\cdot \nabla{v}), 
\end{equation}
we deduce \eqref{eq:lemma_curl_{2}}.
\end{proof}
We now provide the proof of Theorem~\ref{theo:form_un}.
\begin{proof}[Proof of Theorem~\ref{theo:form_un}]
We drop the subscript ${\cdot}_{N}$ for convenience.
Using \eqref{eq:bilinear_form_a} and taking $\varphi=\dot{u}-\VV \cdot\nabla{u}=u^{\prime}$ and $\psi=v$, we obtain
\[
\begin{aligned}
a(u^{\prime},v) =a(\dot{u}-\VV \cdot\nabla{u},v)
&= a(\dot{u},v)
-\intO{\sigma \nabla(\VV \cdot\nabla{u})\cdot \nabla{v}}
-\intO{\bb \cdot \nabla(\VV \cdot \nabla{u})v}\\
&=:\mathbb{I}_{1}+\mathbb{I}_{2}+\mathbb{I}_{3}.
\end{aligned}
\]
We treat each term $\mathbb{I}_{i}$, $i=1,2,3$, separately.
For the first term $\mathbb{I}_{1}$, using \eqref{eq:material_un}, we get
\[
\mathbb{I}_{1} = l({{u}};v),
\]
\fergy{where $l({{u}};v) = l_{0}({{u}};v) + l_{1}({{u}};v)$ is given by \eqref{eq:right_el} with  ${{u}} = \un$ and $\psi =v$} .
Using the expressions for $A$ and $C$ given in \eqref{eq:derivative_of_regular_maps} and $l_{0}$ in \eqref{eq:right_el}, we can write 
\[
\begin{aligned}
l_{0}({{u}};v)
	& = - \intO{\dd \sigma \nabla{u} \cdot \nabla{v}}
 	+ \intO{\sigma D\VV \nabla{u} \cdot \nabla{v} }\\
 	& \qquad 
	+ \intO{\sigma D \VV^{\top}\nabla{u} \cdot \nabla{v}}
	- \intO{ \dd ( \bb \cdot \nabla {u}){v}} 
    	+ \intO{ D \VV  \bb \cdot \nabla {u}{v} }.
\end{aligned}
\]
In view of \eqref{eq:lemma_curl_{2}}, we deduce that
\[
\begin{aligned}
l_{0}({{u}};v)
& = -\intO{\div{\sigma \nabla{u}}\VV \cdot \nabla{v}} 
	+ \intO{\sigma \nabla^{2} u \VV \cdot \nabla{v}}
 	+ \intO{(\nabla\sigma \cdot\VV)(\nabla{u}\cdot \nabla{v})} \\
 	& \qquad 
	+ \intO{\sigma D \VV^{\top}\nabla{u} \cdot \nabla{v}}
	- \intO{ \dd ( \bb \cdot \nabla {u}){v}} 
    	+ \intO{ D \VV  \bb \cdot \nabla {u}{v} }.
\end{aligned}
\]
\color{black}
Moreover, recalling that 
\[
l_{1}(u;v)=- \intO{ \left[  (\nabla \sigma \cdot \VV)\nabla \un  \cdot \nabla v + D\bb  \VV \cdot\nabla \un v \right]},
\]
and replacing both $l_1(u;v)$ and $l_0(u;v)$ in $\mathbb{I}_1$, we obtain:
\color{black}
\[
\begin{aligned}
\mathbb{I}_{1}&=
 - \intO{\div{\sigma \nabla{u}}\VV \cdot \nabla{v}}
 + \intO{\sigma \nabla^{2} u \VV \cdot \nabla{v}}
 + \intO{\sigma D \VV^{\top}\nabla{u} \cdot \nabla{v}}\\
 &\qquad - \intO{ \dd ( \bb \cdot \nabla {u}){v}} 
    + \intO{ D \VV  \bb \cdot \nabla {u}{v} }
    - \intO{ \left( D\bb  \VV \cdot\nabla{u} \right) {v} }.
\end{aligned}
\]
\color{black}
%
%
For the second term $\mathbb{I}_{2}$, we have
\[
\begin{aligned}
\mathbb{I}_{2}
	= -\intO{\sigma \nabla(\VV \cdot\nabla{u})\cdot \nabla{v}}
	= -\intO{\sigma D \VV^{\top}\nabla{u} \cdot \nabla{v}}
		-\intO{\sigma \nabla^{2} u \VV \cdot \nabla{v}},
\end{aligned}
\]
while for the third term $\mathbb{I}_{3}$, we get 
\[
\begin{aligned}
\mathbb{I}_{3}=-\intO{\bb \cdot \nabla(\VV \cdot \nabla{u})v}
&=-\intO{   \bb \cdot D \VV^{\top}\nabla {u}{v} }
-\intO{\bb\cdot \nabla^{2} u \VV v}\\
&=-\intO{ D \VV  \bb \cdot \nabla {u}{v} }
-\intO{\bb\cdot \nabla^{2} u \VV v}.
\end{aligned}
\]
Adding the computed expressions for $\mathbb{I}_{1},\mathbb{I}_{2}$, and $\mathbb{I}_{3}$, we obtain
\[
\begin{aligned}
a(u^{\prime},v)
	& = -\intO{\div{\sigma \nabla{u}}\VV \cdot \nabla{v}} 
	- \intO{ \left[ \left( D\bb  \VV \cdot\nabla{u} \right) {v}  + \bb\cdot \nabla^{2} u \VV v \right]} 
	- \intO{ \dd ( \bb \cdot \nabla {u}){v} }\\
	&= - \left[ \intO{\VV \cdot (\bb\cdot \nabla{u})\nabla{v}} + \intO{\VV \cdot \nabla( \bb \cdot \nabla{u})v} \right]
	- \intO{ \dd ( \bb \cdot \nabla {u}){v} }\\
	&= - \intO{ \left\{ \VV \cdot \nabla \left[ (\bb \cdot \nabla{u})v \right] + \dd ( \bb \cdot \nabla {u}){v} \right\} }\\
	&= - \intO{\div{\VV(\bb\cdot \nabla{u})v}}, \qquad (\delta = \div{\VV}),
\end{aligned}
\]
where the second equation line follows \fergy{from \eqref{eq:state_un} and the identity
\[
	\VV \cdot \nabla( \bb \cdot \nabla{u})v=\left( D\bb  \VV \cdot\nabla{u} \right) {v} +\bb\cdot \nabla^{2} u \VV{v},
\]}
which hold in $\Omega$.
Meanwhile, the last line is a consequence of the fact that for a scalar function $\varphi$ and vector field $\vect{F}$, it holds that
\begin{equation}\label{eq:divergence_identity}
	\intO{ \left( \varphi \div{\vect{F}} + \vect{F} \cdot \nabla \varphi \right) } = \intO{\div{\varphi\vect{F}}}.
\end{equation}
Now, employing the divergence formula, noting that $v \in \HGamma$ and  $\VV = \vect{0}$ on $\Sigma$, leads to
\[
	a(u^{\prime},v)=-\intO{\div{\VV(\bb\cdot \nabla{u})v}} = - \intS{(\bb\cdot \nabla{u})v{\VV}\cdot{\nn}} - \intG{(\bb\cdot \nabla{u})v{\VV}\cdot{\nn}} = 0.
\] 
\color{black}
Using Green’s first identity, we rewrite the bilinear form $a(u^{\prime}, v)$ as follows:
\[
-\intO{\div{\sigma \nabla{u^{\prime}}}v}
+\intO{\bb\cdot \nabla{u^{\prime}} v}
-\intS{\sigma{v}\dn{u}^{\prime}}=0.
\]
\color{black}
We deduce from this equation that $u^{\prime}$ satisfies the equation  $-\div{\sigma \nabla{u^{\prime}}} +\bb\cdot \nabla{u^{\prime}} =0$ in $\Omega$, with the boundary condition 
${\sigma}\dn{u}^{\prime}=0$ on $\Sigma$.
\fergy{Lastly, since $u, \dot{u} \in \HGamma$, i.e., $u = \dot{u} = 0$ on $\Gamma$, we obtain $u^{\prime} = \dot{u} - \nabla{u} \cdot \VV = -\dn{u} \Vn$ on $\Gamma$, completing the proof of the theorem.}
\end{proof}
\begin{remark}
Using similar line of proof, it can be shown that when the state $\ud$ is sufficiently smooth, it is shape differentiable and its the shape derivative $\udp$ satisfies
\begin{equation}\label{eq:udp}
\left\{
\begin{array}{rcll}
    -\div{\sigma\nabla \udp}+\bb \cdot \nabla \udp & = & 0 & \text{in }\;\; \Omega,\\
    \udp & = & -\dn{\ud}\Vn & \text{on }\;\Gamma,\\
    {\sigma}\dn{\udp} & = & 0 & \text{on }\;\Sigma.
\end{array}
\right.
\end{equation}
\end{remark}
\begin{remark}\label{rem:higher_regularity_for_shape_derivative_of_the_state}
The regularity assumptions $\Omega \in {C}^{1,1}$, $f \in H^{3/2}(\Sigma)$ ($f \not\equiv 0$), and $g \in H^{1/2}(\Sigma)$ only allow us to obtain an $H^{2}(\Omega) \cap \HGamma$ regularity for the states.
However, this regularity of the states is not sufficient to justify the existence of their shape derivatives satisfying \eqref{eq:unp} and \eqref{eq:udp}. 
We require higher regularity of the solutions. 
Therefore, we need $\VV \in \sfTheta \cap {C}^{2,1}(\mathbb{R}^{d})$ bounded domains, $f \in H^{5/2}(\Sigma)$ ($f \not\equiv 0$), and $g \in H^{3/2}(\Sigma)$ which allows us to have $H^{3}(\Omega) \cap \HGamma$ regularity for the states.
\end{remark} 
\subsection{Shape derivatives of the shape functionals} 
We will now calculate the shape derivative of the proposed cost functions using two techniques. 
The first is based only on the variational formulation of the equation verified by the material derivative of the state variables, while the second technique uses the shape derivative of the state. 
In both cases, we introduce an adjoint state appropriate to our problem.

Let us now characterize the shape gradient of the shape functions $J_{i}$, $i\in\{D,N\}$.
\begin{proposition}[Shape gradient of $J$]
	\label{prop:shape_gradients}
	Let $\Omega$ be an admissible domain and $\VV \in \sfTheta$.
	The map $t \mapsto J_{i}(\Omega_{t})$, $i\in\{D,N\}$, is ${{C}}^1$ in a neighborhood of $0$.
	Its shape derivative at $0$ is given by $dJ_{i}(\Omega)[\VV] =\displaystyle \intG{G_{i}\nn \cdot \VV}$, where the shape gradient $G_{i}$, \ferj{i.e., the kernel of $dJ_{i}$,} $i\in\{D,N\}$, are respectively given by
	\begin{align}
		G_{{D}} &= F(\un,\pn),\label{eq:shape_gradient_gd}\\
		G_{N} &=  -F(\ud,\pd),\label{eq:shape_gradient_gn}
	\end{align}
where $F(\varphi,\psi) = {\sigma} \dn{\varphi}\dn{\psi}$, for $\varphi, \psi \in H^{2}(\Omega) \cap \HGamma$, and the adjoint variables $\pn, \pd \in \HGamma$ respectively satisfy the following \ferj{adjoint} problems:
	\begin{equation}\label{eq:adjoint_system_pn}
\left\{
\begin{array}{rcll}
        \div{\sigma\nabla \pn} + \bb \cdot \nabla \pn + \pn\operatorname{div}{\bb} & = & 0 & \text{ in $\Omega$},\\
        \pn & = & 0 & \text{on $\Gamma$},\\
        \sigma\dn{\pn} + \pn \bb \cdot \nn & = & \un-f & \text{on $\Sigma$};\\
\end{array}
\right.
\end{equation}
	\begin{equation}\label{eq:adjoint_system_pd}
\left\{
\begin{array}{rcll}
        \div{\sigma\nabla \pd}+\bb \cdot \nabla \pd + \pd\operatorname{div}{\bb} & = & 0 & \text{ in $\Omega$},\\
        \pd & = & 0 & \text{on $\Gamma$},\\
        \pd & = &  \dn{\ud}-g & \text{on $\Sigma$}.
\end{array}
\right.
\end{equation}	 
\end{proposition}
%
%

Before we prove the above proposition, let us shortly discuss the variational formulation of the adjoint systems \eqref{eq:adjoint_system_pn} and \eqref{eq:adjoint_system_pd} in the following remark.

\begin{remark}\label{rem:weak_forms_of_adjoints}
For ${\varphi},{\psi}\in \HGamma$, we let $a_{p}$ be a bilinear form on $\HGamma \times \HGamma$ defined as follows:
\ferj{
\begin{align*}
	a_{p}({\varphi},{\psi}) 
	&= \intO{\sigma\nabla \varphi \cdot \nabla \psi}
		- \intO{(\bb \cdot \nabla \varphi) \psi}
		- \intO{\varphi \, \operatorname{div}{\bb} \psi }
		+ \intS{ \varphi (\bb \cdot \nn) \psi }\\
	&= \intO{\sigma\nabla \varphi \cdot \nabla \psi} 
		+ \intO{(\bb \cdot \nabla \psi) \varphi}
	= a({\psi}, \varphi),
\end{align*}
}
where the bilinear form $a$ is given by \eqref{eq:bilinear_form_a}. 
The equivalence is a consequence \fergy{of the following identity which holds for $\bb \in W^{1,\infty}(D)^{d}$ and $\varphi, \psi \in H^{1}(\Omega)$:
        \begin{equation}\label{eq:important_identity}
        \begin{aligned}
        	\intdO{ \varphi \psi \bb \cdot \nn }
        	= \intO{ \div{\varphi \psi  \bb}}
        	&= \intO{ \left[ \varphi \psi  \operatorname{div}{\bb} + \bb \cdot \nabla (\varphi \psi  )\right]} \\
        	&= \intO{ \left[ \varphi \psi \operatorname{div}{\bb}  + (\bb \cdot \nabla \varphi) \psi  + (\bb \cdot \nabla \psi ) \varphi \right] }.
        \end{aligned}
        \end{equation}
}
The variational formulation of \eqref{eq:adjoint_system_pn} can be expressed as follows:
	\begin{equation}\label{eq:adjoint_weak_form_un}
		\text{Find ${\pn} \in \HGamma$ such that $a_{p}(\pn,\psi) = \intS{(\un-f)\psi}$, for all $\psi \in \HGamma$}.
	\end{equation}
Meanwhile, the variational formulation of \eqref{eq:adjoint_system_pd} can be stated as follows:
	\begin{equation}\label{eq:adjoint_weak_form_ud}
		\text{Find ${\pd} \in \HGamma$, $\pd = \dn{\ud} - g$ on $\Sigma$, such that $a_{p}(\pd,\psi) = 0$, for all $\psi \in \HSigma$,}
	\end{equation}
	\ferj{where $\HSigma:=\{\varphi \in H^{1}(\Omega) \mid \ \text{$\varphi = 0$ on $\Sigma$}\}$.}
\end{remark}	 
The conditions of $\Omega$ being ${C}^{1,1}$ and $f \in H^{3/2}(\Sigma)$ are sufficient to ensure that problem \eqref{eq:adjoint_system_pn} has a unique weak solution within $\HGamma$. 
Similarly, when these conditions are met, $\dn{\ud} - g \in H^{1/2}(\Sigma)$, thereby guaranteeing the existence of a weak solution for \eqref{eq:adjoint_system_pd} in $\HGamma$. 
However, it is important to mention that their validity depends on specific constraints imposed on $\bb$ and $\sigma$. 
In this case, the existence of unique weak solutions for \eqref{eq:adjoint_weak_form_un} and \eqref{eq:adjoint_weak_form_ud} can be inferred from the Lax-Milgram lemma.
It could be shown that $\pn \in H^{2}(\Omega) \cap \HGamma$ given that $\Omega \in {{C}}^{1,1}$ and $(f,g) \in H^{3/2}(\Sigma) \times H^{1/2}(\Sigma)$. 
However, $\pd \not\in H^{2}(\Omega) \cap \HGamma$, \fergy{unless} $\Omega \in {{C}}^{2,1}$ and $(f,g) \in H^{5/2}(\Sigma) \times H^{3/2}(\Sigma)$. 

%
%
%
%
%
%
%
%
%
%
%
%
%
\subsubsection{Proof of Proposition~\ref{prop:shape_gradients} via material derivative}
	Let $\Omega$ be an admissible domain and $\VV \in \sfTheta$.
	By these regularity assumptions, it can easily verified that the map $t \mapsto J_{i}(\Omega_{t})$, $i\in\{D,N\}$, is differentiable around a neighborhood of $0$.
	In fact, this follows from the fact that $[t \mapsto \dd_{t}] \in {{C}}^1(\interval,{{C}}(\overline{\Omega}))$ and $[t \mapsto u_{Dt}] , [t \mapsto u_{Nt}]  \in {{C}}^1(\interval,H^{1}(\Omega) \cap \HGamma)$. 
	So, $J_{i}$ is shape differentiable for $i\in\{D,N\}$.
	On the one hand, this implies that we can apply Hadamard's boundary differentiation formula (cf. \cite[Thm.~4.3, p.~486]{DelfourZolesio2011} or \cite{HenrotPierre2018,SokolowskiZolesio1992}) to $\JD(\Omega_{t})$ and $\JN(\Omega_{t})$ and obtain
	\[
		d\JD(\Omega)[\VV] = \intS{(\un - f) \dotun}
		\qquad\text{and}\qquad
		d\JN(\Omega)[\VV] = \intS{({\sigma}\dn{\ud} - g) \dn{\dotud}},
	\]
	respectively.
	Here, $\dotun := \dotun(\Omega)[\VV]$ and $\dotud :=\dotud(\Omega)[\VV]$ denote the material derivatives of the states.

To demonstrate the desired results, we will utilize the variational problem associated with \eqref{eq:adjoint_system_pn} and \eqref{eq:adjoint_system_pd}, as well as \eqref{eq:state_un} and \eqref{eq:state_ud}, employing an appropriate selection of test functions.
Initially, this approach will be applied to the Neumann case; subsequently, we will deduce the solution for the Dirichlet problem. 
We mention that we will omit $\cdot_{N}$ in $\un$ and $\pn$ and simply refer to these variables as $u$ and $p$, respectively, for easier notation.

\color{black}
To start, let us put $\psi=p \in \HGamma$ into the variational formulation of $\dot{\un}$ \eqref{eq:material_un}, giving us:
\color{black}
\begin{equation*}
\begin{aligned} 
		\ferj{\intO{\sigma\nabla \dot{u} \cdot \nabla{p}} 
		+ \intO{(\bb \cdot \nabla \dot{u}){p} } }
		&= a(\dot{u}, p) 
		= l(u; p),
\end{aligned}
\end{equation*}
where $l(u; p)$ is given by \eqref{eq:right_el}.
Meanwhile, if we set $\psi = \dot{u} \in \HGamma$ in \eqref{eq:adjoint_weak_form_un}, we get
\begin{equation}\label{eq:weak_form_p:dotu} 
	\ferj{\intO{\sigma\nabla{p} \cdot \nabla \dot{u}} 
		+ \intO{(\bb \cdot \nabla \dot{u}){p} }}
	= a_{p}({p},{\dot{u}}) 
	= \intS{(\un-f)\dot{u}}.
\end{equation}
We deduce that, 
\[
\begin{aligned}
	d\JD(\Omega)[\VV]
	&= {{J}}_{1} + {{J}}_{2},
\end{aligned}
\]
where
\[ 
{{J}}_{1}  
	= -\intO {(\sigma A + \nabla \sigma \cdot \VV)\nabla {\un} \cdot  \nabla {\pn}}
\quad \text{and}\quad
{{J}}_{2}  
	=-\intO{ \left[ (C^{\top}\bb + D\bb \VV) \cdot \nabla \un \right] {\pn} }.
\]

To further simplify the sum of ${{J}}_{1}$ and ${{J}}_{2}$, it will be beneficial to establish certain formulas utilizing the curl operator $\curl$ in $\mathbb{R}^{d}$, as defined in \eqref{eq:prod_vect}. 
We will start by presenting alternative expressions for ${{J}}_{1}$ and ${{J}}_{2}$ in the following lemma.
By combining these two newly derived expressions, we can arrive at an expression for $d\JD(\Omega)[\VV]$ on the unknown boundary $\Gamma$.
\begin{lemma}\label{eq:expressions_j}
For $\bb \in W^{1,\infty}(D)^{d}$, $\VV \in \sfTheta$, and $(u,p) \in [\HGamma \cap H^{2}(\Omega)]^{2}$, we have
\begin{itemize}
	\item $\displaystyle \intO{\curl(\sigma \nabla{u} \times \VV)\cdot \nabla{p}}=0~\text{and}
~\intO{\curl(\sigma \nabla{p} \times \VV)\cdot \nabla{u}}=0$;
	\item $\displaystyle \intO{\nabla(\sigma\nabla{p})\VV \cdot \nabla{u}}
			+ \intO{\sigma \nabla^{2} u \VV \cdot \nabla{p}}
			+ \intO{ {\dd}(\sigma\nabla{p} \cdot \nabla{u})}
			= \intG{ \sigma \dn{p}\dn{u} \Vn }$;
	\item $\displaystyle {{J}}_{1}=-\intO{\div{\sigma \nabla{u}}\VV \cdot \nabla{p}}
	- \intO{\div{\sigma \nabla{p}}\VV \cdot \nabla{u}}
	+ \intG{ \sigma \dn{p}\dn{u} \Vn }.
	$
	\item $\displaystyle {{J}}_{2}
	=-\intO{\curl(\bb\times \VV)\cdot\nabla{u}{p} }
	-\intO{\operatorname{div}{\bb}(\VV \cdot \nabla{u})p}$.
\end{itemize}
\end{lemma} 
%
%
%
\begin{proof} 
The first and second formulas are consequence of Lemma~\ref{lem:curl_identity} (cf. \eqref{eq:grad_sig_gradu} and see Appendix~\ref{appx:proofs_of_key_identities}).
We will use these formulas to simplify ${{J}}_{1}$. 
Knowing that $A=\dd\vect{I} -  D\VV - (D\VV)^\top$, we first expand ${{J}}_{1}$ as follows:
\[
\begin{aligned} 
 {{J}}_{1} 
	&=- \intO{ {\dd} ( \sigma\nabla {u} \cdot \nabla {p}) } 
		+ \intO{  D \VV (\sigma \nabla {u} \cdot \nabla {p}) }
		+ \intO{  D \VV^{\top} (\sigma \nabla {u} \cdot \nabla {p})}\\
	&\qquad -\intO{(\nabla \sigma \cdot \VV)(\nabla{u} \cdot \nabla{p})}.
\end{aligned}
\]
Now, using \eqref{eq:prod_vect} with $\vect{\varphi}=\sigma \nabla{u}$ and $\vect{\psi} =\mathbf{V}$ for the first identity and $\vect{\varphi}=\sigma \nabla{p}$ and $\vect{\psi} =\mathbf{V}$ for the second one, we obtain
\[
\begin{cases}
\curl (\sigma \nabla{u} \times \VV)
={\dd} \sigma \nabla{u} - \div{\sigma \nabla{u}}\VV
+\nabla (\sigma \nabla{u})\VV -D\VV(\sigma \nabla{u}),\\
\curl (\sigma \nabla{p} \times \VV)
={\dd} \sigma \nabla{p} - \div{\sigma \nabla{p}}\VV
+\nabla (\sigma \nabla{p})\VV -D\VV(\sigma \nabla{p}).
\end{cases}
\]
Therefore, by taking the scalar product with $\nabla{p}$ in the first equation and with $\nabla{u}$ in the second one, we get
\[
\begin{cases}
\begin{aligned}
\curl (\sigma \nabla{u} \times \VV)\cdot \nabla{p}
&={\dd} \sigma \nabla{u} \cdot \nabla{p}
	- \div{\sigma \nabla{u}}\VV\cdot \nabla{p}
	+ \nabla (\sigma \nabla{u})\VV \cdot \nabla{p}
	- D\VV(\sigma \nabla{u})\cdot \nabla{p},\\
\curl (\sigma \nabla{p} \times \VV)\cdot \nabla{u}
&={\dd} \sigma \nabla{p} \cdot \nabla{u}
	- \div{\sigma \nabla{p}}\VV\cdot \nabla{u}
	+ \nabla (\sigma \nabla{p})\VV \cdot \nabla{u}
	-D\VV(\sigma \nabla{p})\cdot \nabla{u}.
\end{aligned}
\end{cases}
\]
By integrating over $\Omega$ and applying the first formula in Lemma~\ref{eq:expressions_j} and \eqref{eq:grad_sig_gradu} for $v=p$, we obtain
\begin{equation}\label{eq:1}
\begin{aligned}
0&=\intO{{\dd} \sigma \nabla{u} \cdot \nabla{p}}
 - \intO{\div{\sigma \nabla{u}}\VV\cdot \nabla{p}}
+\intO{\sigma \nabla^{2} u \VV \cdot \nabla{p}}\\
&\qquad -\intO{D\VV(\sigma \nabla{u})\cdot \nabla{p}}+
\intO{(\nabla \sigma \cdot \VV)(\nabla{u} \cdot \nabla{p})},
\end{aligned}
\end{equation}
while the second equation becomes
\begin{equation}\label{eq:2}
\begin{aligned}
0&=\intO{{\dd} \sigma \nabla{p} \cdot \nabla{u}}
 	- \intO{\div{\sigma \nabla{p}}\VV\cdot \nabla{u}}
	+ \intO{\nabla (\sigma \nabla{p})  \VV \cdot \nabla{u}}\\
&\qquad- \intO{D\VV^{\top}(\sigma \nabla{u}\cdot \nabla{p})}.
\end{aligned}
\end{equation}
By summing \eqref{eq:1} and \eqref{eq:2}, and then using the second formula of the lemma, we obtain a new expression for ${{J}}_{1}$:
\[
{{J}}_{1}
	= - \intO{\div{\sigma \nabla{u}}\VV \cdot \nabla{p}}
	- \intO{\div{\sigma \nabla{p}}\VV \cdot \nabla{u}}
	+ \intG{ \sigma \dn{p}\dn{u} \Vn }.
\]

Finally, let us obtain an equivalent expression for the integral ${{J}}_{2}$.
Knowing that $C=\dd\vect{I} - (D\VV)^\top$, we have
\[
\begin{aligned}
{{J}}_{2}
	&=-\intO{ \left( C^{\top}\bb \cdot \nabla{u} 
	+ D\bb \VV \cdot  \nabla{u} \right) p }\\
	&=-\intO{  \dd(\bb \cdot \nabla{u})p }
	+\intO{D\VV (\bb\cdot \nabla{u})p}
	- \intO{D\bb( \VV \cdot  \nabla{u}  )p }.
\end{aligned}
\]
Now, utilizing \eqref{eq:grad_sig_gradu} with $\vect{\varphi} =\bb$ and $\vect{\psi}=\VV$, we obtain
\color{black}
\[
\curl(\bb\times \VV)
	= {\dd} \bb -\operatorname{div}{\bb}\VV + D\bb\VV-D\VV \vect  {b},
\]
\color{black}
from which it easily follows that
\[
\curl(\bb\times \VV)\cdot \nabla{u} p
= {\dd} (\bb \cdot \nabla{u}) p
	- \operatorname{div}{\bb}(\VV \cdot \nabla{u}) p 
	+ D\bb(\VV\cdot \nabla{u}) p
	- D\VV (\vect  {b}\cdot \nabla{u}) p.
\]
By integrating over $\Omega$, we deduce the desired expression.
This completes the proof of the lemma.
\end{proof}
Now, to finish the proof of Proposition~\ref{prop:shape_gradients}, we simply need to utilize the identities established in the previous lemma.
Indeed, by Lemma~\ref{eq:expressions_j} together with the fact that $\div{\sigma \nabla{u}}=\bb\cdot\nabla{u}$ in $\Omega$, $d\JD(\Omega)[\VV]$ equates to
\[
\begin{aligned}
{{J}}_{1}+{{J}}_{2}
&=-\intO{\div{\sigma \nabla{u}}\VV \cdot \nabla{p}}
	-\intO{\div{\sigma \nabla{p}}\VV \cdot \nabla{u}}
	+ \intG{ \sigma \dn{p}\dn{u} \Vn }\\
&\qquad - \intO{\curl(\bb\times \VV)\cdot\nabla{u}{p} }
	-\intO{\operatorname{div}{\bb}(\VV \cdot \nabla{u})p}.
\end{aligned}
\]
At this point, we reconsider the variational formulation of the adjoint system \eqref{eq:adjoint_system_pn} with $\psi = \VV \cdot \nabla{u}$, to get
\[
-\intO{\div{\sigma \nabla{p}}\VV \cdot\nabla{u}}
-\intO{\operatorname{div}{\bb}(\VV\cdot\nabla{u})p}
=\intO{(\bb\cdot \nabla{p})(\VV\cdot \nabla{u})}.
\]
Then, evidently, after a rearrangement of the integrals
\[
\begin{aligned}
{{J}}_{1}+{{J}}_{2}
&=\intO{ \left[ (\bb\cdot \nabla{p})(\VV\cdot \nabla{u})
	- (\bb \cdot\nabla{u})(\VV \cdot \nabla{p})
	- \curl(\bb\times \VV)\cdot\nabla{u} p \right]}
	+ \intG{ \sigma \dn{p}\dn{u} \Vn }\ferj{.}
\end{aligned}
\]
However, it can be demonstrated that the first integral actually vanishes, i.e.,
\begin{equation}\label{eq:vanishing_term}
\intO{ \left[ (\bb\cdot \nabla{p})(\VV\cdot \nabla{u})
	- (\bb \cdot\nabla{u})(\VV \cdot \nabla{p})
	- \curl(\bb\times \VV)\cdot\nabla{u} p \right]}
	=0,
\end{equation}
(see Appendix \ref{appx:proofs_of_key_identities}), leading us to the desired result $d\JD(\Omega)[\VV] = {{J}}_{1}+{{J}}_{2} = \intG{ \sigma \dn{\un}\dn{\pn} \Vn }$.
\subsubsection{Proof of Proposition~\ref{prop:shape_gradients} via shape derivative of state}\label{subsec:computation_of_shape_derivative}
	Let $\Omega \in {C}^{2,1}$ be an admissible domain and $\VV \in \sfTheta \cap {{C}}^{2,1}(\overline{D})^{d}$.
	Clearly, $J_{i}$ is shape differentiable for all $i\in\{D,N\}$.
	By Hadamard's boundary differentiation formula, one obtains
	\[
		d\JD(\Omega)[\VV] = \intS{(\un - f) \un'},
	\]
	where $\unp = \dotun(\Omega)[\VV]-\VV\cdot\nabla \un$ satisfy equation \eqref{eq:unp}.
	Multiplying \eqref{eq:unp} by $p = \pn$ and then integrating over $\Omega$, we get
	\[
	\begin{aligned}
	-\intO{\div{\sigma \nabla{u^{\prime}}}p}
	+\intO{\bb\cdot \nabla{u^{\prime}}{p} }=0
	\end{aligned}
	\]
	Applying Green's formula while noting that $\dn{u}^{\prime}=0$ on $\Sigma$ and $p=0$ on $\Gamma$, we obtain
	\begin{equation}\label{eq:identity_bb}
	\begin{aligned}
	-\intO{\sigma \nabla{u^{\prime}} \cdot \nabla{p}}
	=\intO{(\bb\cdot \nabla{u^{\prime}}){p} }.
	\end{aligned}
	\end{equation}
	Let us now multiply \eqref{eq:adjoint_system_pn} by $u^{\prime} = \unp$ and then integrate over $\Omega$ to obtain
	\[
        \begin{aligned}
        \intO{ \div{\sigma \nabla{p}}u^{\prime} }
        + \intO{ \left [ ({\bb}\cdot \nabla{p})u^{\prime} + p\operatorname{div}{\bb}u^{\prime} \right] }=0.
        \end{aligned}
	\]
	Employing Green's formula and then utilizing identity \eqref{eq:identity_bb} after, we get the following equation after some rearrangements
	\begin{equation}\label{eq:first_equation_line}
        \begin{aligned}
	\intO{\left[ p u^{\prime} \operatorname{div}{\bb} + ({\bb}\cdot \nabla{p}) u^{\prime} + (\bb\cdot \nabla{u^{\prime}}) p \right]}
	+\intdO{\sigma \dn{p} u^{\prime}}=0.
        \end{aligned}
	\end{equation}
	Hence, with $p=\pn=0$ on $\Gamma$, we obtain from \eqref{eq:first_equation_line} and \eqref{eq:important_identity} (putting $\varphi = u^{\prime}$ and $\psi = p$) the equation
	\[
	\intS{({p} {\bb}\cdot {\nn}+\sigma \dn{p})u^{\prime}}
	=-\intG{\sigma \dn{p} u^{\prime}}.
	\]
	However, we know that ${p} {\bb}\cdot {\nn}+\sigma \dn{p}= u-f$ on $\Sigma$ form \eqref{eq:adjoint_system_pn} and $u^{\prime}=-\dn{u}\Vn$ on $\Gamma$ from \eqref{eq:unp}.
	Thus, we deduce that 
	\[
	d\JD(\Omega)[\VV]
	= \intS{(u-f)u^{\prime}}
	= \intG{\sigma\dn{p} \dn{u} \Vn},
	\]
	as desired. 
\section{Numerical Experiments} 
\label{sec:numerical_experiments}

To implement the proposed shape methods  numerically, we will utilize a conventional shape-gradient-based descent technique combined with a finite element method (FEM) based on our previous work \cite{CherratAfraitesRabago2025}. 
Before moving forward, we emphasize that identifying $\omega$ is the same as identifying $\Omega$ since $\Omega = D \setminus \overline{\omega}$. Therefore, defining $\omega$ also means defining $\Omega$, and vice versa.
\subsection{Conventional numerical algorithm}
\label{subsec:Numerical_Algorithm}
To compute the $k$th approximation $\Omega^{k}$, we carry out the following procedures:
\begin{description}
\setlength{\itemsep}{0.1em}
	\item[1. \it{Initilization}] Fix a number $\mu > 0$ and choose an initial shape $\Omega^{0}$.  
	\item[2. \it{Iteration}] For $k = 0, 1, 2, \ldots$, do the following:
		\begin{enumerate}
			\item[2.1] Solve the state and adjoint state systems on the current domain $\Omega^{k}$. 
			\item[2.2] Compute the vector $\VV^{k}$ in $\Omega^{k}$ according to the following problem: 
                            Find a vector $\VV \in \HSigma^{d}$ which solves the variational equation
                            \[
                            	\intO{( \nabla \VV: \nabla \vect{\varphi} + \VV\cdot \vect{\varphi} ) } 
                            		= - \intG{{G}\nn \cdot \vect{\varphi}}, \qquad  \forall \vect{\varphi} \in \HSigma^{d},
                            \]
                            where $G$ denotes the kernel of the shape gradient.

			\item[2.3] Compute the step size using the formula $t^{k} = \mu J(\Omega^{k})/|\VV^{k}|^2_{{H}^{1}(\Omega^{k})^{d}}$.
			
			\item[2.4] Update the current domain by setting $\Omega_{k+1} = (\operatorname{id}+t^{k}\VV^{k})\Omega^{k}$. 
		\end{enumerate}
	\item[3. \it{Stop Test}] Repeat the \textit{Iteration} until convergence.
\end{description}
\fergy{The algorithm above generates a sequence of approximations to the exact inclusion $\omega^{\ast}$, starting from an initial guess and employing domain variation techniques in shape optimization \cite{Doganetal2007}.
In Step~2.2, a Riesz representation of the shape gradient $G$ is computed as part of an extension-regularization strategy aimed at suppressing rapid oscillations along the free boundary.
This stabilizes the approximation process and prevents premature termination.
As a result, we obtain a \textit{Sobolev gradient} representation $\VV \in \HSigma$ of the deformation vector in the normal direction $-G\nn$, yielding a smoothed and preconditioned extension of $-G\nn$ throughout the entire domain $\Omega$.
This, in turn, allows for the deformation of the discretized computational domain by adjusting the positions of movable mesh nodes--modifying not only the boundary interface but also the interior nodes of the domain.
Further details on discrete gradient flows for shape optimization problems can be found in \cite{Doganetal2007}.

We emphasize that the domains are discretized using Delaunay triangulation. 
In our approach, the mesh nodal points serve as the design variables to approximate the exact inclusion solution. 
Alternatively, a parametrization method coupled with an adaptive technique can also be applied. Interested readers are referred to \cite{CaubetDambrineKatebTimimoun2013} for an application of this approach to related work, specifically in the context of obstacle detection in fluid flow.
}

In Step 2.3, $\mu > 0$ serves as a scaling factor, adjusted downwards to prevent the formation of inverted triangles in the mesh after the update. The determination of this step size follows an Armijo-Goldstein-like criterion for the shape optimization approach, as detailed in \cite[p.~281]{RabagoAzegami2020} where the step size is further reduced to avoid reversed triangles after the mesh update.
Essentially, the step size is dependent on the mesh size of the triangulation.
Additionally, convergence is reached when a finite number of iterations is completed.

In the subsequent numerical experiments, we will first analyze the case where $\sigma = \text{constant} > 0$, followed by the case $\sigma(x) = \sigma_{0}(x)\vect{I}$ for $x \in D$ (see subsections \ref{subsec:Numerical_Examples_2D}, \ref{subsec:ADMM_2D_numerics}, and \ref{subsec:ADMM_3D_numerics}). 
Here, $\sigma_{0}$ is a scalar function and $\vect{I}$ is the identity matrix. 
Afterwards, we will examine a case where $\sigma$ is a more general matrix; see subsection \ref{subsec:ADMM_3D_numerics_general_case}.
%
%
%
\subsection{Numerical examples in 2D}
\label{subsec:Numerical_Examples_2D} 
\fergy{For the first set of numerical examples in two} spatial dimensions, we make the following broad assumptions: the domain $D$ is the unit circle centered at the origin, $\sigma(x) = 1.1$ and $\bb(x) = (1.0 + 0.5 \sin{(\arctan(x_{2}/x_{1}))},  1.0 + 0.5 \cos{(\arctan(x_{2}/x_{1}))})^{\top}$, where $x = (x_1, x_{2}) \in D \subset \mathbb{R}^{2}$.
Additionally, the data is synthetically constructed. 
Specifically, we consider the Neumann boundary condition $g(x) = e^{x_{1}}$. 
We then compute the trace of the state solution $u$ of \eqref{eq:state_un} to extract the measurement $f = u$ on the accessible boundary $\partial \Omega$. 
To avoid \textit{inverse crimes} (see \cite[p.~154]{ColtonKress2013}) in generating the measurements, we construct the synthetic data using a different numerical scheme. 
This involves employing a larger number of discretization points and applying ${P2}$ finite element basis functions in the \textsc{FreeFem++} code \cite{Hecht2012}, compared to the inversion process. 
In the inversion procedure, all variational problems are solved using ${P1}$ finite elements, and we discretize the domain with a uniform mesh size of $h = 0.03$.

We shall test our proposed identification procedure by considering the following exact geometries for the unknown boundary $\Gamma^{\ast}$:
\begin{description}
	\item[Case 1:] $\Gamma^{\ast} = \Gamma^{\ast}_{1} := \left\{\begin{pmatrix} -0.25 + \dfrac{0.6+0.54\cos{t}+0.06\sin{2t}}{1+0.75\cos{t}}\cos{t} \\[0.75em] 0.05 + \dfrac{0.6+0.54\cos{t}+0.06\sin{2t}}{1+0.75\cos{t}}\sin{t}\end{pmatrix}, \forall t \in [0, 2\pi) \right\}$;
	\color{black}
	\item[Case 2:] $\Gamma^{\ast} = \Gamma^{\ast}_{2},$ where $\Gamma^{\ast}_{2}$ is the boundary of the domain $\omega^{\ast}=(-0.55,0.55)^{2}\setminus[0,0.55]^{2}$.	
\end{description}
\fergy{For the forward problem, the exterior boundary is discretized using $N_{ext}^{\ast} = 500$ points in both cases. 
In Case 1, the exact interior boundary is discretized with $N_{int}^{\ast} = 700$ points, while in Case 2, we use $N_{int}^{\ast} = 900$ points. 
For the inversion procedure, the exterior and interior boundaries are discretized with $N_{ext} = 120$ and $N_{int} = 100$ points, respectively, in all cases.}
\color{black}

The numerical algorithm terminates after $N$ iterations. 
\fergy{In the numerical experiments, $N$ is chosen to be large enough that the cost value has already saturated. 
For the 2D cases, we set $N = 200$ iterations for Case 1 and $N = 600$ iterations for Case 2}\footnote{The algorithm could run for extended periods, and the stopping criterion could be enhanced. However, this straightforward termination condition already yields satisfactory results based on our experience.}, and we set $\mu = 0.5$. 
\fergy{For all the 3D numerical experiments considered below, we use $N = 600$.}

We will also test our algorithm with noisy data, where the noise level $\delta$ is expressed as a percentage. 
\fergy{Specifically, we introduce noise into the measurements by setting $\ud^{\delta} = (1 + \delta)\ud^{\ast}$ and define $g|_{\Sigma} := \dn{\ud^{\delta}}$, where $\ud^{\ast}$ is the exact solution of the forward problem for a given input $f$.}
At this juncture, it is important to note that least-squares formulations of the boundary identification problem are ill-posed. 
In simpler terms, this means that even small changes in the measurements can lead to significant differences in the identified unknown boundaries. 
\fergy{Moreover, the presence of noise can destabilize the reconstruction, as it introduces irregularities along the boundary after some iterations in the procedure.}
\fergy{To ensure the smooth evolution of the free boundary during each iteration, regularization techniques are typically incorporated alongside the extension-regularization method applied to the deformation field. A common approach involves penalizing the perimeter (or surface area in three dimensions) of the free boundary through a term like $ \eta \intG{1} $, where $ \eta > 0 $ is a small regularization parameter.}

In our numerical tests, we will also examine the effect of perimeter regularization, as well as the effect of adaptive mesh refinement. 
\fergy{The mesh refinement is performed using the built-in \texttt{adaptmesh} function in FreeFEM++. 
This process refines the mesh by increasing the number of discretization points, which are adapted according to the gradient of the solution. 
This procedure can be considered a form of mesh regularization, as repeated mesh deformations may lead to a deterioration in mesh quality, resulting in overly thin or stretched/elongated elements.}
We focus on the reconstruction using \ferj{the Dirichlet-data tracking approach} given by Equation~\eqref{eq:Dirichlet_cost_function} since \ferj{the Neumann-data tracking approach} \eqref{eq:Neumann_cost_function} is less stable from our experience; see \cite{CherratAfraitesRabago2025}.


The numerical results for the test scenarios are shown in Figure~\ref{fig:figure1} and Figure~\ref{fig:figure2}. 
These figures present the identified inclusion with and without noise ($\delta = 0\%, 10\%, 15\%$). 
The thick black line outlines the object's surface, while the red line indicates the exact inclusion geometry. 
Black dotted lines represent initial guesses, and other lines show identification results at different noise levels.

Our shape optimization method achieves reasonable reconstructions, especially in noise-free cases. 
However, it accurately locates the inclusion but cannot capture its exact shape, particularly in concave regions, as expected. 
This limitation, consistent with prior findings, persists even in less regular cases, such as those violating ${C}^{1,1}$ regularity assumptions \cite{AfraitesRabago2025,CherratAfraitesRabago2025}. 
The cost histories in the figures confirm higher computed costs with measurement noise.
\begin{figure}[htp!]
\centering 
\resizebox{0.3\linewidth}{!}{\includegraphics{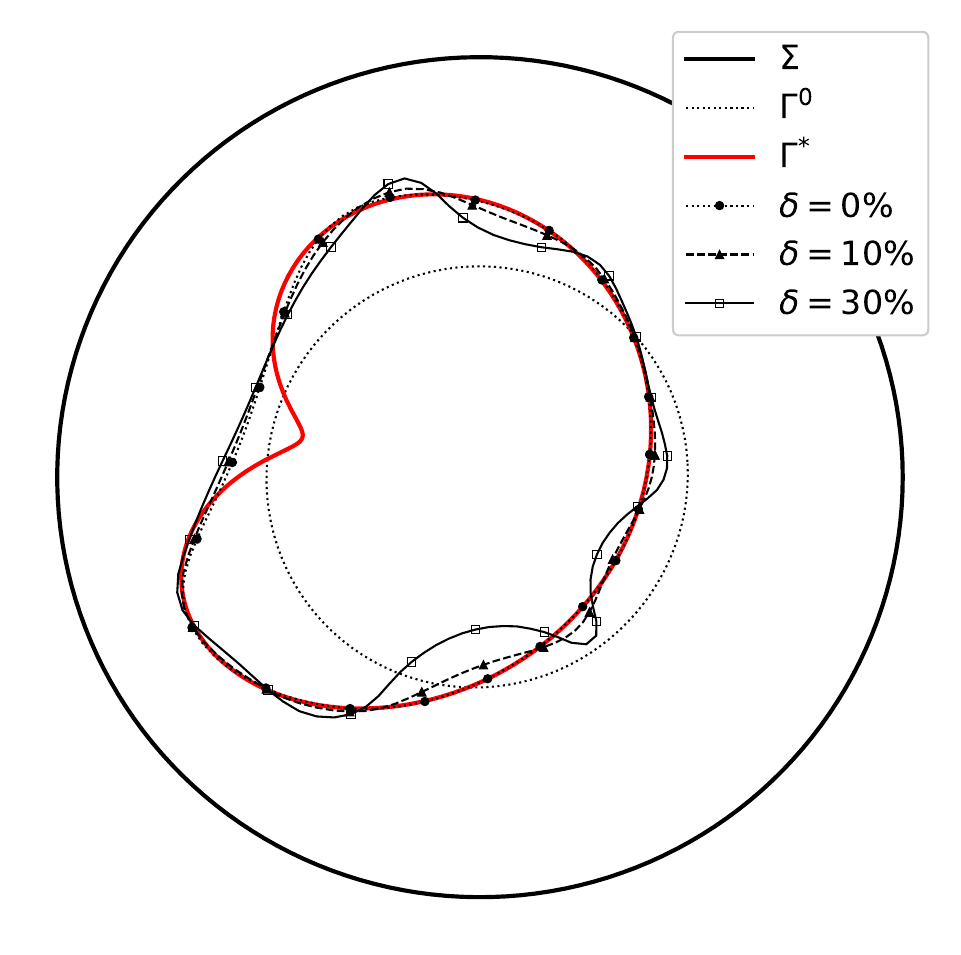}} \quad
\resizebox{0.3\linewidth}{!}{\includegraphics{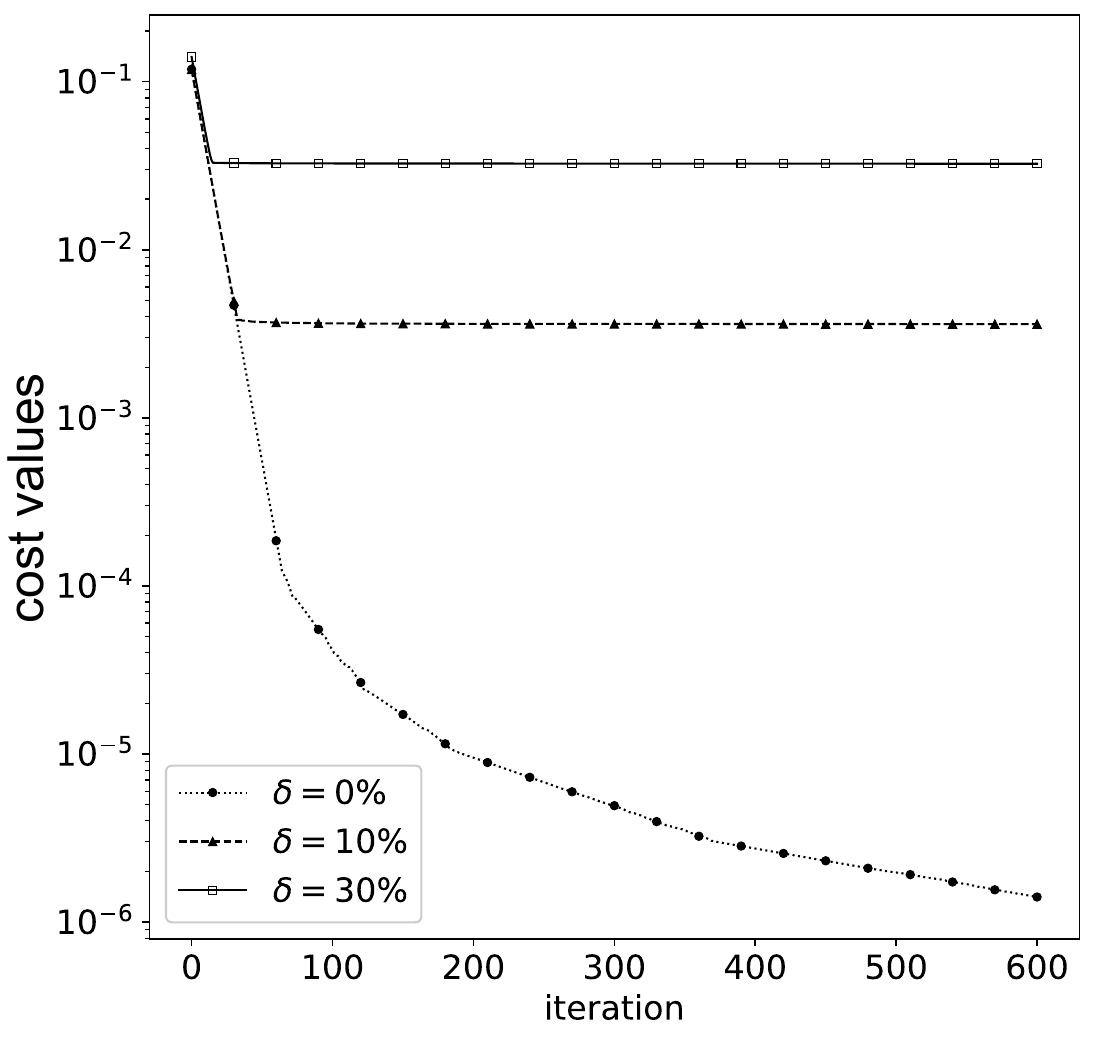}}
\caption{Reconstruction of $\Gamma^{\ast}_{1}$ in the absence and presence of noise ($\delta = 0\%, 10\%, 30\%$).
\ferj{The plot on the right shows the cost value histories for each case.}}
\label{fig:figure1}
\end{figure}
\begin{figure}[htp!]
\centering 
\resizebox{0.3\linewidth}{!}{\includegraphics{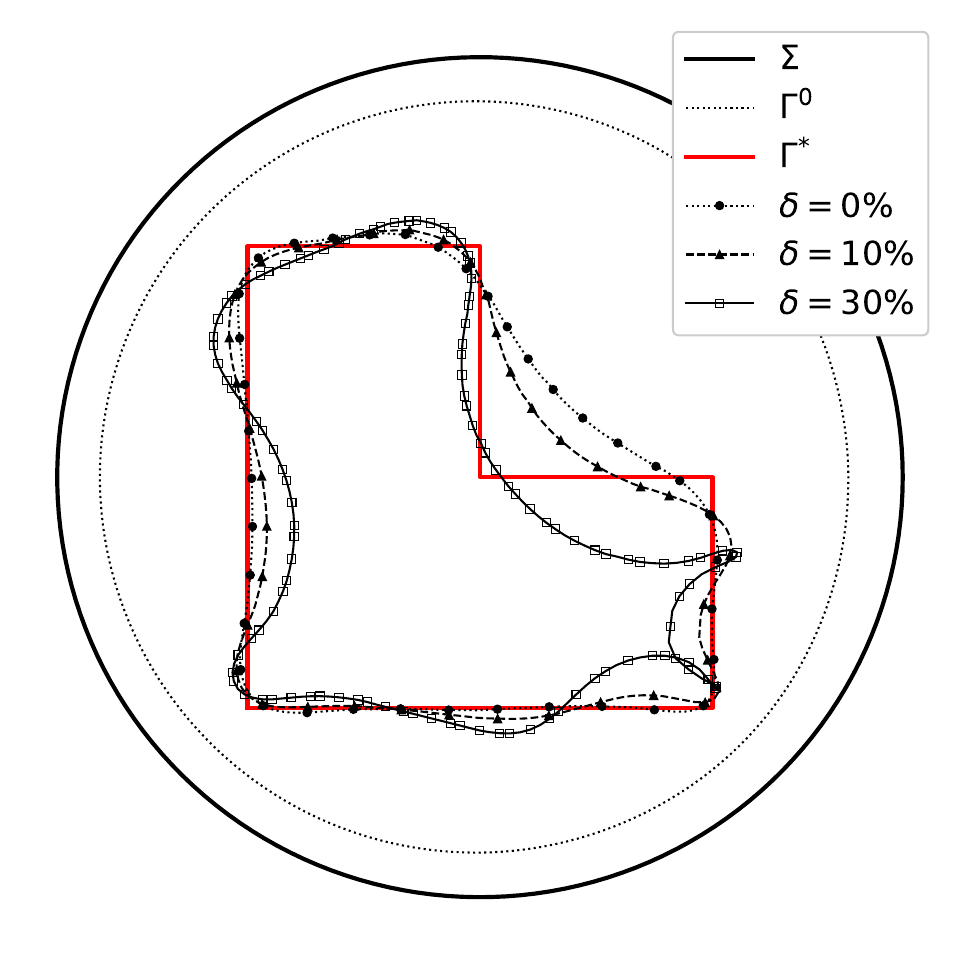}} \quad
\resizebox{0.3\linewidth}{!}{\includegraphics{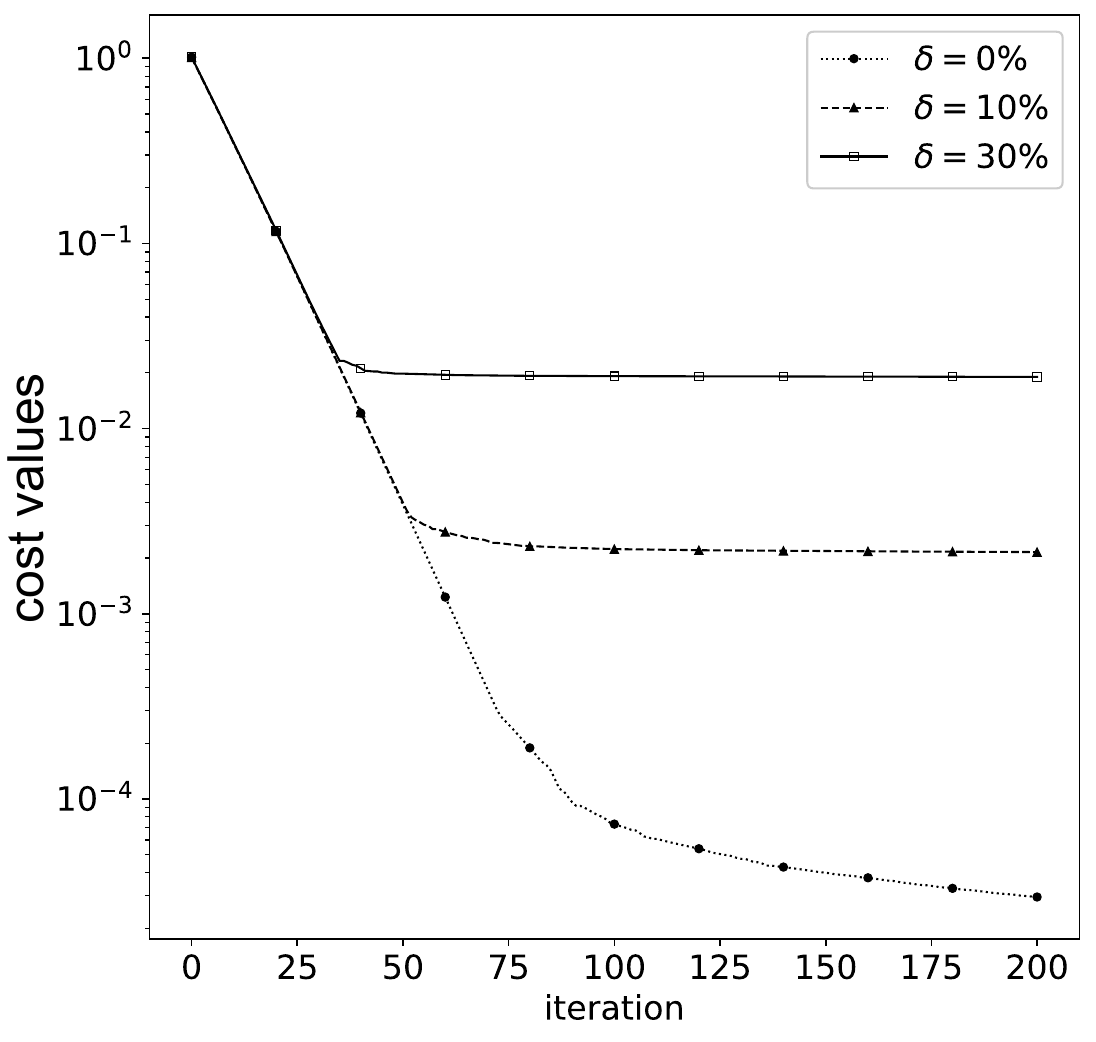}}
\caption{Reconstruction of $\Gamma^{\ast}_{2}$ in the absence and presence of noise ($\delta = 0\%, 10\%, 30\%$).
\ferj{The plot on the right shows the cost value histories for each case.}}
\label{fig:figure2}
\end{figure}

To evaluate the impact of employing perimeter regularization and adaptive mesh refinement, we conduct further tests focusing on reconstructing $\Gamma^{\ast}_{2}$. 
The summarized results, presented in Figure~\ref{fig:figure3}, compare scenarios with and without regularization (with specific weights). 
The graph also includes reconstructions with \fergy{perimeter} regularization and adaptive mesh refinement at each iteration. 
Although some improvements are observed with perimeter regularization, they are not substantial, even with adaptive remeshing. 
Additionally, while there are differences in the reconstructed shapes, the variations in cost histories are minimal. 
Despite the difficulty the method has in detecting sharp inclusion corners, the reconstructed shapes are fairly accurate representations of the actual inclusion geometry, even in the presence of significant noise.
%
%
%
%
%
%

\begin{figure}[htp!]
\centering 
\resizebox{0.3\linewidth}{!}{\includegraphics{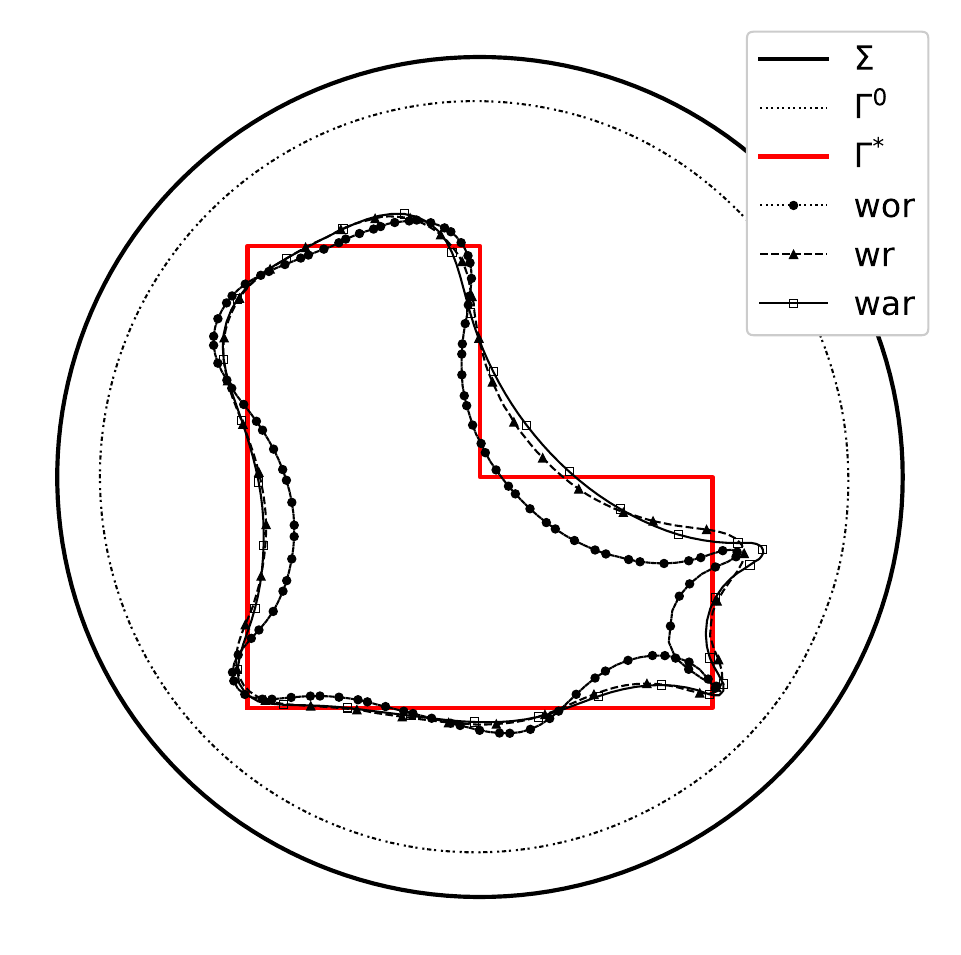}} \quad
\resizebox{0.3\linewidth}{!}{\includegraphics{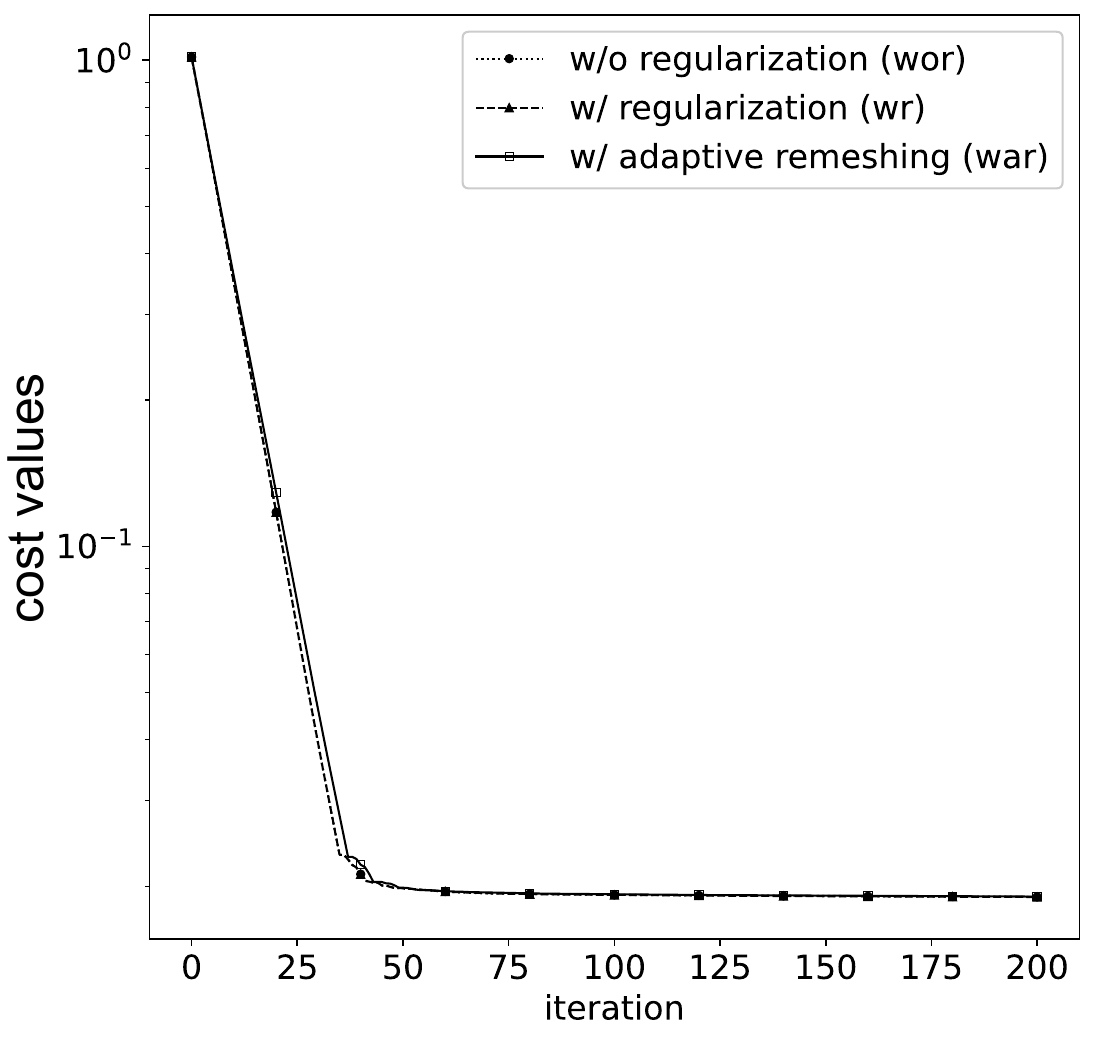}}
\caption{Reconstruction of $\Gamma^{\ast}_{2}$ with noisy data ($\delta = 30\%$) with perimeter regularization and with adaptive mesh refinement.
\ferj{The plot on the right shows the cost value histories for each case.}}
\label{fig:figure3}
\end{figure}

Finally, for the last set of numerical examples for 2D case, we set $\sigma = 2 + 0.5\sin(0.5\pi x) \cos(0.5\pi y)$, $g(t) = e^{\sin{t}}$, $t \in [0,2\pi)$, and $\beta = \lambda = 0.0001$. 
The results are shown in Figures~\ref{addfig:figure1} and \ref{addfig:figure2} using the conventional shape optimization method of tracking the Dirichlet-data in $L^{2}$ sense with perimeter regularization imposed for reconstruction with noisy data.
Notice from the plots that in the absence of noise, we obtain fair reconstructions of the obstacles, indicating the unknown shapes have concavities. 
As expected, however, the reconstruction is much less accurate at the presence of noise, and it is difficult to \fergy{indentify} correctly the concave parts of the obstacles.

\begin{figure}[htp!]
\centering 
\resizebox{0.3\linewidth}{!}{\includegraphics{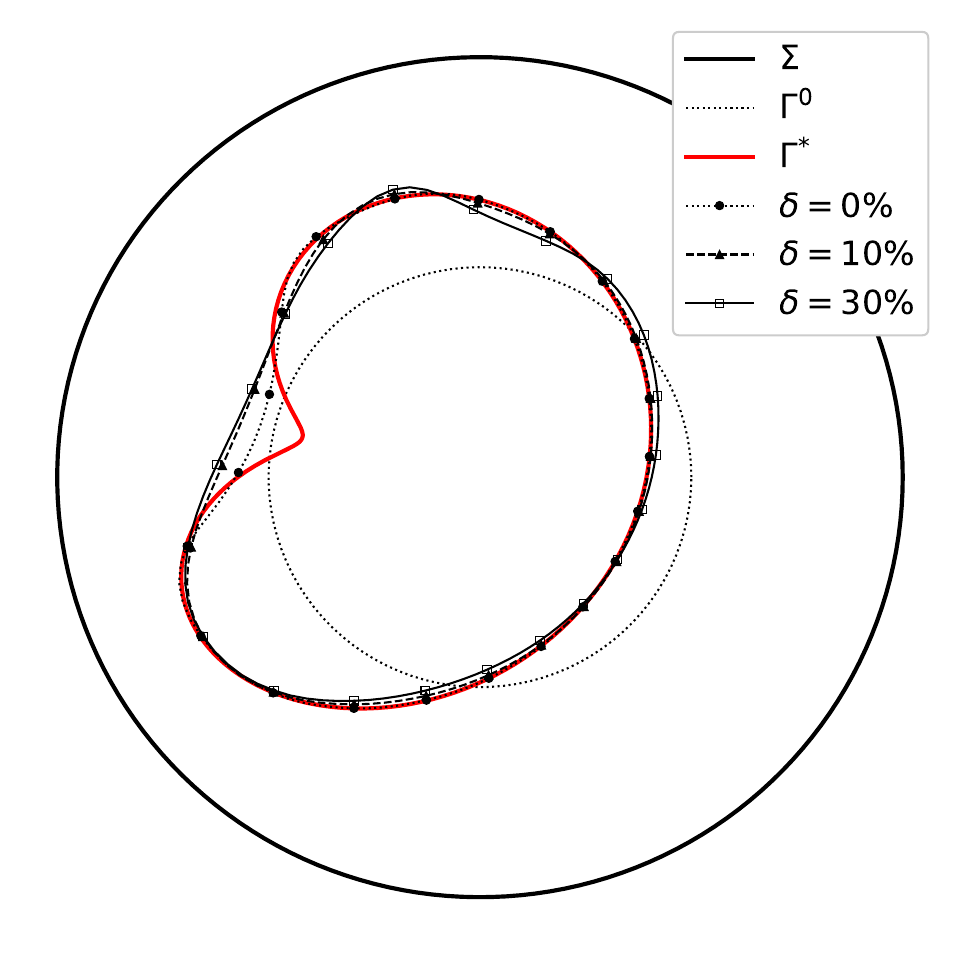}} \quad
\resizebox{0.3\linewidth}{!}{\includegraphics{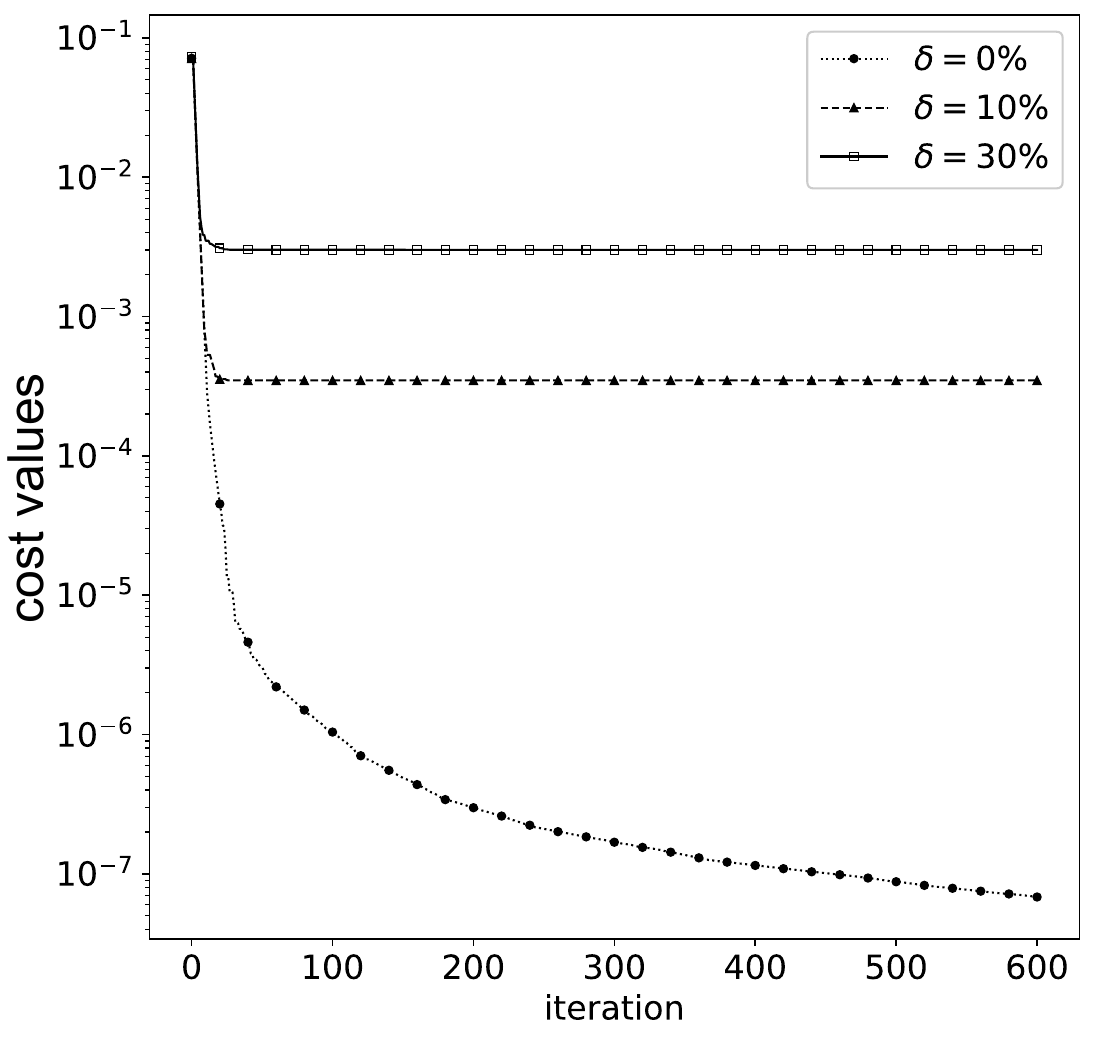}}
\caption{Reconstruction of $\Gamma^{\ast}_{1}$ with $\delta = 0\%, 10\%, 30\%$ using \eqref{eq:Dirichlet_cost_function} with perimeter regularization but without adaptive mesh refinement.
\ferj{The right plot shows the histories of cost values for each noise levels.}}
\label{addfig:figure1}
\end{figure}

\begin{figure}[htp!]
\centering 
\resizebox{0.3\linewidth}{!}{\includegraphics{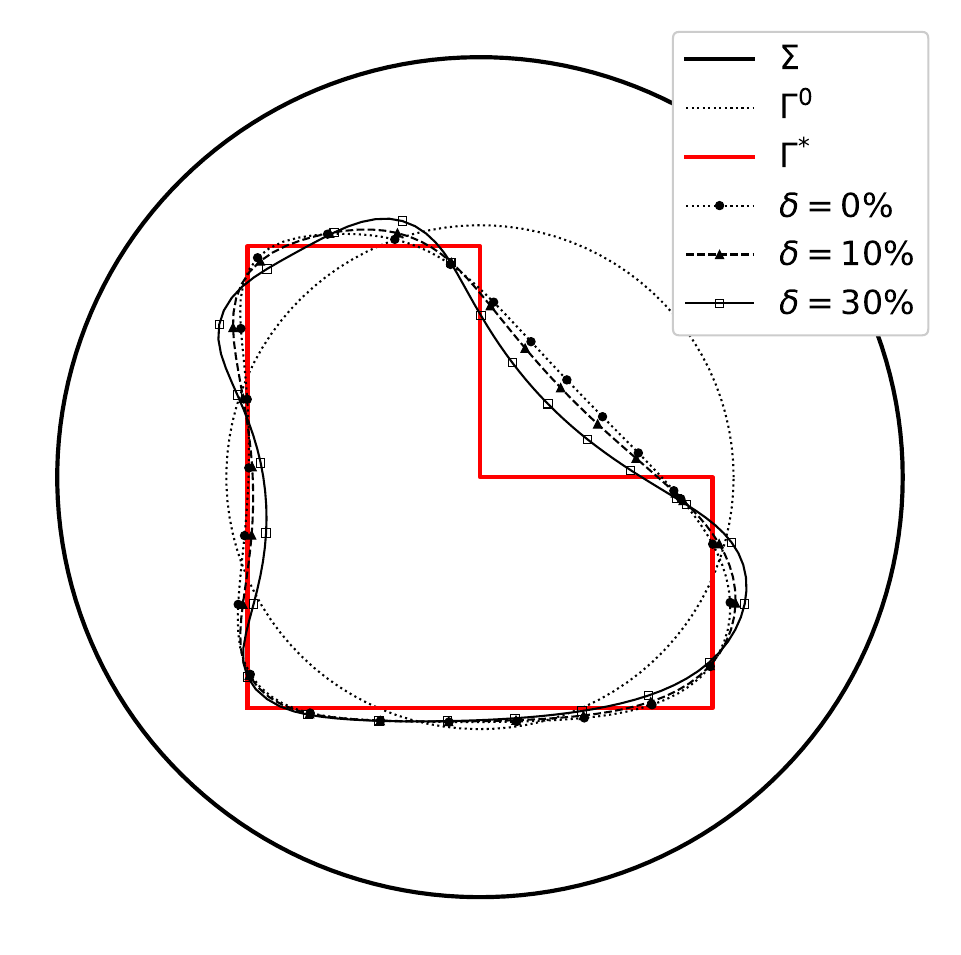}} \quad
\resizebox{0.3\linewidth}{!}{\includegraphics{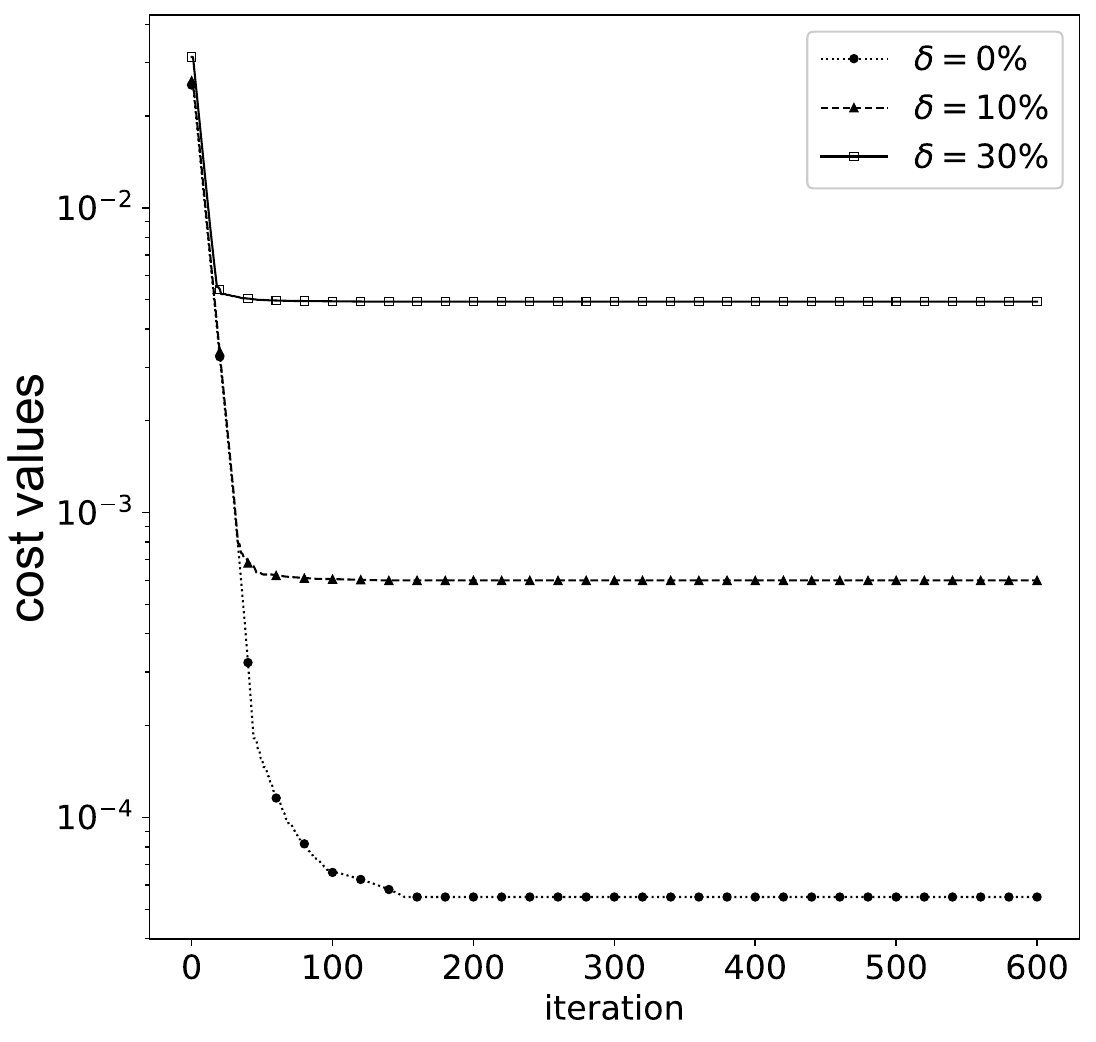}}
\caption{Reconstruction of \ferj{$\Gamma^{\ast}_{2}$} with $\delta = 0\%, 10\%, 30\%$ using \eqref{eq:Dirichlet_cost_function} with perimeter regularization but without adaptive mesh refinement.
\ferj{The right plot shows the histories of cost values for each noise levels.}}
\label{addfig:figure2}
\end{figure}

\subsection{Alternating direction method of multipliers}\label{subsec:ADMM_algorithm}
To improve detection in the case of measurements with a high level of noise, we propose \fergy{an application of the Alternating Direction Method of Multipliers (ADMM) in the context of shape optimization, originally introduced in \cite{RabagoHadriAfraitesHendyZaky2024} for the cavity identification problem governed by a simple Laplace equation}. 
The proposed method is based on the introduction of an auxiliary variable into the cost functions provided in \eqref{eq:Dirichlet_cost_function} and \eqref{eq:Neumann_cost_function}.
Since the new minimization formulation we will develop essentially involves the same modifications as applied to \eqref{eq:Dirichlet_cost_function} and \eqref{eq:Neumann_cost_function}, we shall focus solely on tracking the Dirichlet data.
We will then proceed to solve the resulting minimization problem via an alternating direction method of multipliers or ADMM developed in \cite{RabagoHadriAfraitesHendyZaky2024}. 
Hence, the \ferj{following} discussion will be based on \cite{RabagoHadriAfraitesHendyZaky2024}.

Hereinafter, we will denote $J := \JD$, and as before, $\Omega \in {C}^{2,1}$ is an admissible domain and for later use, \fergy{we assume that $\VV \in \sfTheta \cap {{C}}^{2,1}(\overline{D})^{d}$ and $g \in H^{3/2}(\Sigma)$; see Remark~\ref{rem:higher_regularity_for_shape_derivative_of_the_state}}.
Now, to start, we reformulate our original shape inverse problem \eqref{eq:inverse_problem} into the following shape optimization problem with inequality constraints.
\begin{problem}\label{prob:optimal_shape_problem}
Let $a$ and $b$, $b\geqslant a$, be given fixed constants.
Find the shape $\omega^{\ast}$ in the space of admissible set
\[
	\mathcal{O}_{ad}=\left\{ \omega \in \mathcal{A} \mid \text{$a\leqslant \uo\leqslant b$ a.e. in $\Omega$ where $\uo$ solves problem \eqref{eq:state_un}} \right\}
\]
such that
	\begin{equation}\label{eq:control} 
\omega^{\ast} 
	= \argmin_{\omega \in \mathcal{O}_{ad}} J({\Omega})
	:= \argmin_{\omega \in \mathcal{O}_{ad}} \left\{ \frac{1}{2} \intS{\abs{\uo-f}^2} \right\}.
	\end{equation}	
\end{problem}

\fergy{A comment on the choise of $a$ and $b$ is necessary.
To choose suitable values for $a$ and $b$, omne can apply the maximum principle in Sobolev spaces \cite[Thm.~8.1, p.~179]{GilbargTrudinger2001}
in the case where the data $f$ is prescribed, one can set $b = \sup_{\Sigma} f$. 
On the other hand, because $u = 0$ on $\Gamma$, then a straightforward choice is to take $a=0$.
These choices of $a$ and $b$ will be used later in our numerical experiements.}

We highlight that identifying $\omega$ is equivalent to identifying $\Omega$, given that $\Omega = D \setminus \overline{\omega}$. Thus, whenever we define $\omega$, we are simultaneously defining $\Omega$, and vice versa.
To directly incorporate the inequality constraint in the cost function, we will introduce the auxiliary variable $v$ which satisfies $v=\uo$ a.e. in ${\Omega}$ and consider the set $\mathcal{E}$ defined as follows
\[
	\mathcal{E}=\left\{ (\omega,v)\in \mathcal{O}_{ad} \times L^{2}({\Omega}) \mid \text{$\uo=v$ a.e. in ${\Omega}$} \right\}.
\]
Then, we can rewrite \eqref{eq:control} as follows  
\begin{equation}\label{eq:control_Uad}
	(\omega^{\ast},v^{\ast}) = \argmin_{(\omega,v)\in \mathcal{E}} \left\{J({\Omega})+U_{\mathcal{K}}(v)\right\},
\end{equation} 
where the set $\mathcal{K}$ is the closed convex non-empty set of $L^{2}({\Omega})$ defined by
\[
	\mathcal{K}=\left\{ v \in L^{2}({\Omega}) \mid a\leqslant v\leqslant b \ \ \text{a.e. in ${\Omega}$} \right\},
\]
while $U_{\mathcal{K}}$ is the indicator functional of the set $\mathcal{K}$; that is, $U_{\mathcal{K}}(v) = 0$ if $v \in {\mathcal{K}}$, otherwise, $U_{\mathcal{K}}(v) = \infty$ if $ v \in L^{2}({\Omega})\setminus \mathcal{K}$.

To apply ADMM to the control model $\eqref{eq:control_Uad},$ we need to define the augmented Lagrangian functional first.
This is possible since the minimum of problem $\eqref{eq:control_Uad}$ is the saddle point of the following augmented Lagrangian functional
\begin{equation}\label{eq:augmented_lagrangian}
	{L}_{\beta}(\omega,v;\lambda)=J({\Omega})+U_{\mathcal{K}}(v)+\frac{\beta}{2}  \intO{\vert \uo-v\vert^2}  +   \intO{\lambda  (\uo-v )},
 \end{equation}
where $\lambda$ is the Lagrange multiplier and $\beta>0$ is a penalty parameter.
In this work, as in \cite{RabagoHadriAfraitesHendyZaky2024}, we consider a fixed value for the penalty parameter $\beta$.
While it is possible to develop an optimization scheme for $\beta$ within our main algorithm using bilevel optimization \cite{Dempe2020}, we have opted to keep $\beta$ fixed to simplify our discussion. 
This choice consistently yields good results, as we will demonstrate further below.

Now, to find a saddle point of the Lagrangian functional ${L}_{\beta}$, we will implement an approximation procedure based on ADMM. 
Specifically, starting with initial values $v^{0}, \lambda^{0} \in L^{2}({\Omega})$, we will iteratively compute the optimizer of ${L}$ for $k = 1, 2, \ldots$ by solving the following sequence of minimization problems:
\begin{align}
	\omega^{k+1}	&= \argmin_{\omega \in \mathcal{O}_{ad}}{L}_{\beta}(\omega,v^{k};\lambda^{k}); \label{eq:controle}\tag{SP1} \\
	v^{k+1}		&= \argmin_{v \in L^{2}({\Omega})}{L}_{\beta}(\omega^{k+1},v;\lambda^{k}); \label{eq:etat}\tag{SP2}\\
	\lambda^{k+1}	&= \lambda^{k}+\beta ({\uokp}-v^{k+1} ), \label{eq:parametre1}\tag{SP3} 
\end{align}
where $\uokp:=\un(\Omega^{k+1})$.

Now, utilizing ${L}_{\beta}$ given in \eqref{eq:augmented_lagrangian}, we can outline the ADMM scheme in Algorithm \ref{algo:ADMM_algorithm}.
\begin{algorithm}
\begin{enumerate}\itemsep0.3em 
	\item \textit{Input} Fix $\beta$, $a$, and $b$, and define the Cauchy pair $(f, g)$.
	\item \textit{Initialization} Choose an initial shape $\omega_{0}$. Also, set the initial values $v^{0}$ and $\lambda^{0}$.
	\item \textit{Iteration} For $k = 1, 2, \ldots$, compute ${v^{k},\lambda^{k}}$ using equations \eqref{eq:controle}--\eqref{eq:parametre1} through sequential computations:
	\[
		\{v^{k},\lambda^{k}\} \ \stackrel{\eqref{eq:controle}}{\longrightarrow} \ \omega^{k+1} 
		\ \stackrel{\eqref{eq:etat}}{\longrightarrow} \ v^{k+1}
		\ \stackrel{\eqref{eq:parametre1}}{\longrightarrow} \ \lambda^{k+1}.\vspace{-3pt}
	\]
	\item \textit{Stop Test} Repeat \textit{Iteration} until convergence.
\end{enumerate}
\caption{ADMM algorithm for the solution of problem \eqref{eq:control}.}
\label{algo:ADMM_algorithm}  
\end{algorithm}

In the following two subsections, we will outline the resolution of \eqref{eq:controle} and \eqref{eq:etat}.
\subsubsection{Solution of $\omega$-subproblem \eqref{eq:controle}}
We first look for the solution of the first $\omega$-subproblem \eqref{eq:controle} in which we need to minimize ${L}_{\beta}$ with respect to $\omega$.
The $\omega$-subproblem \eqref{eq:controle} is given as follows
\[
 \omega^{k+1}=\argmin_{\omega \in \mathcal{O}_{ad}} \left\{ J({\Omega})+U_{\mathcal{K}}(v^{k})+\frac{\beta}{2} \intO{\vert \uo-v^{k}\vert^2} +   \intO{\lambda^{k}  (\uo-v^{k})}\right\}.
\]
Let us consider the following shape functional
\[
{{Y}}^{k}({\Omega}):={L}_{\beta}(\omega,v^{k};\lambda^{k})=J({\Omega})+\frac{\beta}{2} \intO{\vert \uo-v^{k}\vert^2}  +   \intO{\lambda^{k}  (\uo-v^{k} )},
\]
where $J({\Omega})=\displaystyle\frac{1}{2}\intO{\vert \uo-f\vert^2}$ and $\uo$ solves problem \eqref{eq:state_un} over $\Omega={D}\setminus\overline{\omega}$.

Obviously, resolving the $\omega$-subproblem \eqref{eq:controle} necessitates the shape derivative of ${{Y}}^{k}$. 
Following the computations outlined in subsection \ref{subsec:computation_of_shape_derivative}, the shape derivative of \ferj{${{Y}}^{k}$} at $\Omega$, in the direction of the vector field $\VV  \in \sfTheta \cap {{C}}^{2,1}(\overline{D})^{d}$, can be formally computed as follows:
\color{black}
\begin{equation}
\begin{aligned}\label{derivative:lagrangian}
	d{{Y}}^{k}({\Omega})[\VV]
		&= \displaystyle \intS{ \big(\uo-f\big) {\unp}}
			+ \beta  \intO{\big(\uo-v^{k} \big) {\unp}} 
			+ \frac{\beta}{2} \intO{\div{ \big(\uo-v^{k} \big)^{2} \VV}} \\
		&\qquad
			+ \intO{\lambda^{k}  {\unp}}
			+ \intO{\div{ \lambda^{k} \big(\uo-v^{k} \big) \VV}} \\
		&= \displaystyle \intS{ \big(\uo-f\big) {\unp}}
			+ \beta  \intO{\big(\uo-v^{k} \big) {\unp}} 
				+ \frac{\beta}{2} \intG{ \big(v^{k} \big)^2 \Vn } \\
		&\qquad		+ \intO{\lambda^{k}   {\unp}}
				- \intG{\lambda^{k} v^{k} \Vn },
\end{aligned}
\end{equation}
\color{black}
where $\unp$ solves \eqref{eq:unp}.
In \eqref{derivative:lagrangian}, we have used the fact that $\un = 0$ on $\Gamma$ and $\VV = \vect{0}$ on $\Sigma$.

We point out that Equation~\eqref{derivative:lagrangian} is difficult to handle since we cannot find explicitly the direction $\VV$.
In fact, the computed expression with the shape derivative ${\un}^{\prime}$ is not useful for practical applications, especially in the numerical realization of the proposed shape optimization problem via an iterative procedure.
This is because the implementation requires the solution of \eqref{eq:unp} for each velocity field $\VV$, at every iteration.
To get around this difficulty, we apply the adjoint method as in subsection \ref{subsec:computation_of_shape_derivative}. 
For this purpose, we will introduce another variable $w$ -- in order to eliminate from the gradient expression the shape derivative $\unp$ -- which solves the following adjoint problem
\begin{equation}\label{eq:adjoint_problem}
\left\{
\begin{array}{rcll}
        \div{\sigma\nabla{w}} + \bb \cdot \nabla{w} + {w}\,\operatorname{div}{\bb} & = &\beta \big(\uo-v^{k}\big)+\lambda^{k}  & \text{in $\Omega$},\\
       {w} & = & 0 & \text{on $\Gamma$},\\
        \sigma\dn{w} + w \bb \cdot \nn & = & -(\uo-f) & \text{on $\Sigma$}.
\end{array}
\right.
\end{equation}
This leads us to the following expression for the shape derivative of ${{Y}}^{k}$:
\begin{equation}\label{eq:shape_derivative_of_the_Lagrangian}
	d{{Y}}^{k}({\Omega})[\VV] 
		= \intG{ {{H}}^{k}\nn \cdot \VV}
		= \intG{ \left( - \sigma\dn{w} \dn{\un} + \frac{\beta}{2} \big( v^{k} \big)^{2} - \lambda^{k} v^{k} \right) \vect{n} \cdot \VV},
\end{equation}
In practice, the computed shape derivative $-{{H}}\nn$ is not directly used in numerical procedures as a descent direction because it can cause boundary oscillations during the approximation process, leading to algorithmic instabilities. 
To mitigate this problem, we use the Sobolev-gradient method \cite{Neuberger1997} (refer to the algorithm in subsection \ref{subsec:Numerical_Algorithm}), which we will discuss next.
\subsubsection{Extension and regularization of the deformation field}
The shape gradient of $Y$, similar to $J$, is only supported on $\Gamma$ and may lack the necessary smoothness for numerical implementation, especially when using finite element methods (FEMs). 
To improve the regularity of the descent direction ${{H}}$ (omitting $k$ for simplicity) and extend its definition across the entire domain ${\Omega}$, we will utilize its $H^{1}$ Riesz representative of ${{H}}$ as done in subsection \ref{subsec:Numerical_Algorithm}.
We can then formulate a Sobolev gradient-based descent (SGBD) algorithm laid out in Algorithm \ref{algo:SGBD_algorithm} to solve \eqref{eq:controle}. 
\begin{algorithm} 
\caption{SGBD algorithm for $\omega$-subproblem \eqref{eq:controle}}\label{algo:SGBD_algorithm} 
\begin{enumerate}
	\item \textit{Input} Fix $\beta$, $a$,  $b$, and {$\varepsilon$} and set $\lambda^{k}$, $\mu$, $\Omega^{k}_{m}=\Omega^{k}$,  $u^{k}_{m}=u^{k}$, $v^{k}_{m}=v^{k}$. Also, set $m=0$. 	
	\item \textit{Iteration} For $m = 1, 2, \ldots$,
	\begin{enumerate}\itemsep0.3em  
		\item[2.1] solve \eqref{eq:state_un} and \eqref{eq:adjoint_problem} over the current domain $\Omega = \Omega^{k}_{m}$;
		\item[2.2] set $\VV_{m}^{k} = \VV $ where $\VV \in \HSigma^{d}$, $\Omega = \Omega^{k}_{m}$, solves the variational equation
		\[
		\intO{( \nabla \VV: \nabla \vect{\varphi} + \VV\cdot \vect{\varphi} ) } 
		= - \intG{{H}\nn \cdot \vect{\varphi}}, \qquad  \forall \vect{\varphi} \in H_{\Sigma,0}^1(\Omega)^{d}.
		\]
		\item[2.3] for some scalar $t^{k}=\mu J^{k}({\Omega}_{m}^{k})/\|\VV^{k}\|_{H^{1}({\Omega}_{m}^{k})^{d}}$, set $\Omega_{m+1}^{k}:=\left\{x+t^{k} \VV_{m}^{k}(x) \mid x\in \Omega^{k}_{m}\right\}$.
	\end{enumerate}
	\item \textit{Stop test} Repeat \textit{Iteration} until convergence; i.e., \textbf{while} $\Vert  d{{Y}}^{k}(\Omega_{m}^{k})[\VV^{k}_{m}] \Vert \geqslant \varepsilon$ \textbf{do} \textit{Iteration} 
	\item \textit{Output} $\Omega^{k+1}=\Omega_{m+1}^{k}$.
\end{enumerate}
\end{algorithm}

In Step 2.3 of Algorithm \ref{algo:SGBD_algorithm}, we initialize the step size $t^{k}$ using the formula $t^{0} = \mu J({\Omega}^{0})/\|\VV^{0}\|_{H^{1}({\Omega}^{0})^{d}}$, where $\mu = 0.5$. 
We maintain this step size in later iterations but adjust it to prevent inverted triangles (or tetrahedrons) within the mesh after each update. 
Alternatively, we could employ a backtracking procedure, starting with the initial step size $t^{k} = \mu J({\Omega}^{k})/\|\VV^{k}\|_{H^{1}({\Omega}^{k})^{d}}$ (where $\mu > 0$ is sufficiently small), based on a line search method for shape optimization as in \cite[p.~281]{RabagoAzegami2020}. 
However, the previously mentioned step size choice is more effective for reconstructing the unknown obstacle. 
\subsubsection{Solution of the $v$-subproblem \eqref{eq:etat}}
Now we turn our attention to the resolution of $v$-subproblem \eqref{eq:etat} by minimizing ${L}_{\beta}$ with respect to $v$.
That is, we solve the $v$-subproblem $\eqref{eq:etat}$ given by 
\[
\begin{aligned}
  v^{k+1}&=\argmin_{v \in L^{2}({\Omega})}
  	\Big\{  J(\Omega^{k+1})+U_{\mathcal{K}}(v)
		+ \frac{\beta}{2}  \intO{\vert {\uokp}-v \vert^2} 
		+ \intO{\lambda^{k}  ( {\uokp}-v )} 
	\Big\}.
  \end{aligned}
\]
Applying the projection method, we obtain
\[
	v^{k+1}=P_{\mathcal{K}}\big({\uokp}+ \lambda^{k}/\beta \big),
\]
where $P_{\mathcal{K}}(\varphi) := \max(a, \min(b, \varphi))$, for all $\varphi \in L^{2}(\fergy{\Omega})$ is the projection operator onto the admissible set $\mathcal{K}$.
\subsubsection{ADMM-SGBD algorithm}
Finally, based on the discussions above, we can now propose a modification of Algorithm \ref{algo:ADMM_algorithm} for the numerical solution of the constrained shape optimal control problem \eqref{eq:control} with an inequality constraint subject to \eqref{eq:state_un}.
More precisely, Algorithm \ref{algo:ADMM_algorithm} can be specified as a nested iterative ADMM-SGBD scheme for the optimal control problem \eqref{eq:control_Uad} following the instructions given in Algorithm \ref{algo:ADMM-SGBD}. 

\begin{algorithm}[!htp] \caption{ADMM-SGBD}\label{algo:ADMM-SGBD}
\begin{enumerate} \itemsep0.1em 
    \item \textit{Initialization} Specify $(f,g)$, and choose $\omega^{0}$, $\lambda^{0}$, $\beta$, $a$, $b$, $v^{0}$, $\mu$, and $\varepsilon$.
    \item \textit{Iteration} For $k=0,\ldots,N$,
    \begin{enumerate}\itemsep0.1em
   	 \item[2.1] compute ${\uok}$  solution of the state \eqref{eq:state_un} associated to $\omega^{k}$;
   	 \item[2.2] compute ${w}^{k}$ solution of the adjoint state \eqref{eq:adjoint_problem}; 
   	 \item[2.3] update $\omega^{k+1}$ by the gradient-descent method in Algorithm \ref{algo:SGBD_algorithm};
   	 \item[2.4] update $v^{k+1}$ as $\displaystyle v^{k+1}=\max \left( a, \min \left({\uokp}+\lambda^{k}/\beta, b \right) \right)$;
   	 \item[2.5] set $\lambda^{k+1}=\lambda^{k}+\beta ({\uokp}-v^{k+1} )$.
	\end{enumerate}
	\item \textit{Stop test} Repeat \textit{Iteration} until convergence.
\end{enumerate}
\end{algorithm}

\begin{remark}
The techniques and algorithms described earlier can be readily adapted for scenarios involving noisy data. 
When considering the addition of a regularization term, whether the data is affected by noise or not (such as through perimeter or volume regularization), these terms are integrated into the Lagrangian functional. 
This modification entails including the corresponding shape derivatives in \eqref{derivative:lagrangian}.
\end{remark}
\subsection{Numerical experiments in 2D via ADMM with space-dependent diffusion coefficient} \label{subsec:ADMM_2D_numerics}
We now demonstrate the effectiveness of the proposed ADMM scheme and its advantage over the conventional algorithm discussed in subsection \ref{subsec:Numerical_Algorithm}. 
To do so, we first replicate the numerical experiments detailed in subsection \ref{subsec:Numerical_Examples_2D} using noisy data with a noise level of $\delta = 30\%$ where $\sigma \equiv 1.1$, then after that we consider the same test case in the latter part of subsection \ref{subsec:Numerical_Examples_2D} where the diffusion coefficient $\sigma$ is non-constant.
The results of the reconstruction using Algorithm \ref{algo:ADMM-SGBD} without any additional regularization term in the Lagrangian functional are depicted in Figures~\ref{fig:figure4}, showcasing two different initial obstacle guesses. 
The reconstructions notably outperform those shown in Figures~\ref{fig:figure2} and \ref{fig:figure3}, exhibiting increased accuracy and reduced oscillations compared to the conventional approach. 
Additionally, Figures~\ref{fig:figure4} present the histories of the cost functions and gradient norms for further insight.
%
%
%
%
\begin{figure}[htp!]
\centering 
\resizebox{0.235\linewidth}{!}{\includegraphics{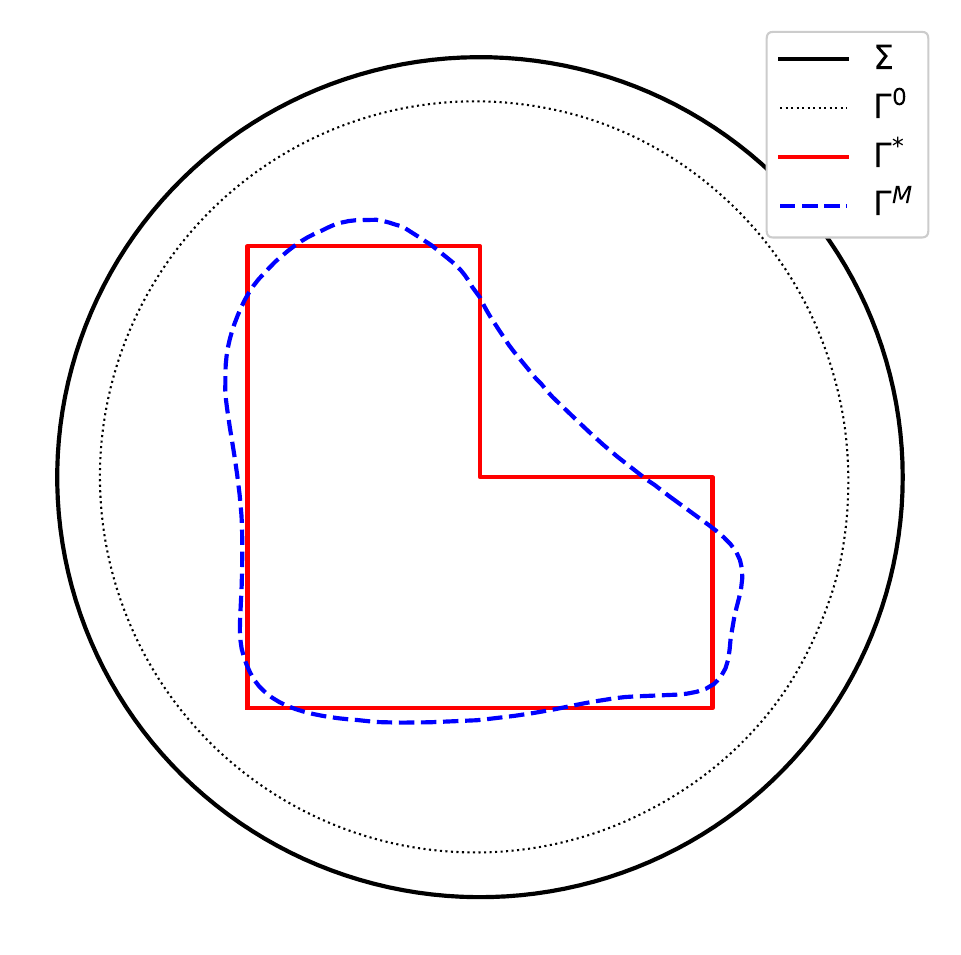}}
\resizebox{0.235\linewidth}{!}{\includegraphics{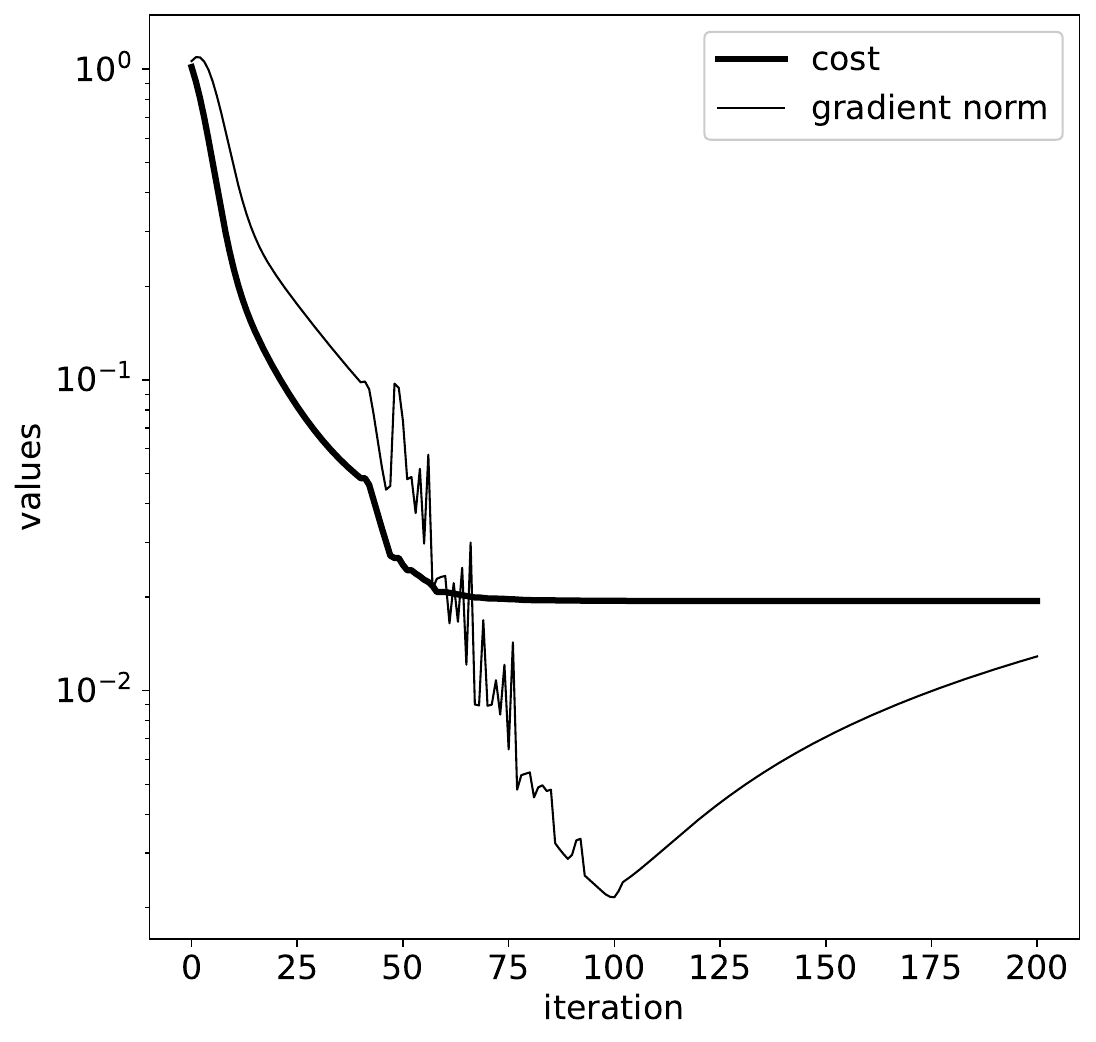}} \quad
\resizebox{0.235\linewidth}{!}{\includegraphics{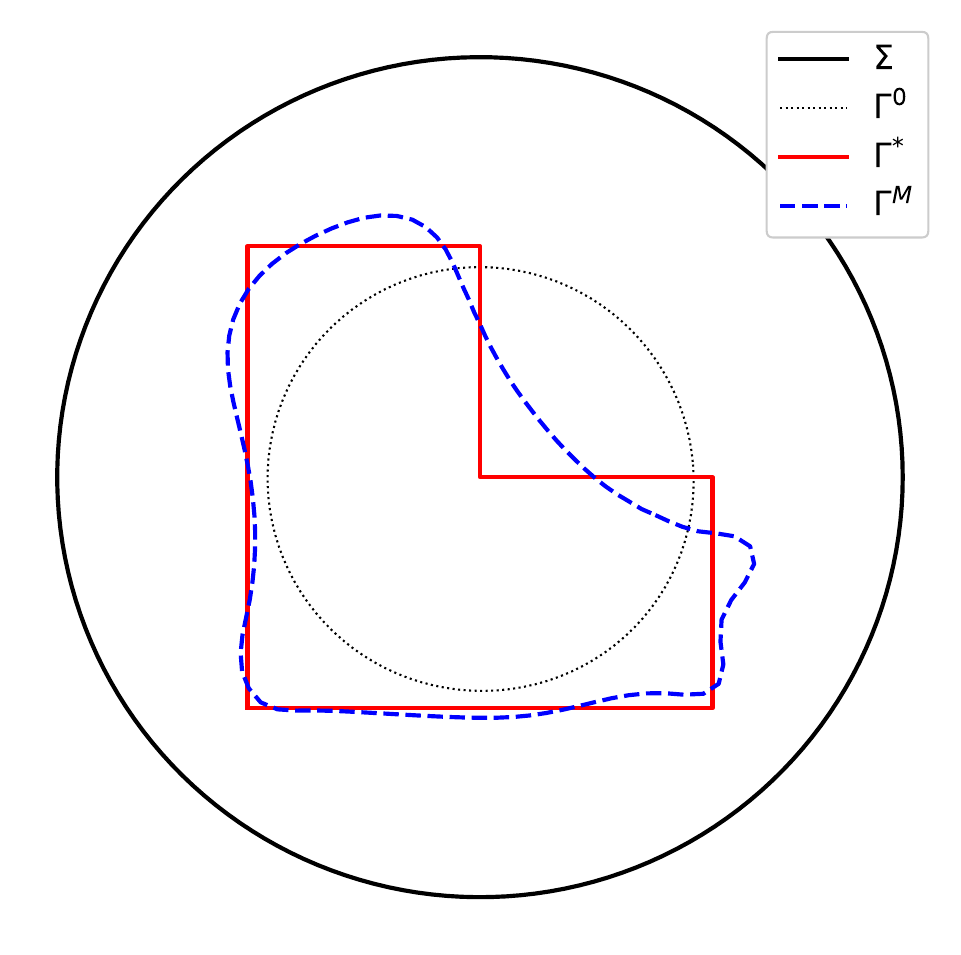}}
\resizebox{0.235\linewidth}{!}{\includegraphics{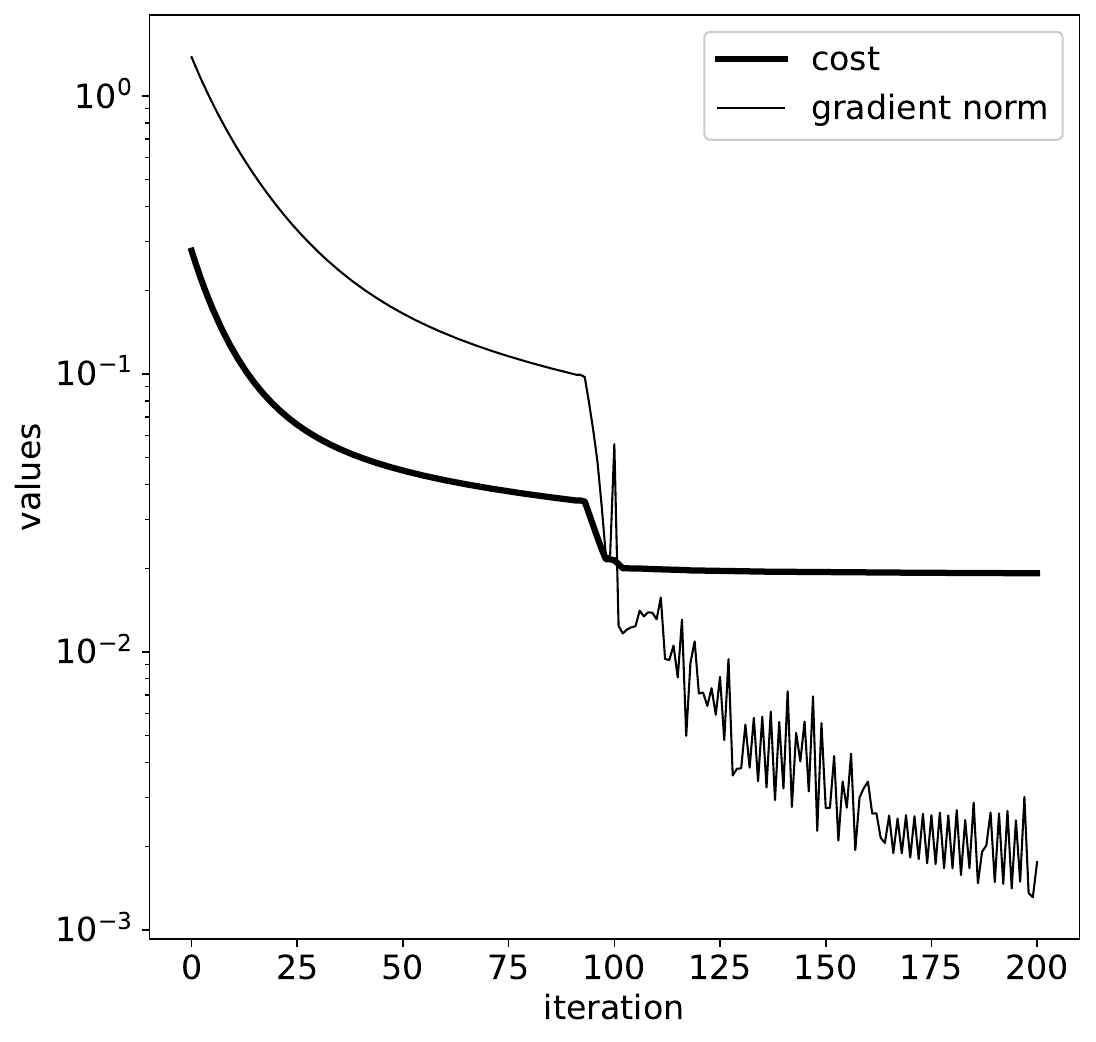}}
\caption{Reconstruction of $\Gamma^{\ast}_{2}$ with noisy data ($\delta = 30\%$) and $\Gamma^{0} = B(0,0.9)$ using ADMM without regularization and without adaptive mesh refinement.}
\label{fig:figure4}
\end{figure}


%
%
%
%
For the non-constant case of $\sigma$, the results are illustrated in Figure~\ref{addfig:figure3}. When compared with Figures~\ref{addfig:figure1} and \ref{addfig:figure2}, which depict results from the conventional shape optimization method, the improvement in this case appears to be minimal.
\ferj{We suspect that the choice of cost function is one of the main reasons behind this issue. 
A more suitable Lagrangian functional, denoted by $L_{\beta}$, should be constructed to explicitly incorporate the diffusion coefficient into the cost function. 
To address this problem, it would be interesting to explore how a Kohn--Vogelius-type cost functional \cite{KohnVogelius1984} performs in this context. 
Incorporating the diffusion coefficient into the Lagrangian corresponding to the ADMM formulation could potentially improve the reconstruction.}
\begin{figure}[htp!]
\centering 
\resizebox{0.3\linewidth}{!}{\includegraphics{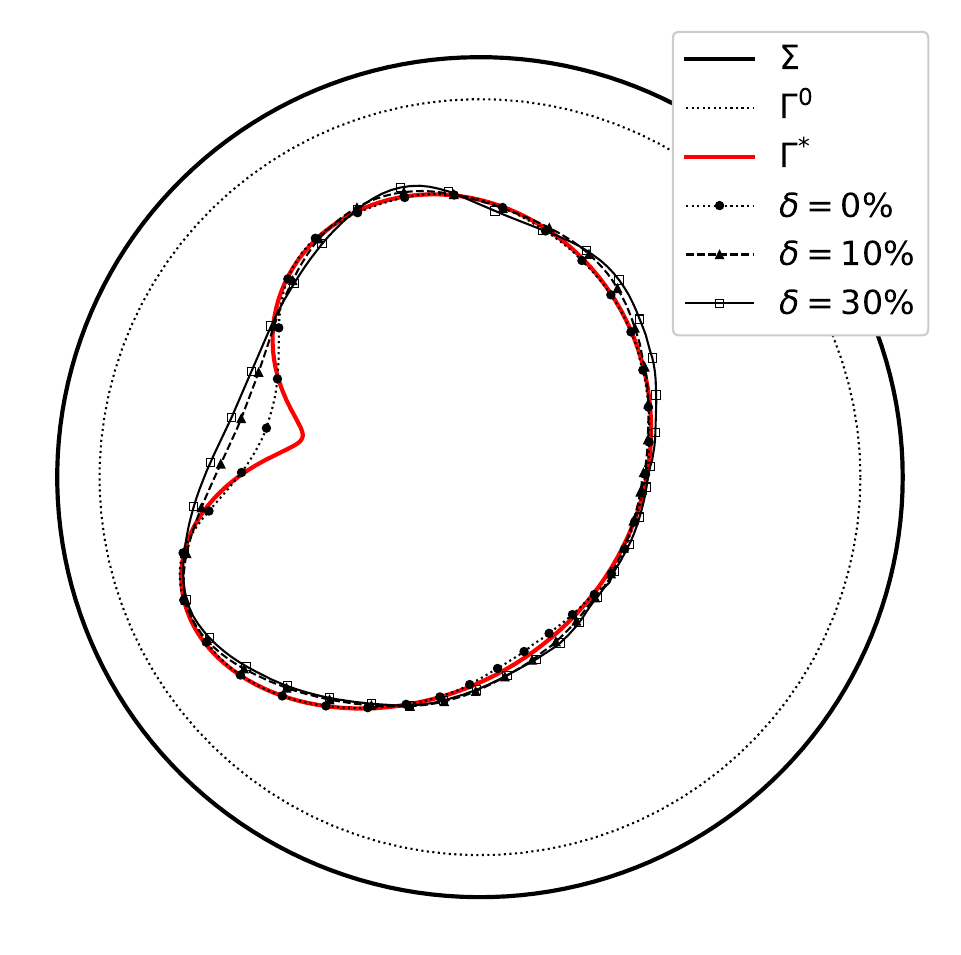}} \quad
\resizebox{0.3\linewidth}{!}{\includegraphics{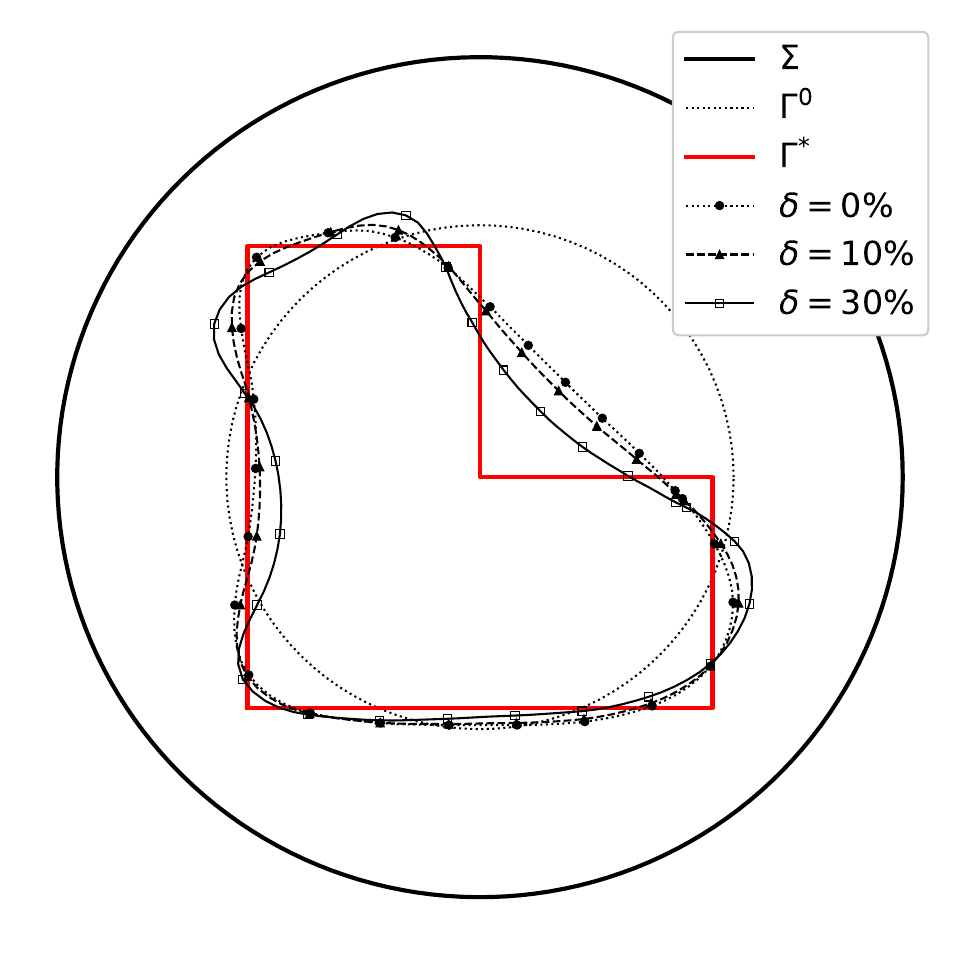}}
\caption{Reconstruction of $\Gamma^{\ast}_{1}$ (left, also note the difference in initial guesses) and $\Gamma^{\ast}_{2}$ (right) via ADMM both with perimeter regularization and without adaptive mesh refinement.}
\label{addfig:figure3}
\end{figure}


\subsection{Numerical experiments in 3D via ADMM with space-dependent diffusion coefficient} \label{subsec:ADMM_3D_numerics}
Let us now examine test cases in three spatial dimensions to further evaluate our algorithm. 
The domain $D$  is the unit sphere centered at the origin, $\sigma(x) = 1.1$ and $\bb(x) = (1.0 + 0.5 \sin{{(\arctan(x_{2}/x_{1}))}},  1.0 + 0.5 \cos{{(\arctan(x_{2}/x_{1}))}}, 1.5)^{\top}$, $x=(x_1, x_{2}, x_{3}) \in D \subset \mathbb{R}^{3}$.
Again, the data is synthetically constructed and we set $g(x) = \exp(x_{1}^{2} + x_{2}^{2})$, $x=(x_1, x_{2}, x_{3}) \in \partial{D}$. 
On the other hand, the computational setup remains largely the same as in the 2D case, with only a few adjustments. 
Specifically, we set $N = 600$, $\lambda^{0} = 0.001$, $a = 0.5 \min u(\Omega \setminus \overline{\omega}^{\ast})$, $b = 1.5 \max u(\Omega \setminus \overline{\omega}^{\ast})$, $v^{0} = 1$, $\varepsilon = 10^{-6}$, and $\omega^{0} = B(\mathbf{0}, 0.8)$.

For the exact obstacle, we analyze two shapes: a dumbbell shape and a star shape, setting $\beta = 0.1$ in the ADMM scheme. 
In the forward problem, the exact domain is discretized with minimum and maximum mesh sizes $h_{\min}^{\ast} = 0.05$ and $h_{\max}^{\ast} = 0.1$ (see, e.g., the first row of plots in Figure~\ref{fig:figure6}), using tetrahedrons with a maximum volume of $0.001$. 
For the inversion process, the domain $(\Omega \setminus \overline{\omega})^{0}$ is discretized with a coarse mesh, having $h_{\min} = 0.15$ and $h_{\max} = 0.2$, using tetrahedrons with a volume of $0.005$.
When dealing with noisy data, we combine the ADMM and \ferj{the conventional shape optimization (SO)} schemes with perimeter regularization (cf. \cite{RabagoHadriAfraitesHendyZaky2024}), using a small parameter (equal to $0.003$), to reduce excessive irregularities on the surfaces of the obstacles; see \cite{RabagoAzegami2018}.

The figures, from Figure~\ref{fig:figure6} to Figure~\ref{fig:figure9}, show the numerical results including both exact measurements and data affected by noise.
The key observations align with those from the 2D experiments. 
Notably, the ADMM results are more accurate than those from SO, as ADMM effectively reconstructs the exact obstacle even with high noise levels in the data. 
Additionally, it is observed that smaller obstacles are harder to reconstruct accurately, as shown in Figure~\ref{fig:figure8} and Figure~\ref{fig:figure9}. 
Nevertheless, ADMM detects the concavities of obstacles more effectively than SO, demonstrating a significant advantage.
To provide more insights on these methods, Figure~\ref{fig:figure10} displays the plots showing the histories of cost values and gradient norms for the test cases considered.
\fergy{From the plots, it is evident that the cost values for SO are consistently lower compared to those for ADMM. 
This is primarily due to the additional term in the objective functional of the ADMM formulation. 
Furthermore, the increase in cost values for ADMM can be attributed to the last integral term appearing in the objective functional, which is not always positive.}
We also considered cases where $\sigma$ is non-constant. 
However, as in the 2D case, the improvement is minimal.
\begin{figure}[htp!]
\centering 
\resizebox{0.235\linewidth}{!}{\includegraphics{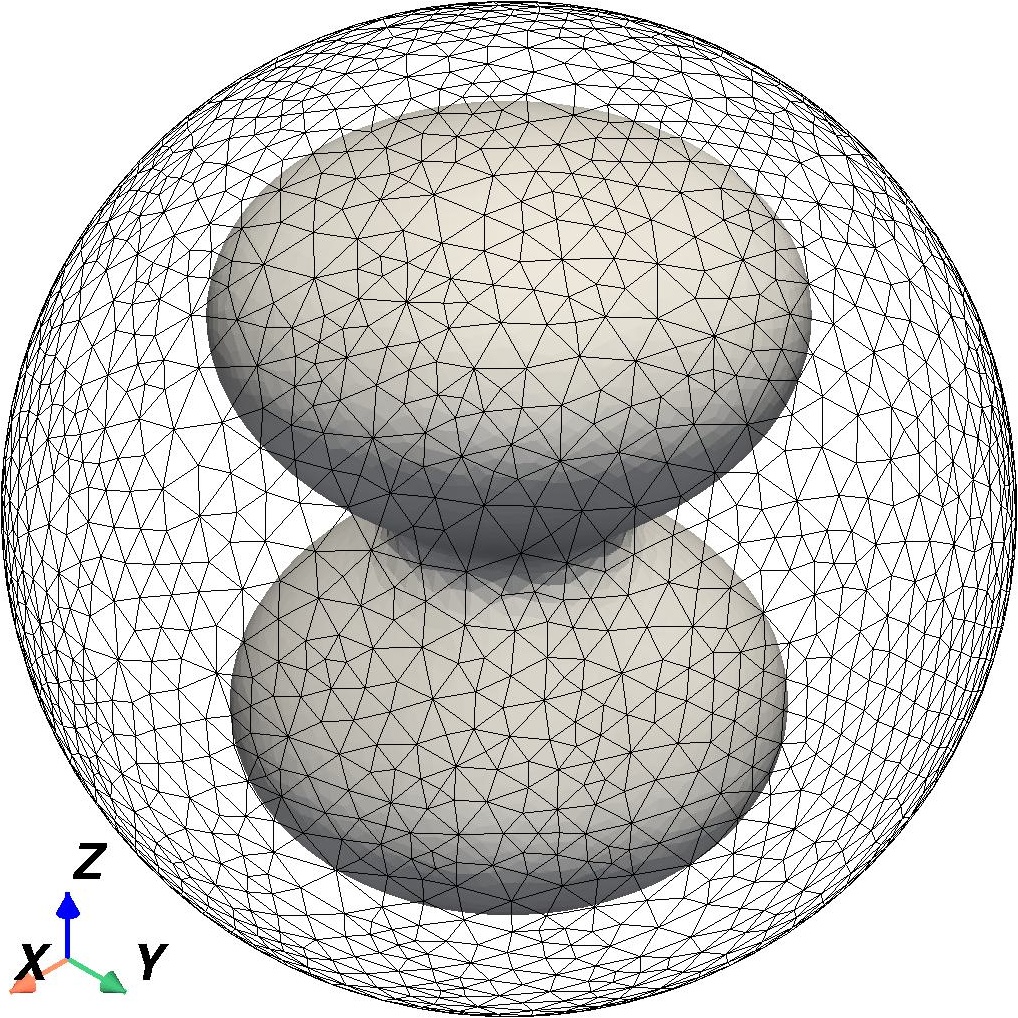}} \quad 
\resizebox{0.235\linewidth}{!}{\includegraphics{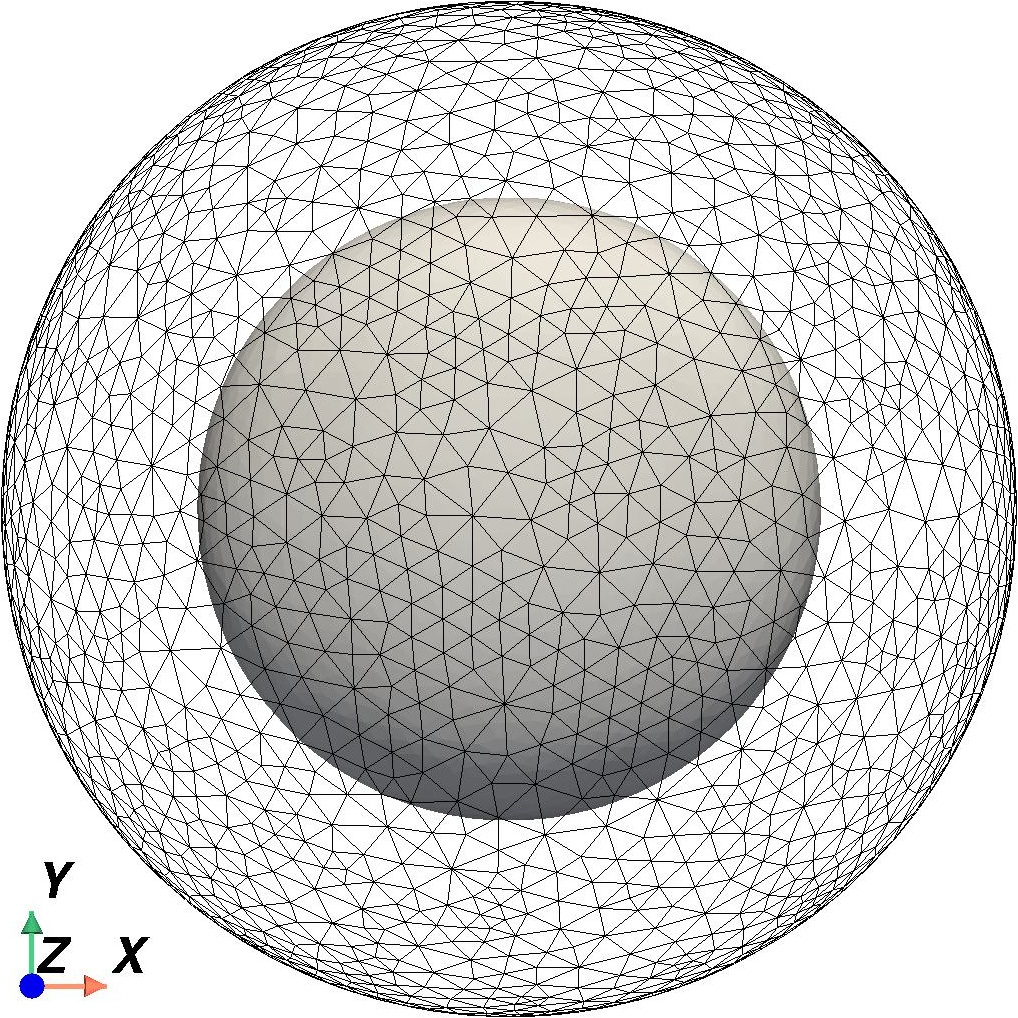}} \quad 
\resizebox{0.235\linewidth}{!}{\includegraphics{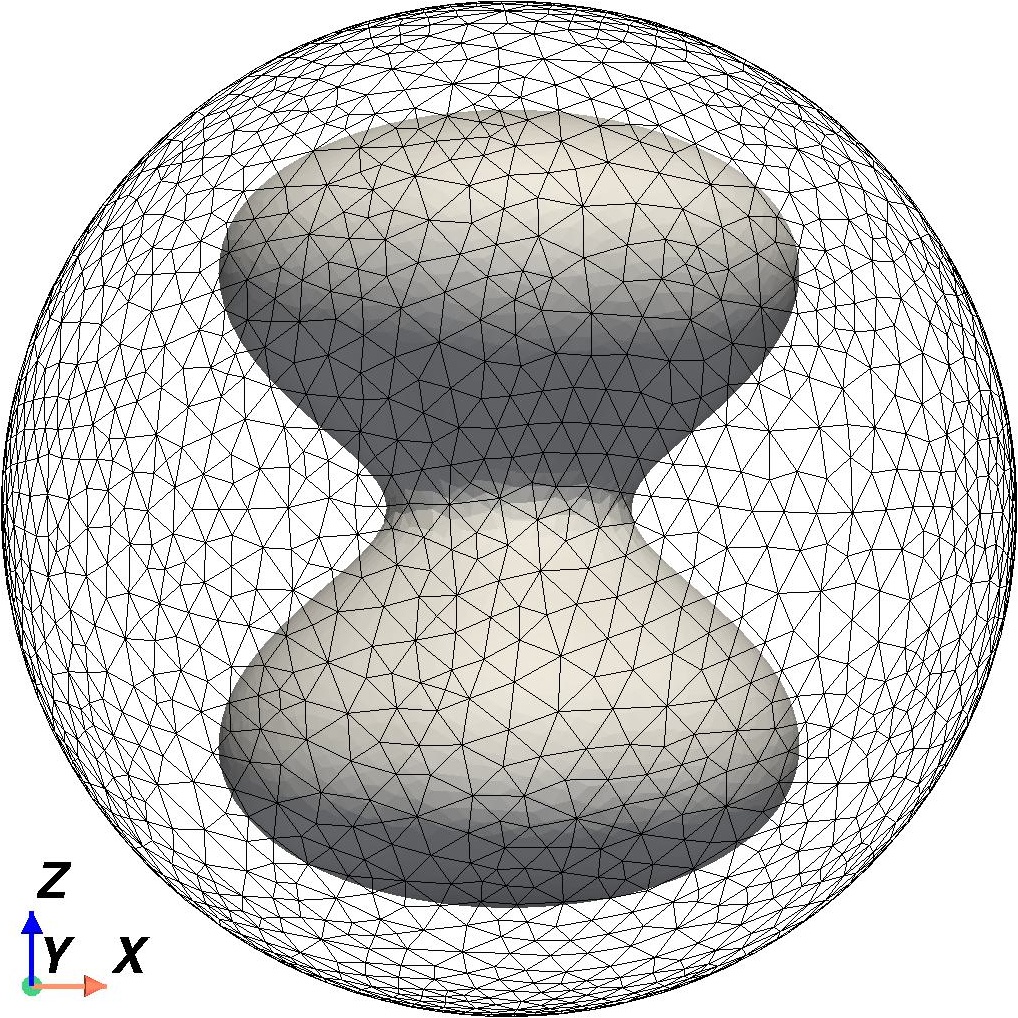}}\\[0.5em]
\resizebox{0.235\linewidth}{!}{\includegraphics{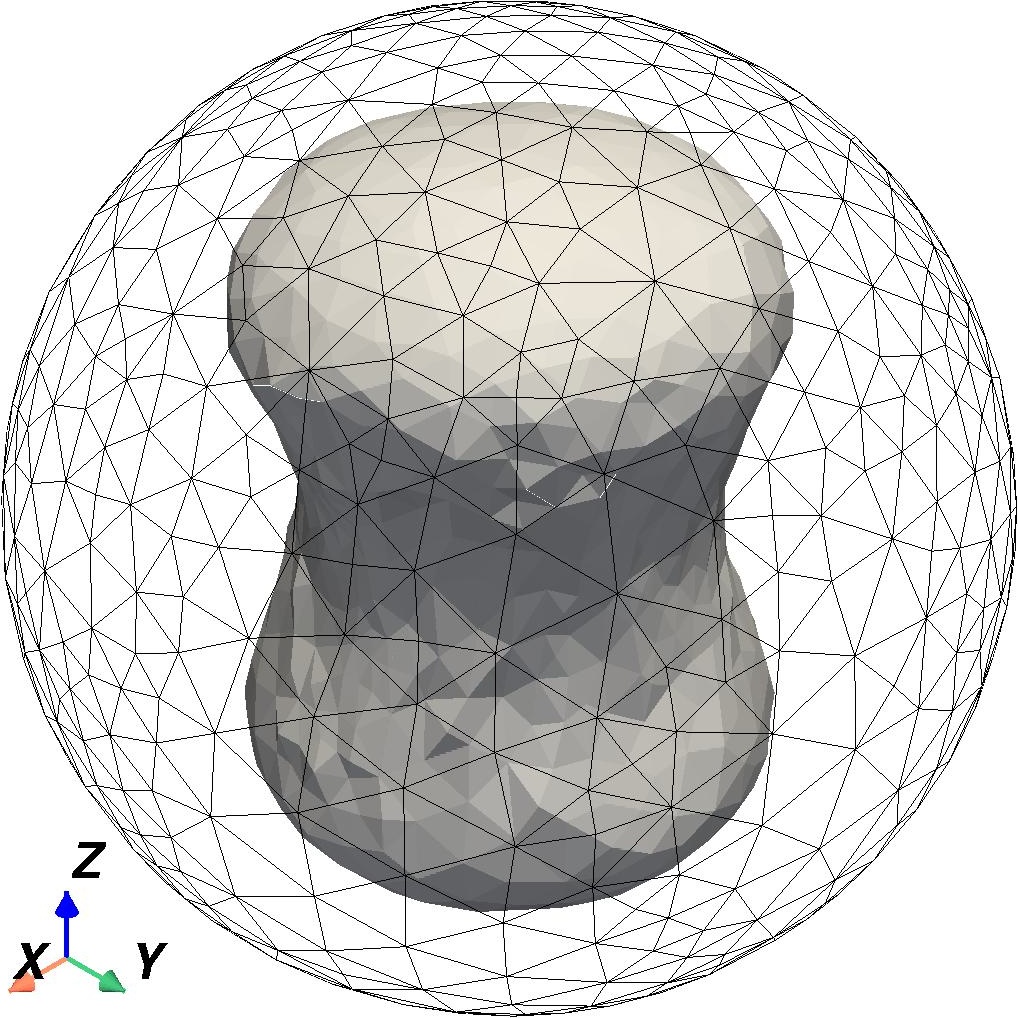}} \quad 
\resizebox{0.235\linewidth}{!}{\includegraphics{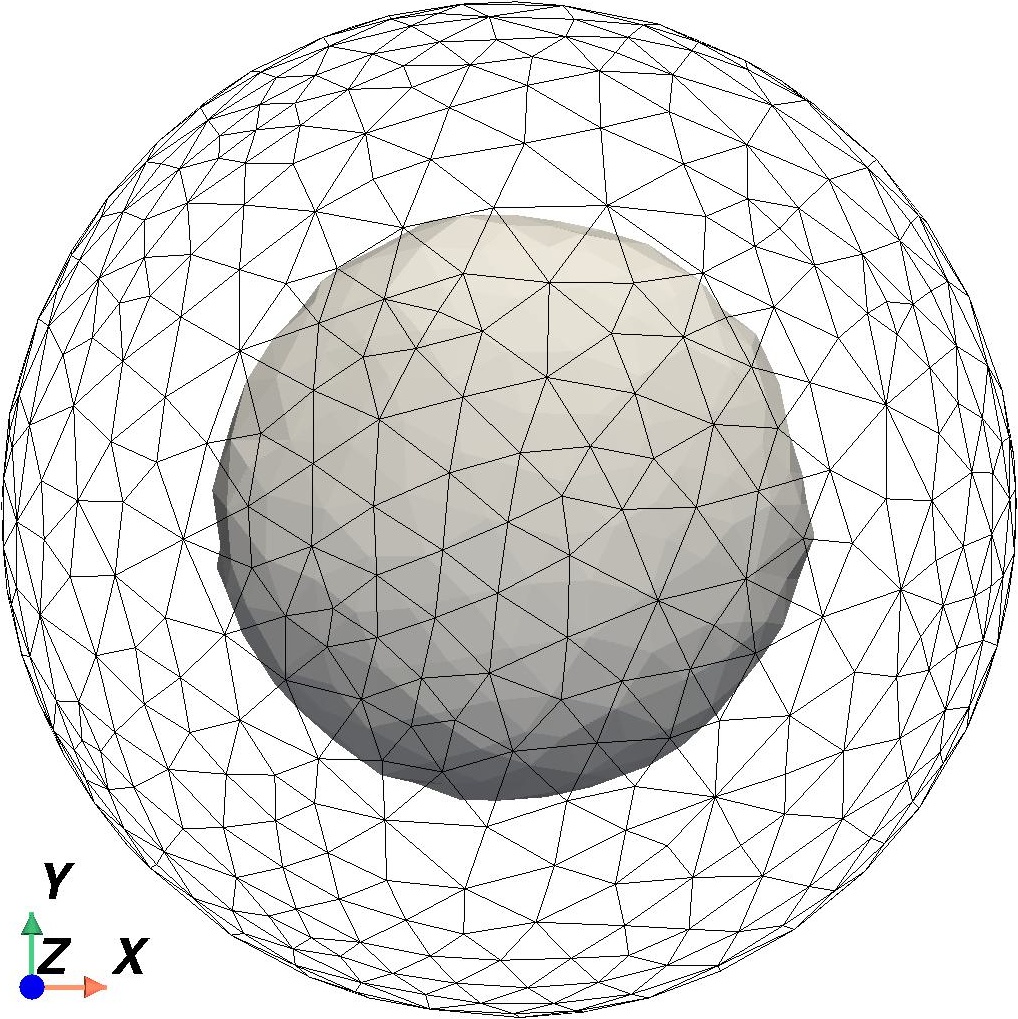}} \quad 
\resizebox{0.235\linewidth}{!}{\includegraphics{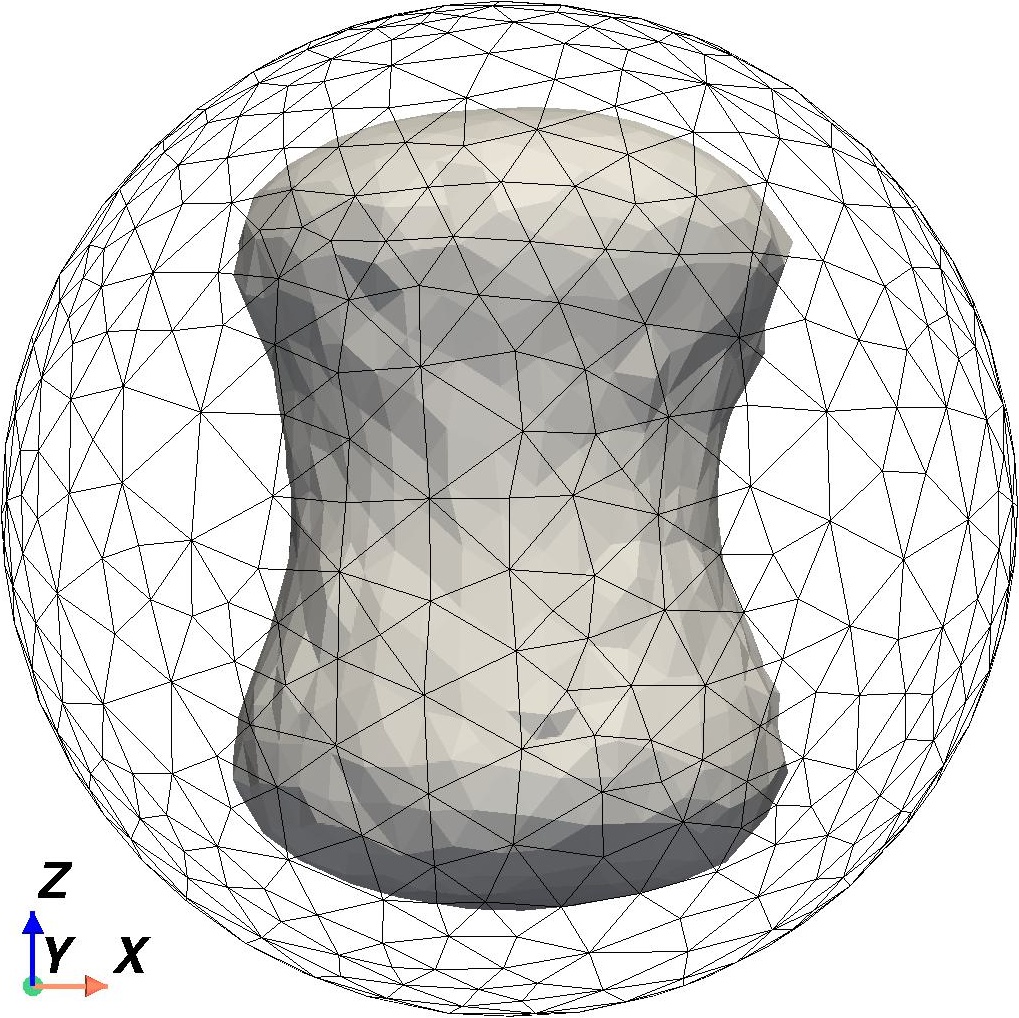}}\\[0.5em]
\resizebox{0.235\linewidth}{!}{\includegraphics{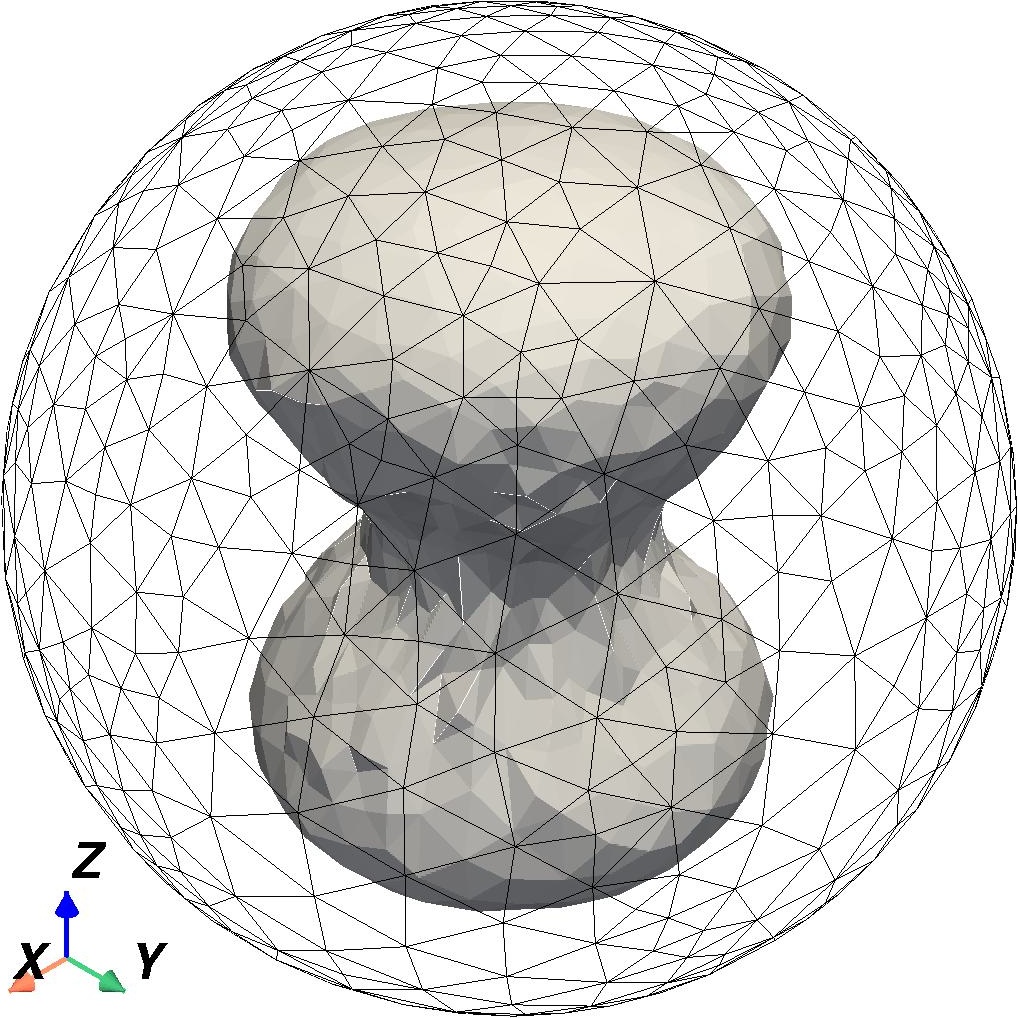}} \quad 
\resizebox{0.235\linewidth}{!}{\includegraphics{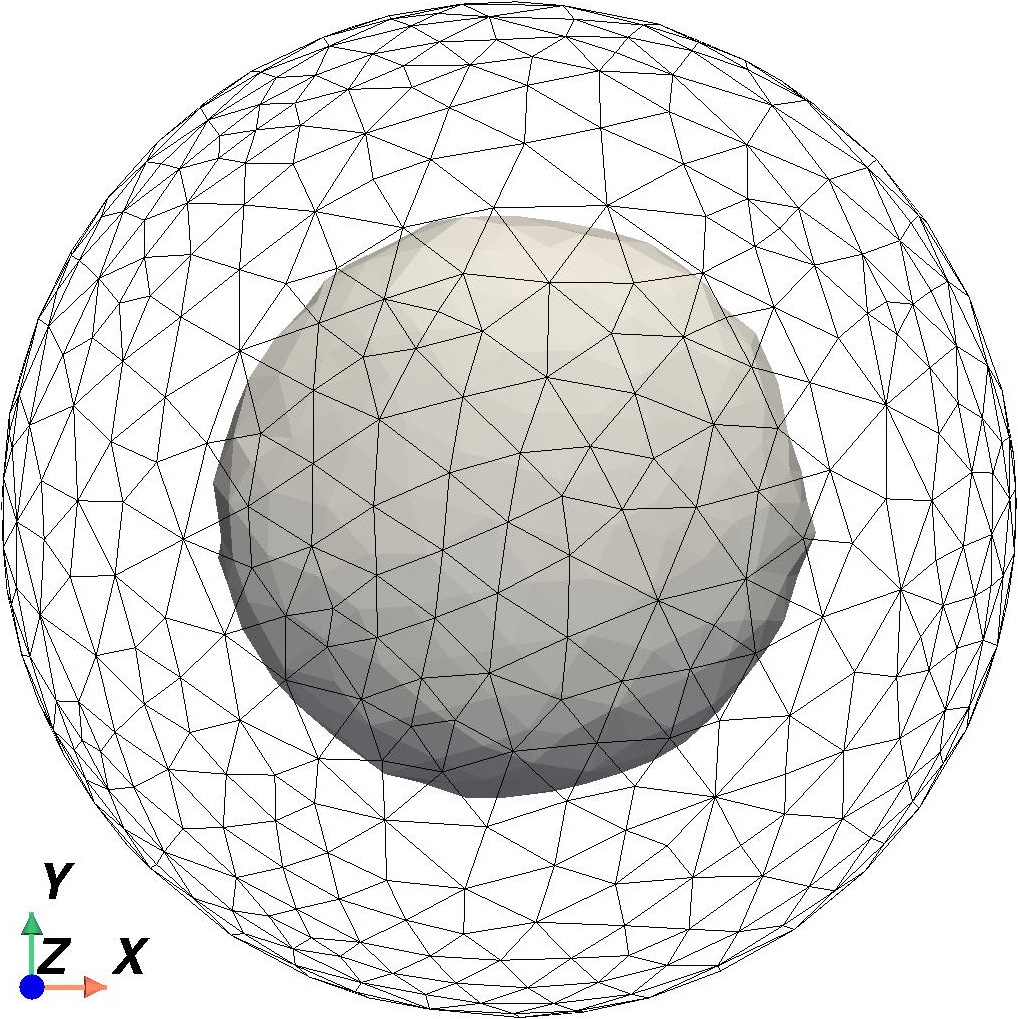}} \quad 
\resizebox{0.235\linewidth}{!}{\includegraphics{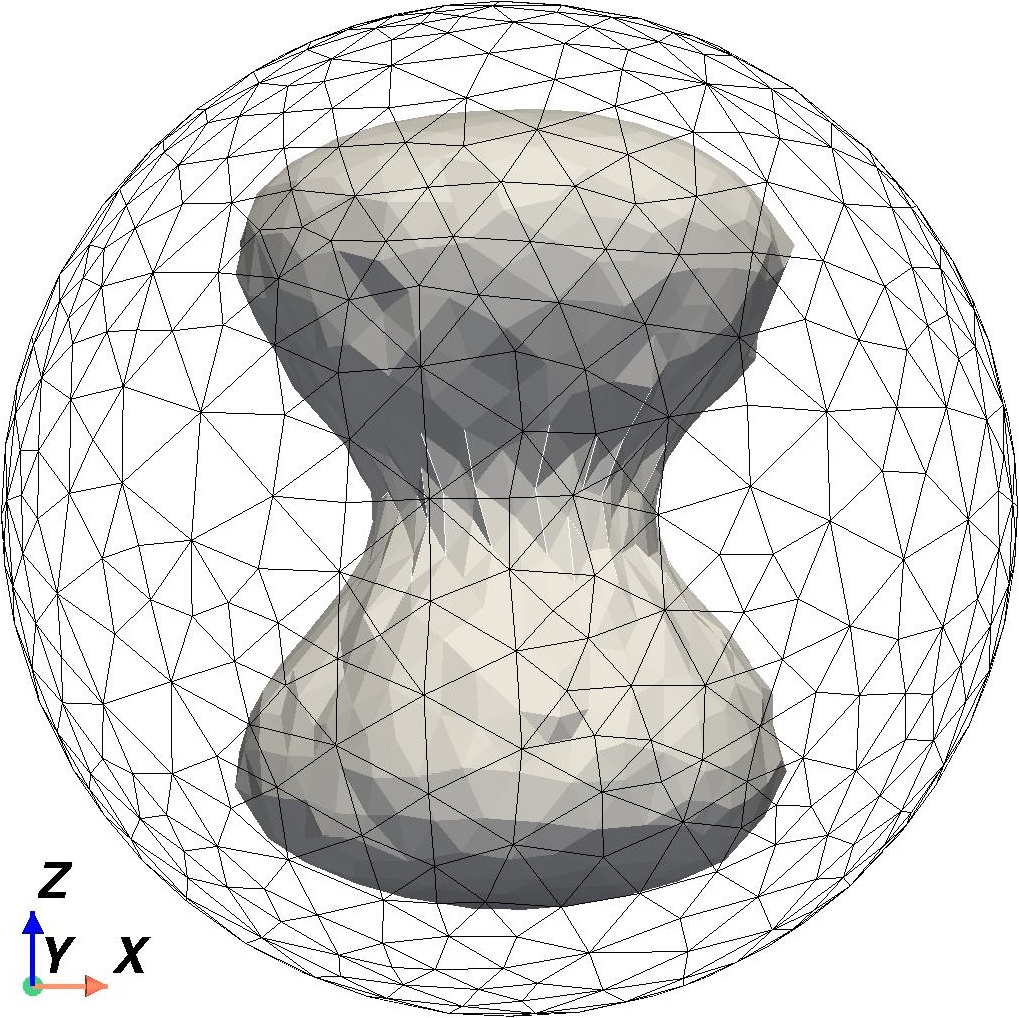}}
\caption{Exact geometry of a dumb-bell shape obstacle (top/first row) and reconstructed shapes obtained via SO (middle/second row) and ADMM (bottom/third row) with exact data.}
\label{fig:figure6}
\end{figure}

\begin{figure}[htp!]
\centering 
\resizebox{0.235\linewidth}{!}{\includegraphics{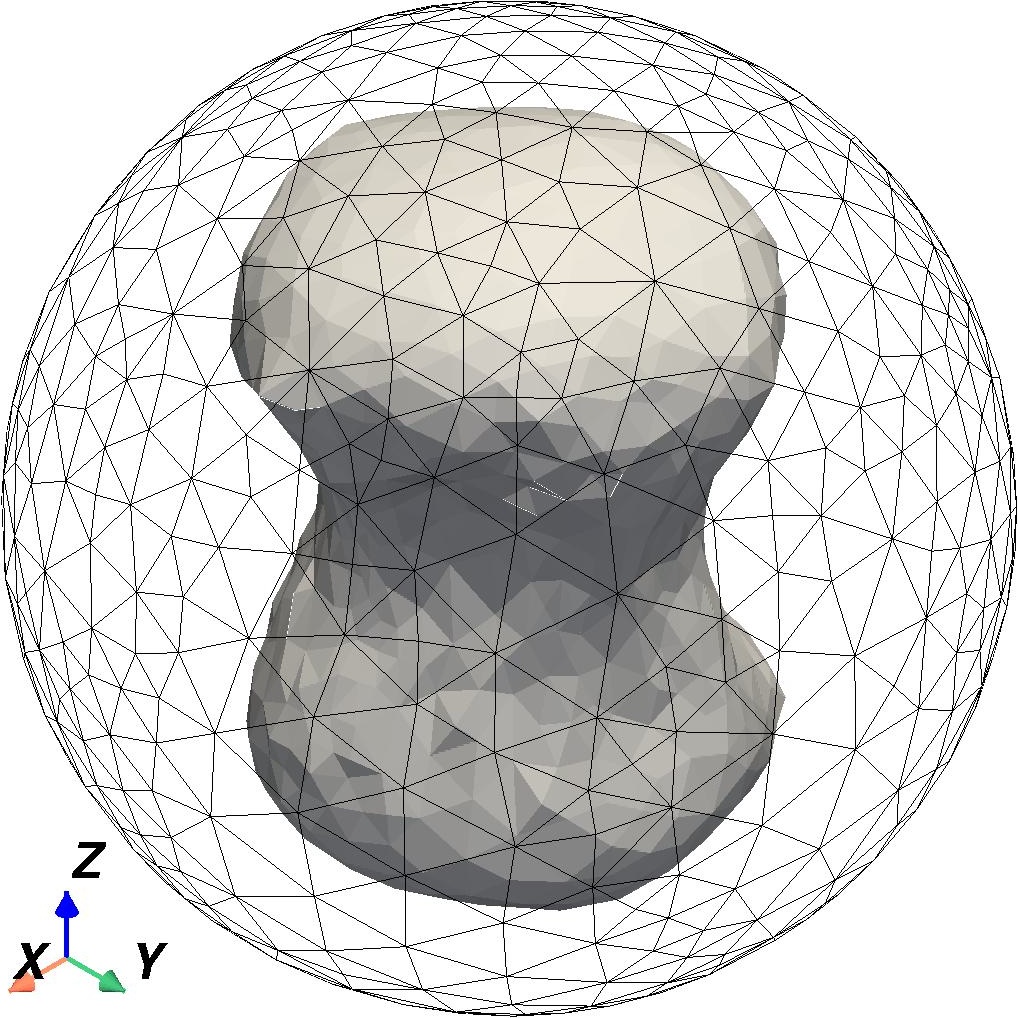}} \quad 
\resizebox{0.235\linewidth}{!}{\includegraphics{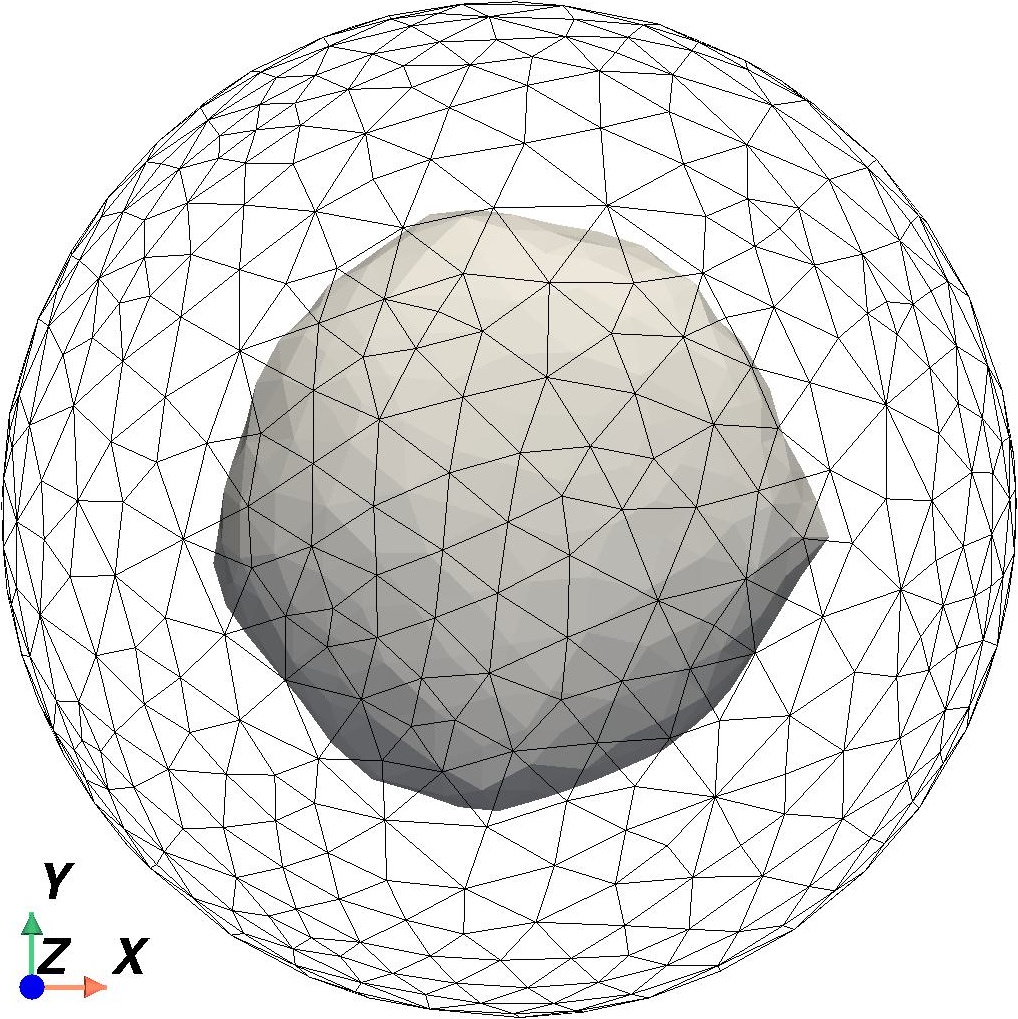}} \quad 
\resizebox{0.235\linewidth}{!}{\includegraphics{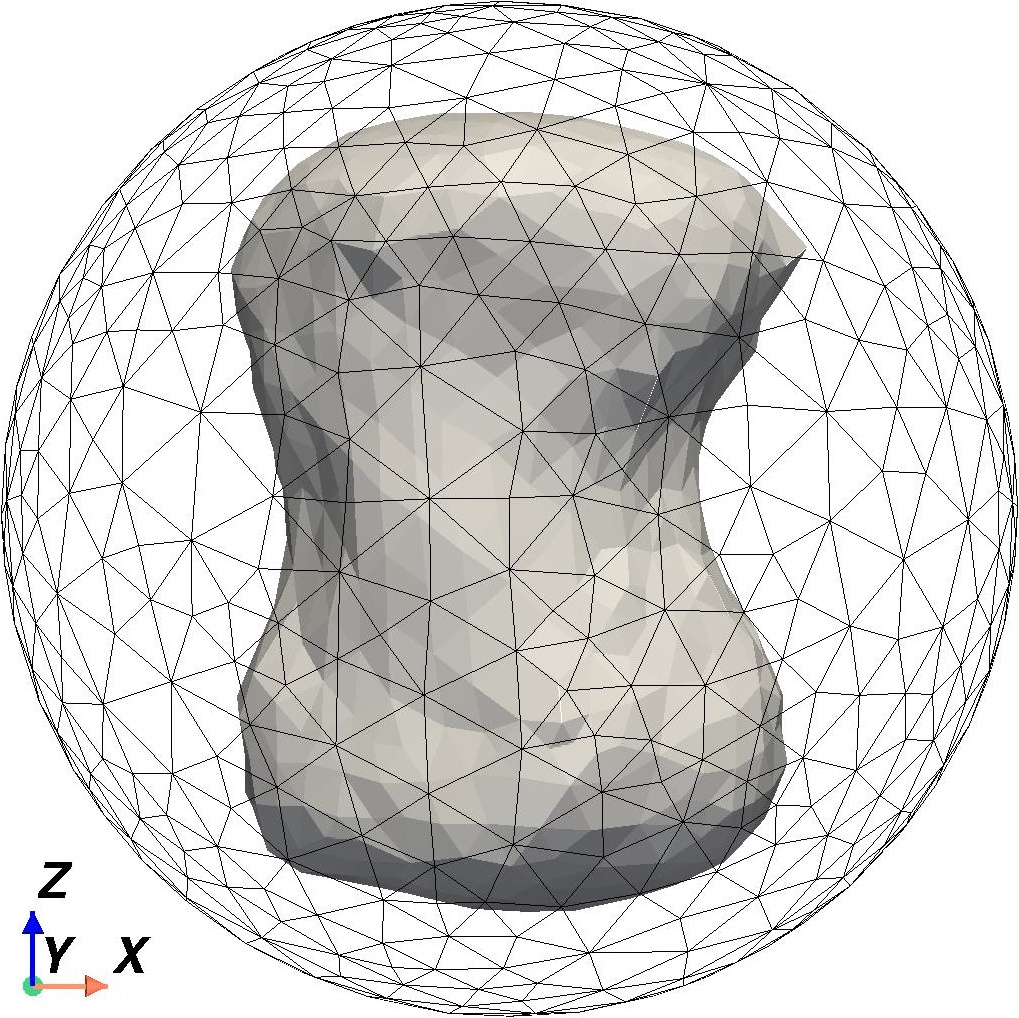}}\\[0.5em]
\resizebox{0.235\linewidth}{!}{\includegraphics{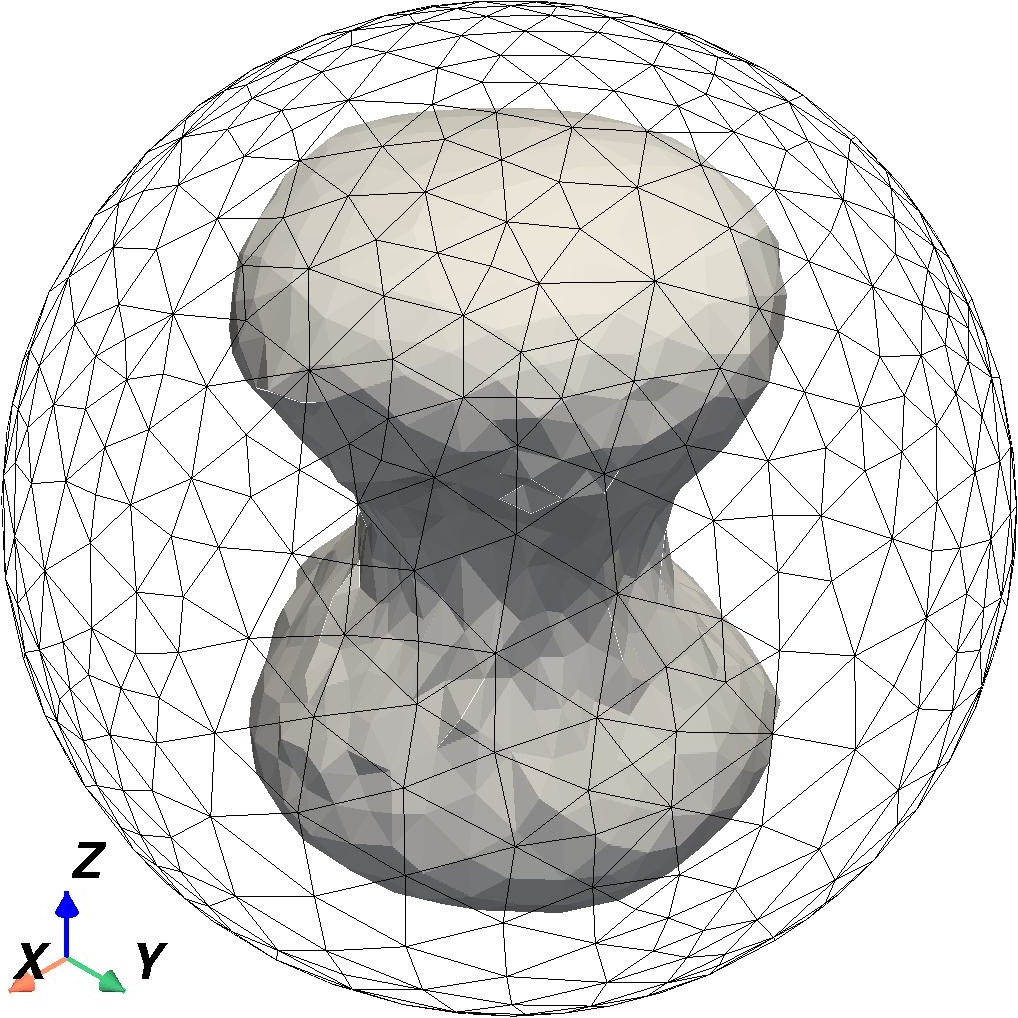}} \quad 
\resizebox{0.235\linewidth}{!}{\includegraphics{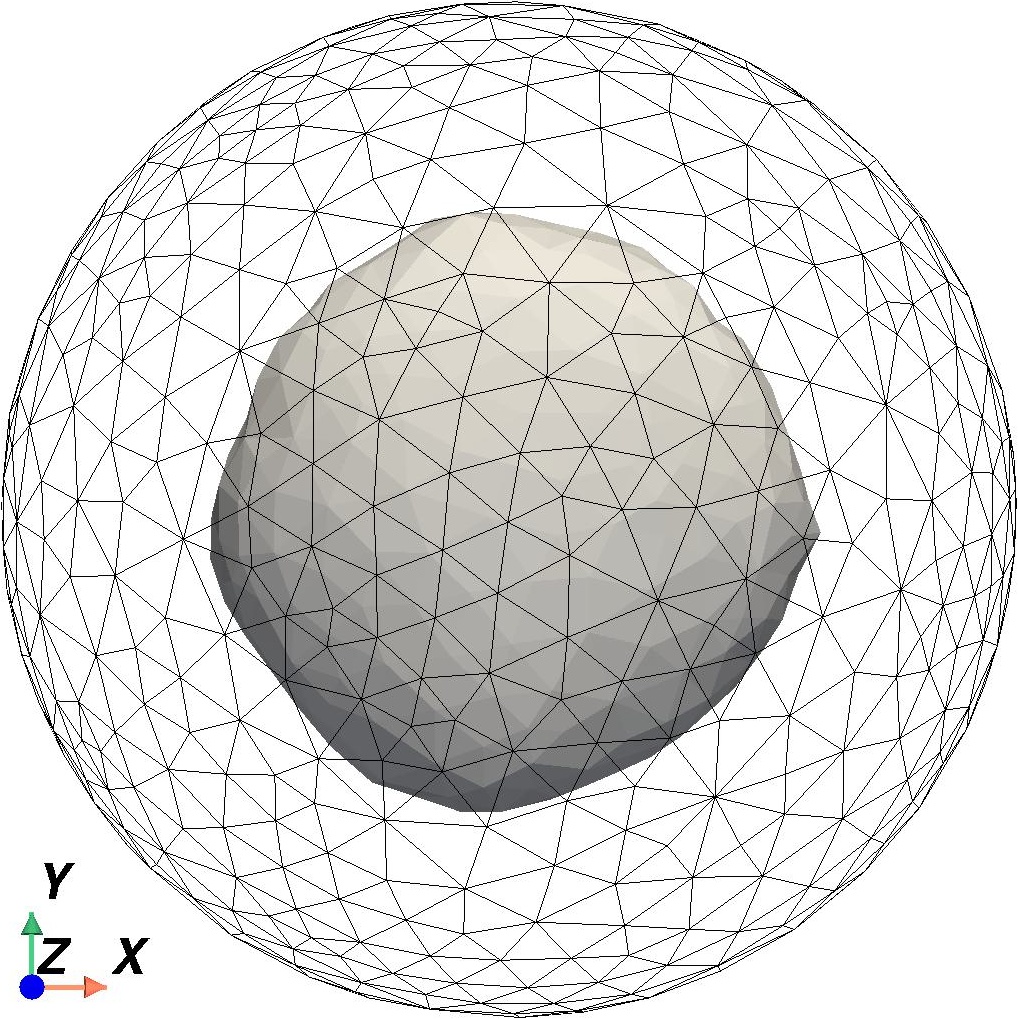}} \quad 
\resizebox{0.235\linewidth}{!}{\includegraphics{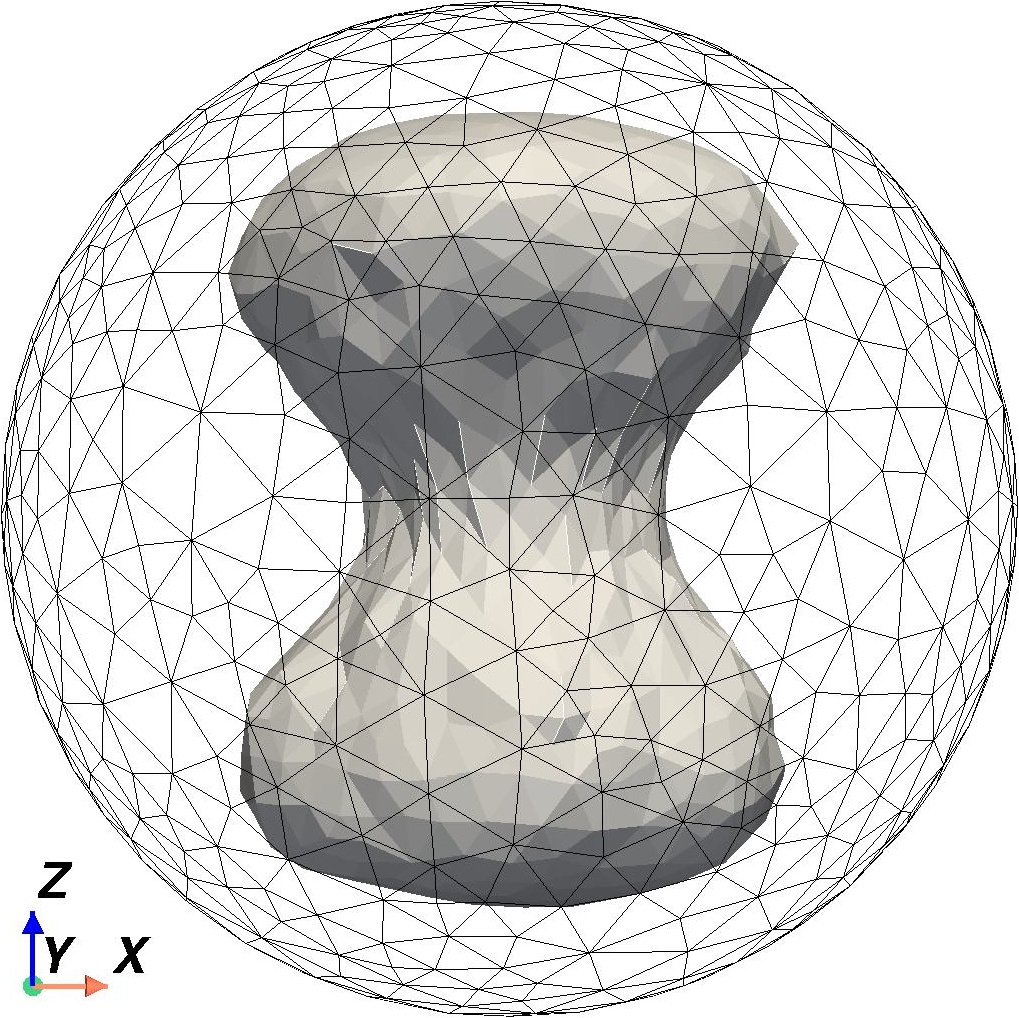}}
\caption{Reconstructed shapes obtained via SO (top row) and ADMM (bottom row) with noisy data at a $30\%$ noise level.}
\label{fig:figure7}
\end{figure}

\begin{figure}[htp!]
\centering 
\resizebox{0.235\linewidth}{!}{\includegraphics{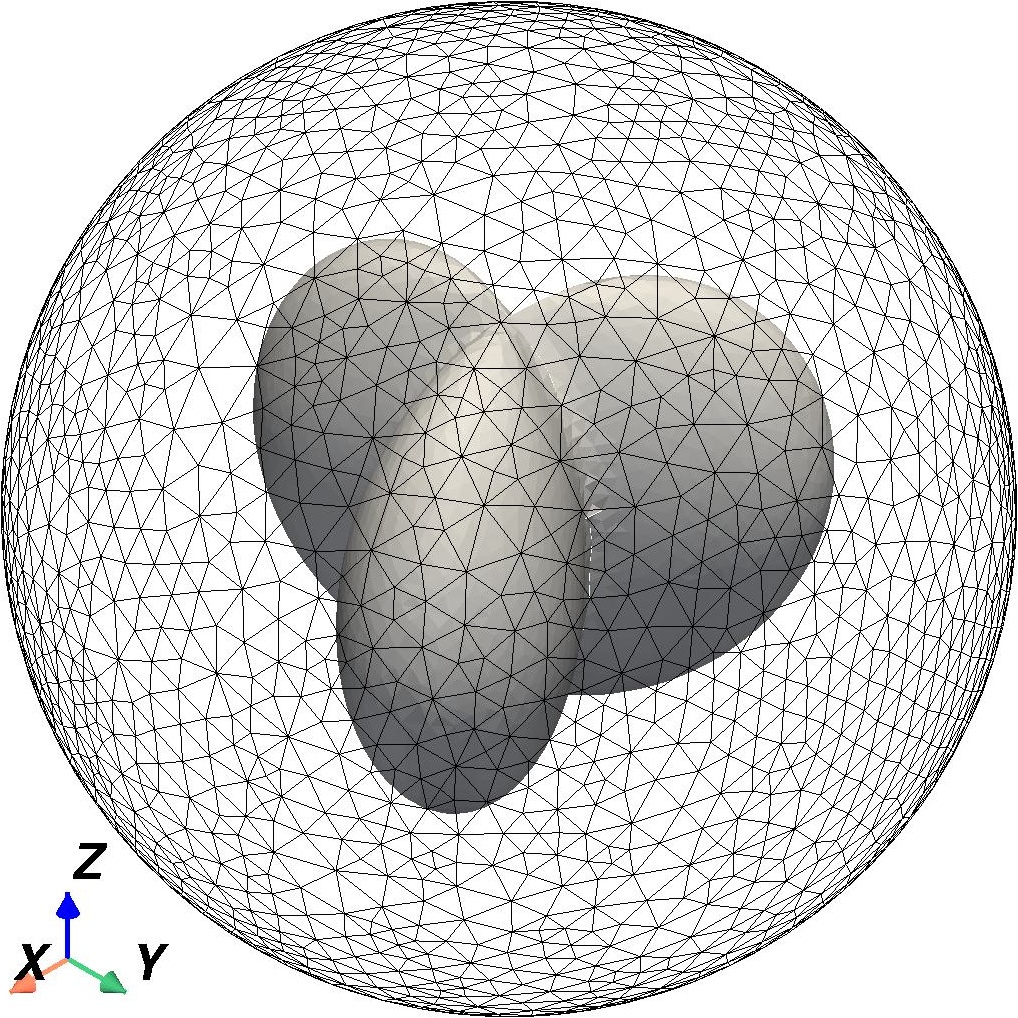}} \quad 
\resizebox{0.235\linewidth}{!}{\includegraphics{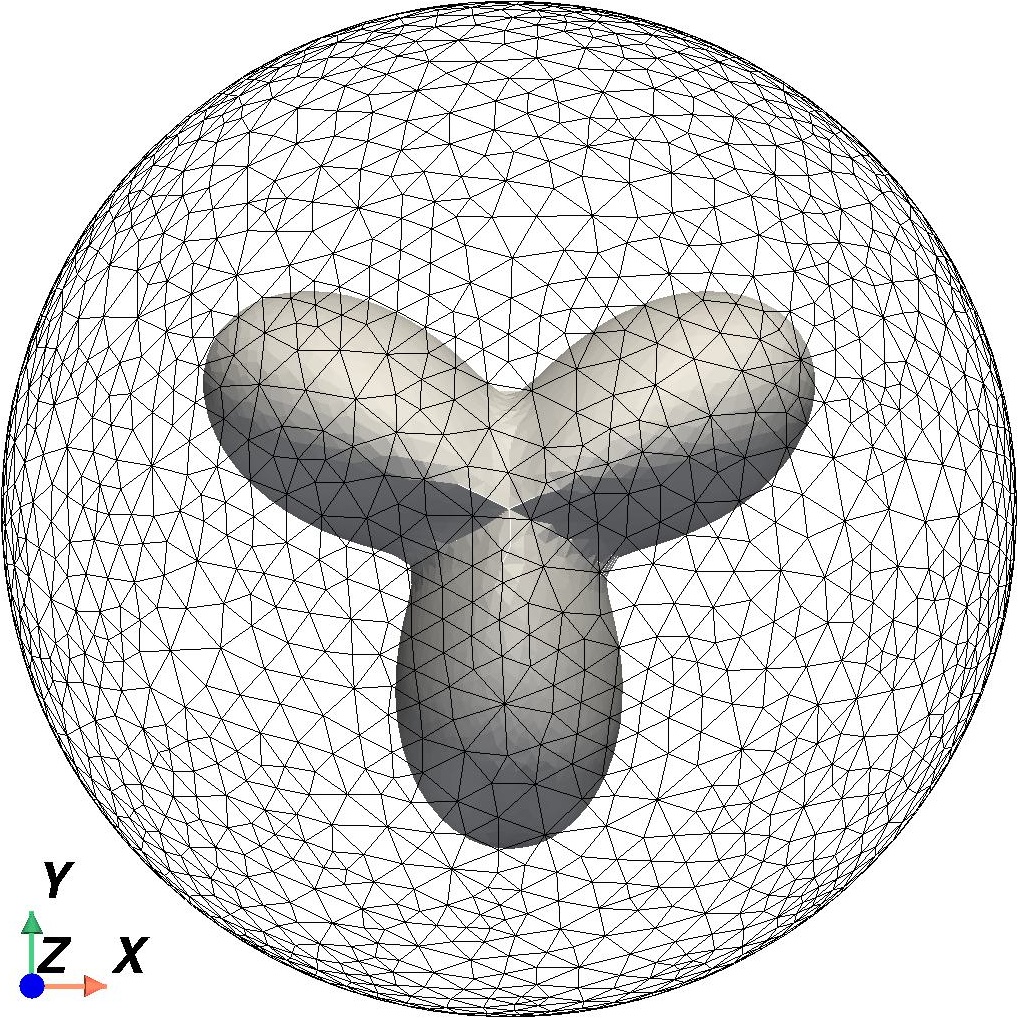}} \quad 
\resizebox{0.235\linewidth}{!}{\includegraphics{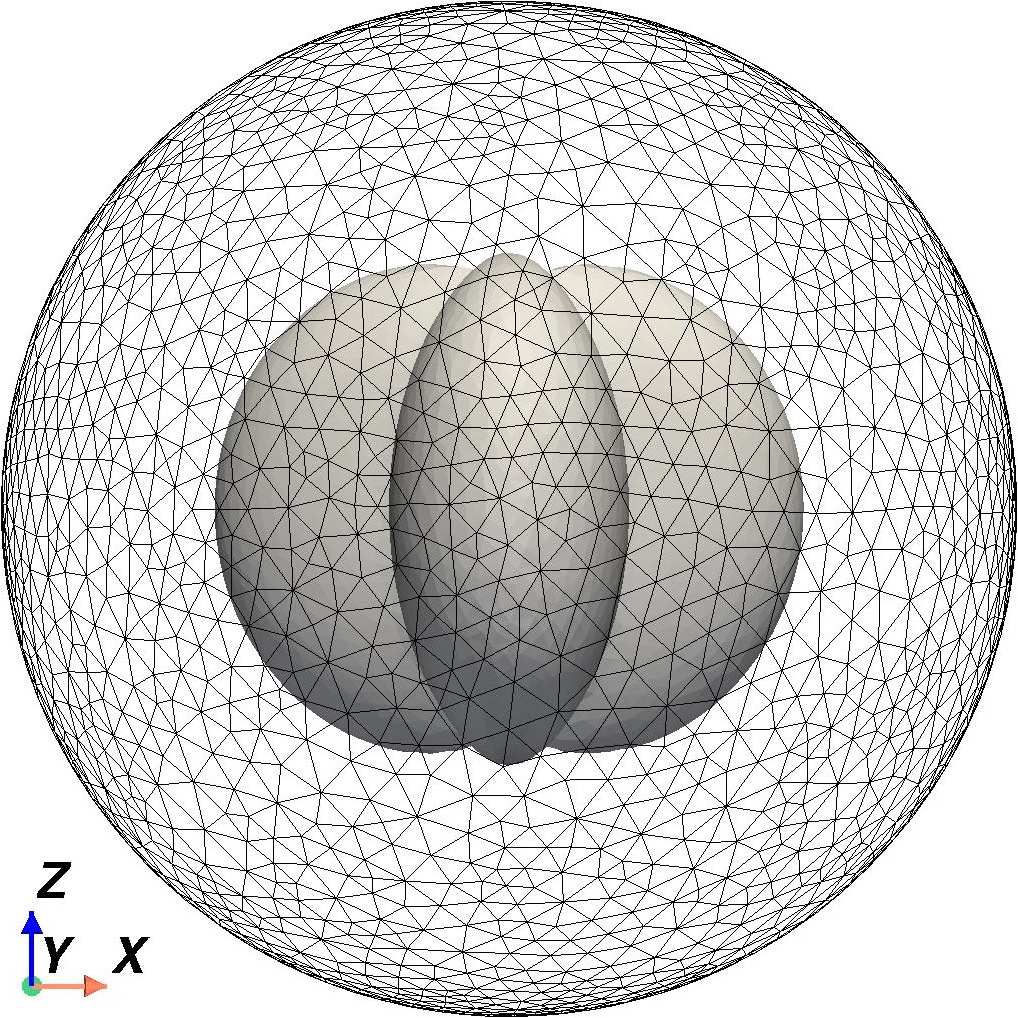}}\\[0.5em]
\resizebox{0.235\linewidth}{!}{\includegraphics{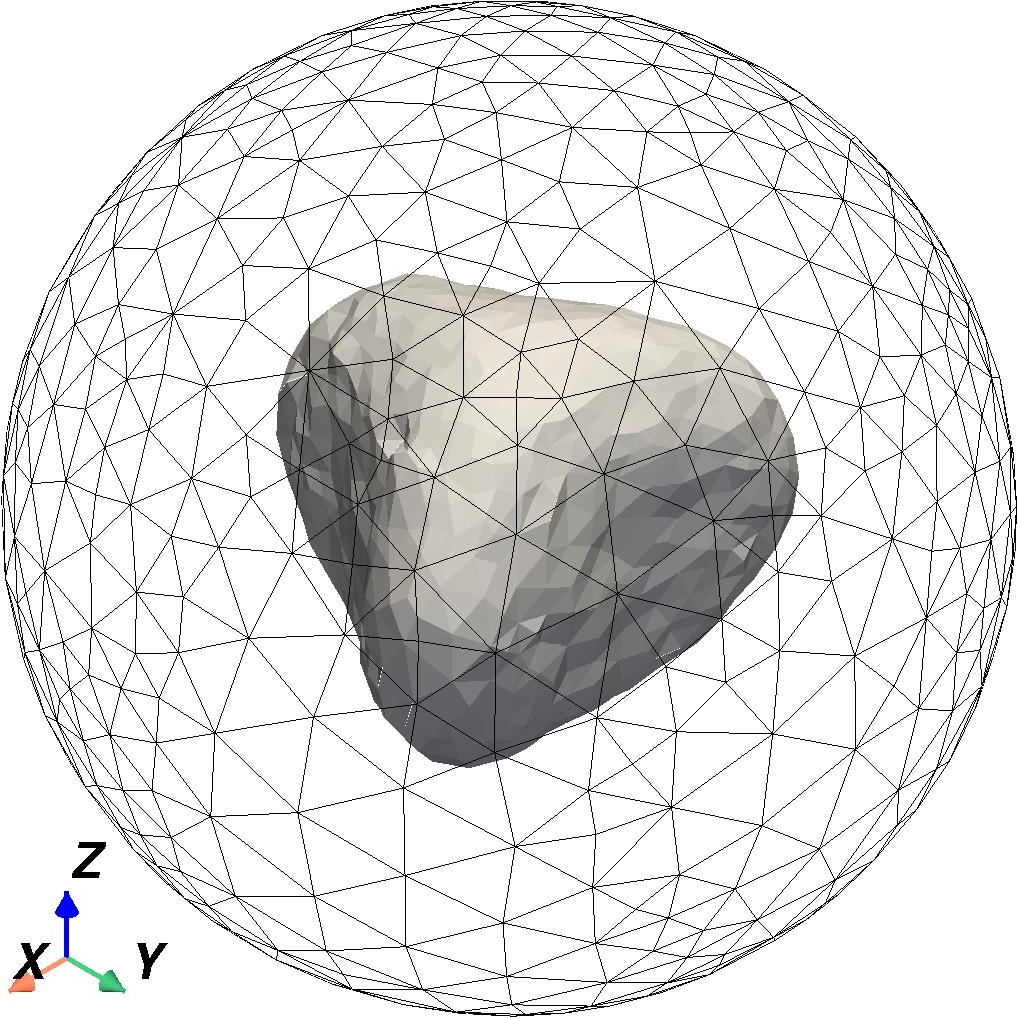}} \quad 
\resizebox{0.235\linewidth}{!}{\includegraphics{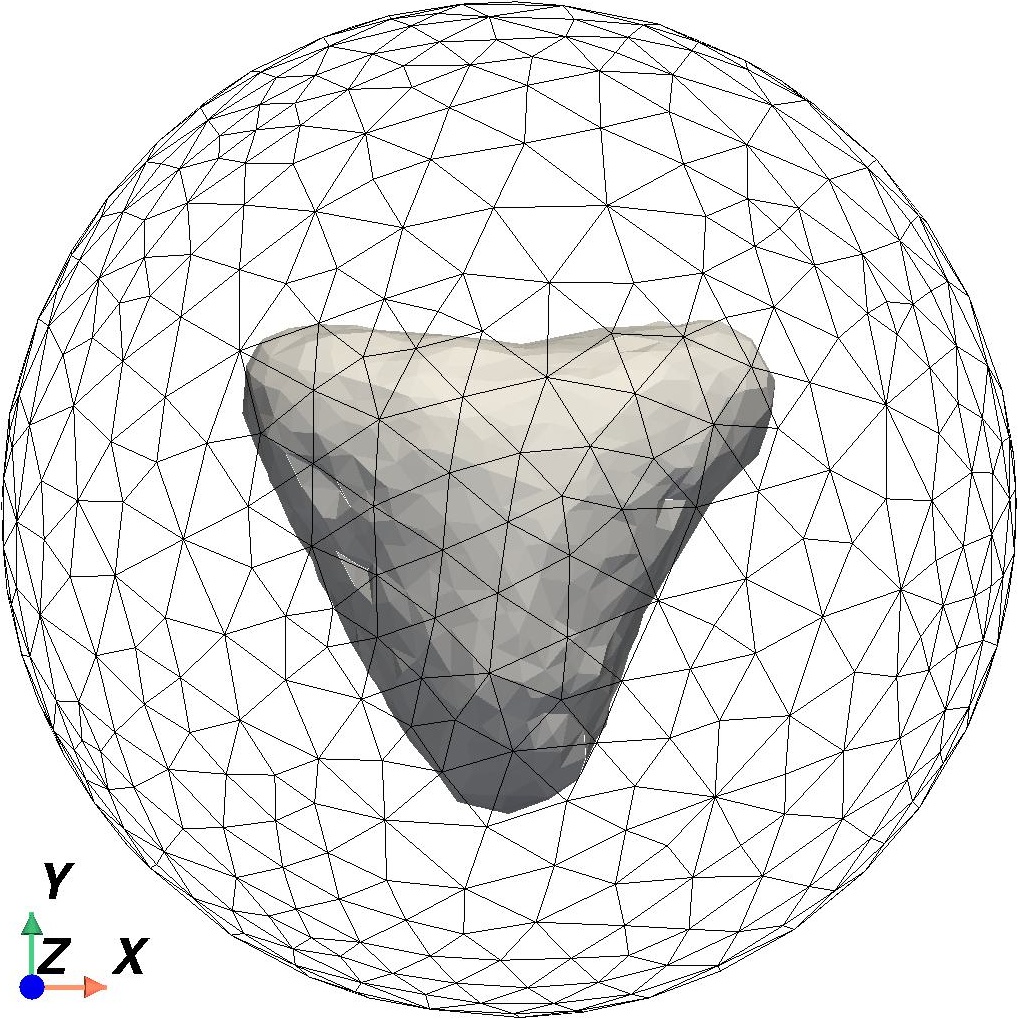}} \quad 
\resizebox{0.235\linewidth}{!}{\includegraphics{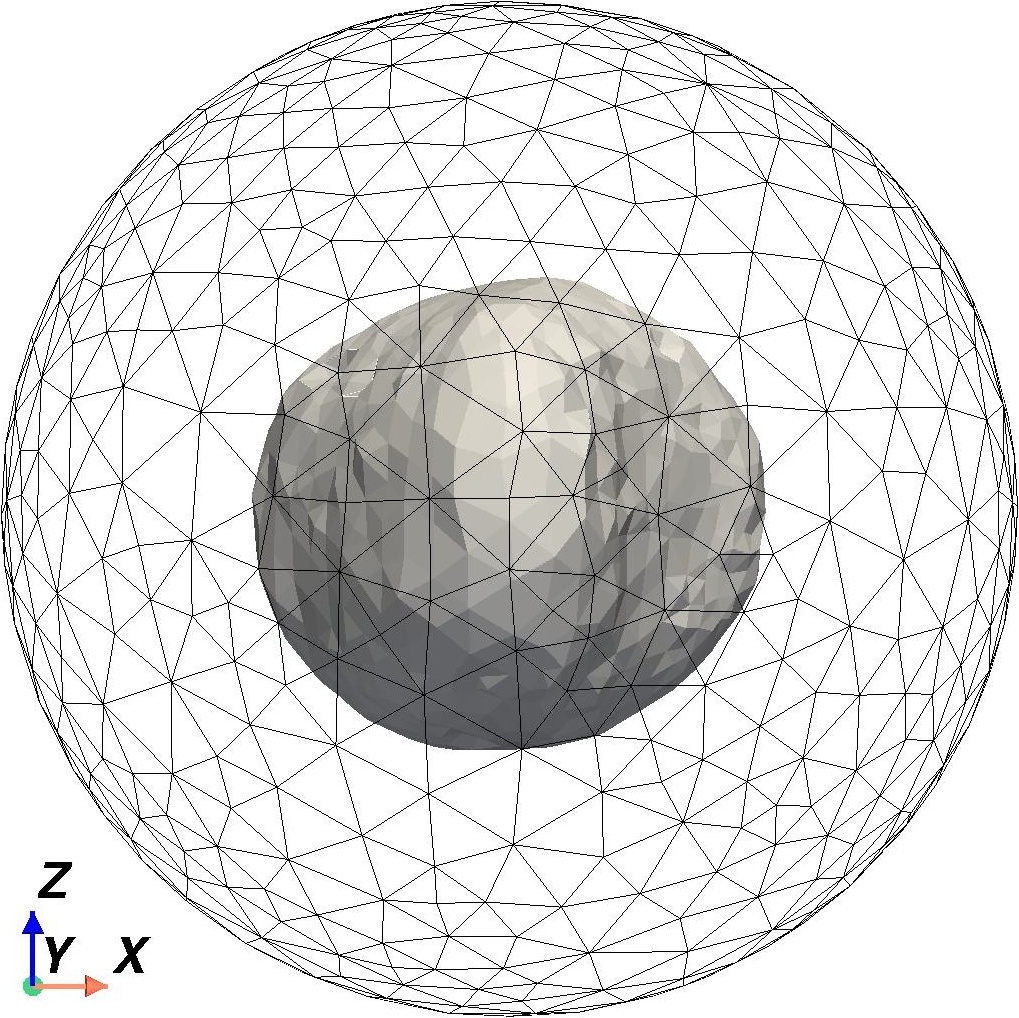}}\\[0.5em]
\resizebox{0.235\linewidth}{!}{\includegraphics{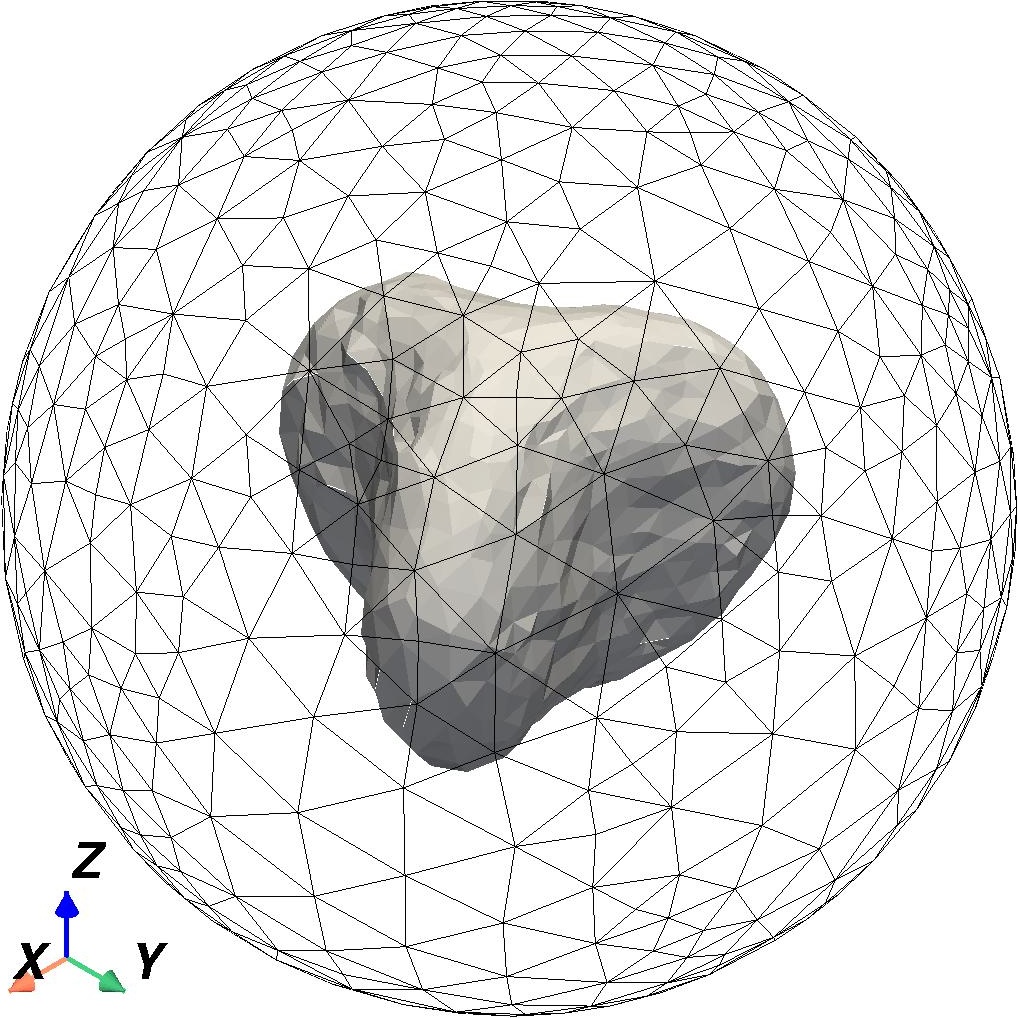}} \quad 
\resizebox{0.235\linewidth}{!}{\includegraphics{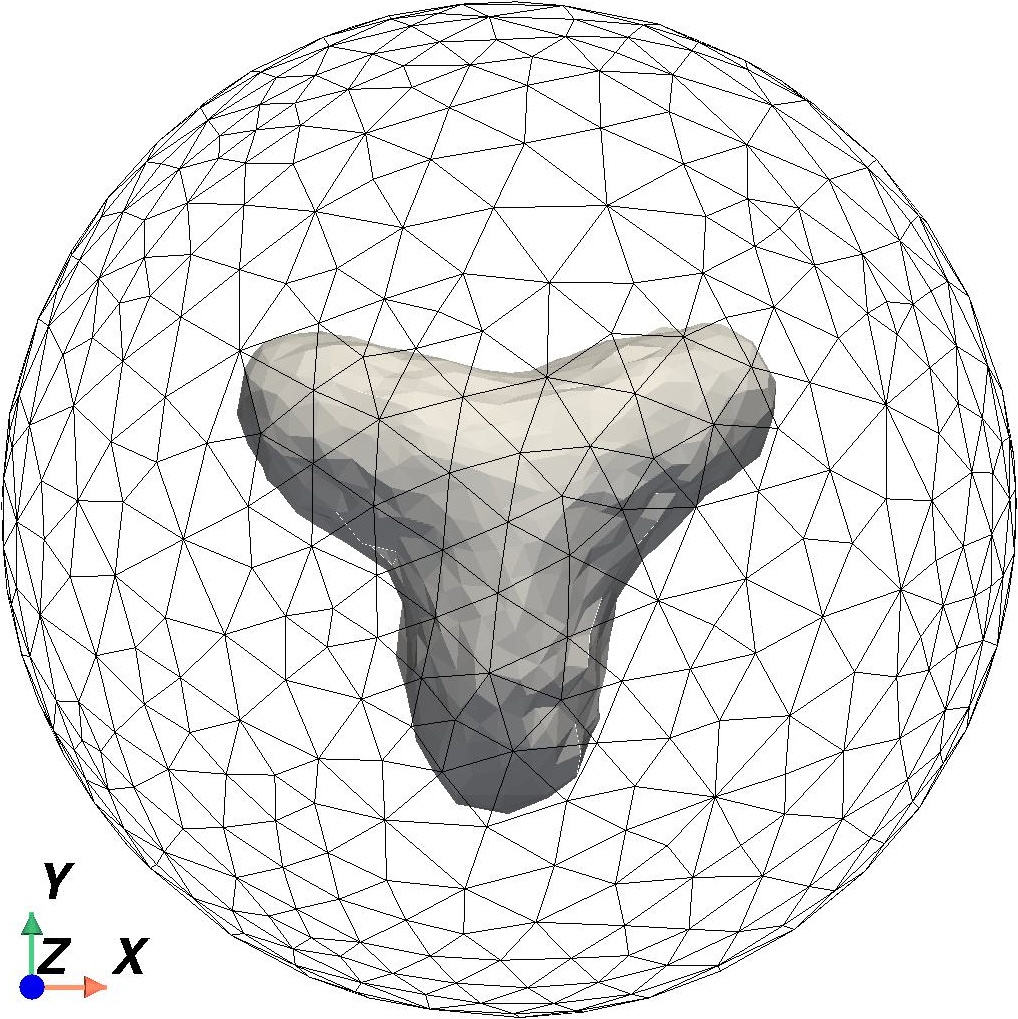}} \quad 
\resizebox{0.235\linewidth}{!}{\includegraphics{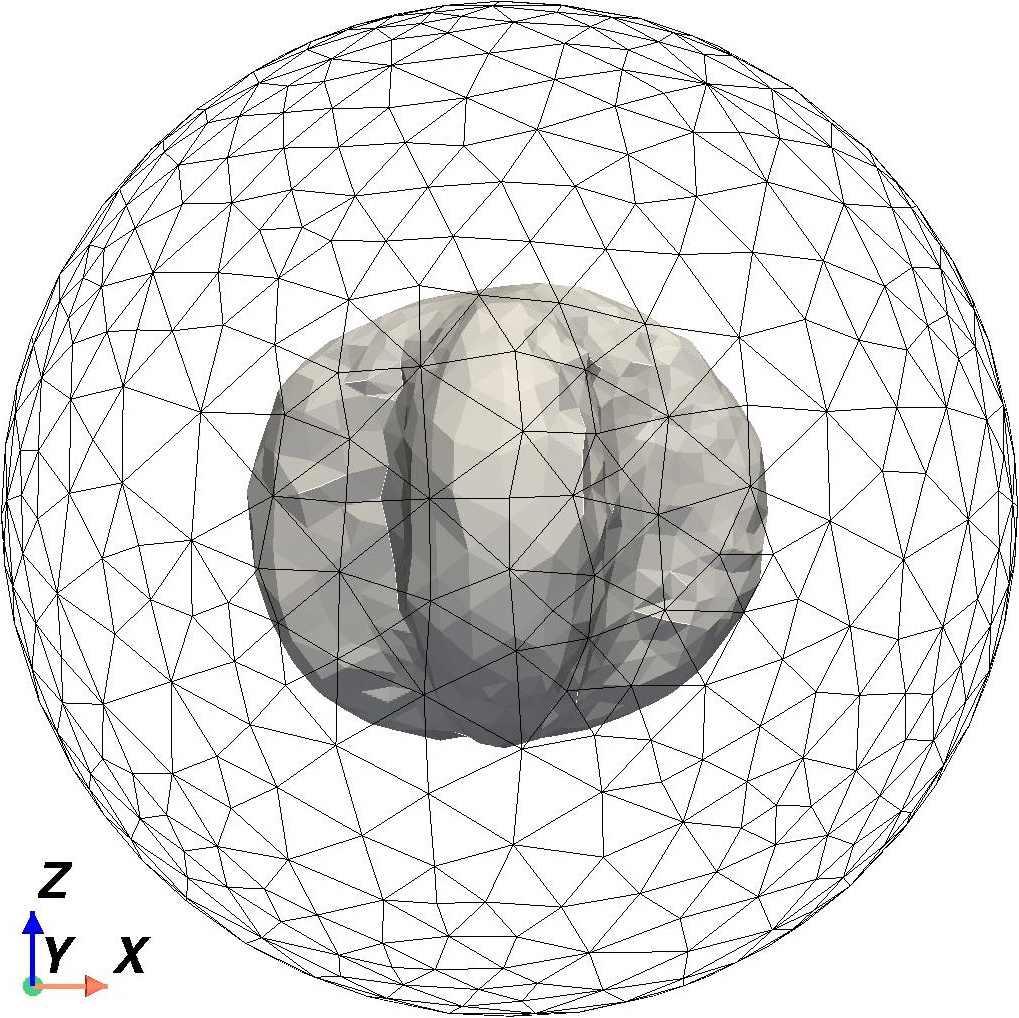}}
\caption{Exact geometry of a star-shape obstacle (top/first row) and reconstructed shapes obtained via SO (middle/second row) and ADMM (bottom/third row) with exact data.}
\label{fig:figure8}
\end{figure}

\begin{figure}[htp!]
\centering 
\resizebox{0.235\linewidth}{!}{\includegraphics{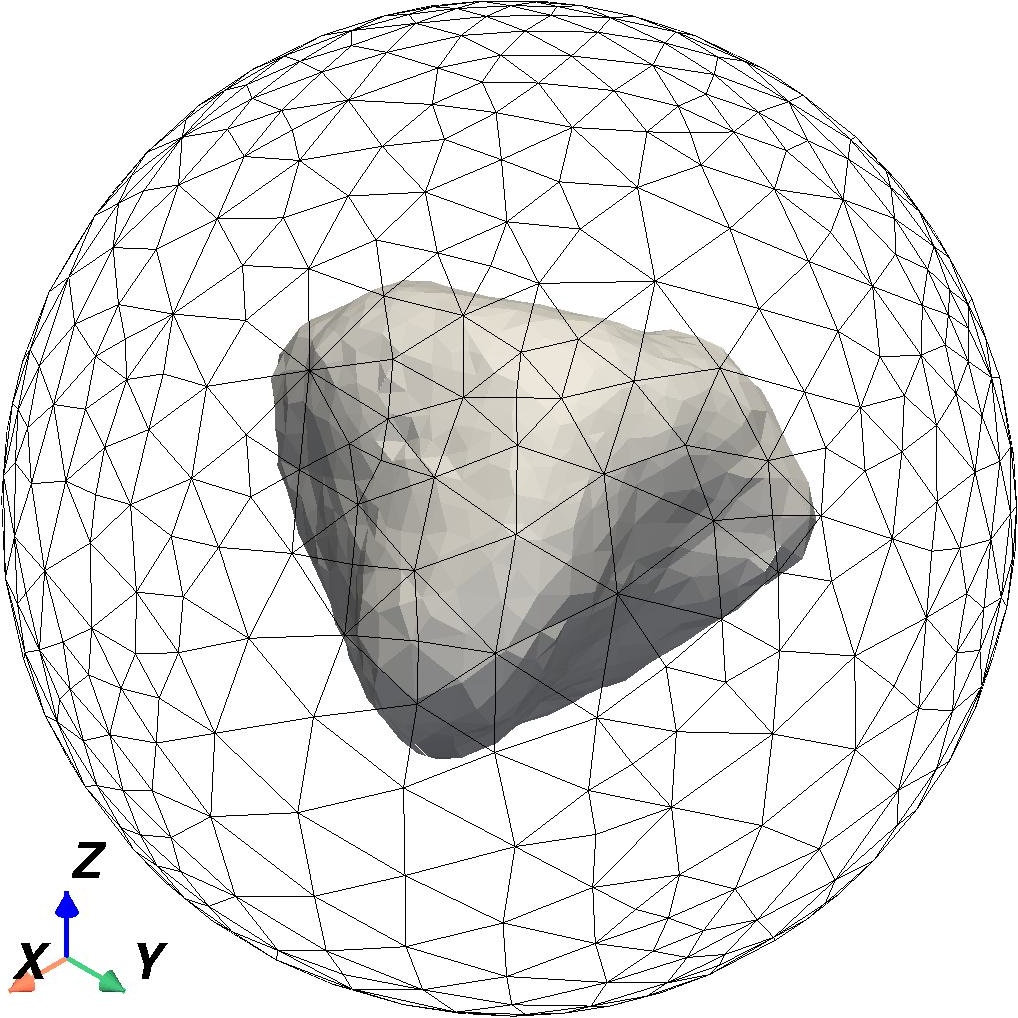}} \quad 
\resizebox{0.235\linewidth}{!}{\includegraphics{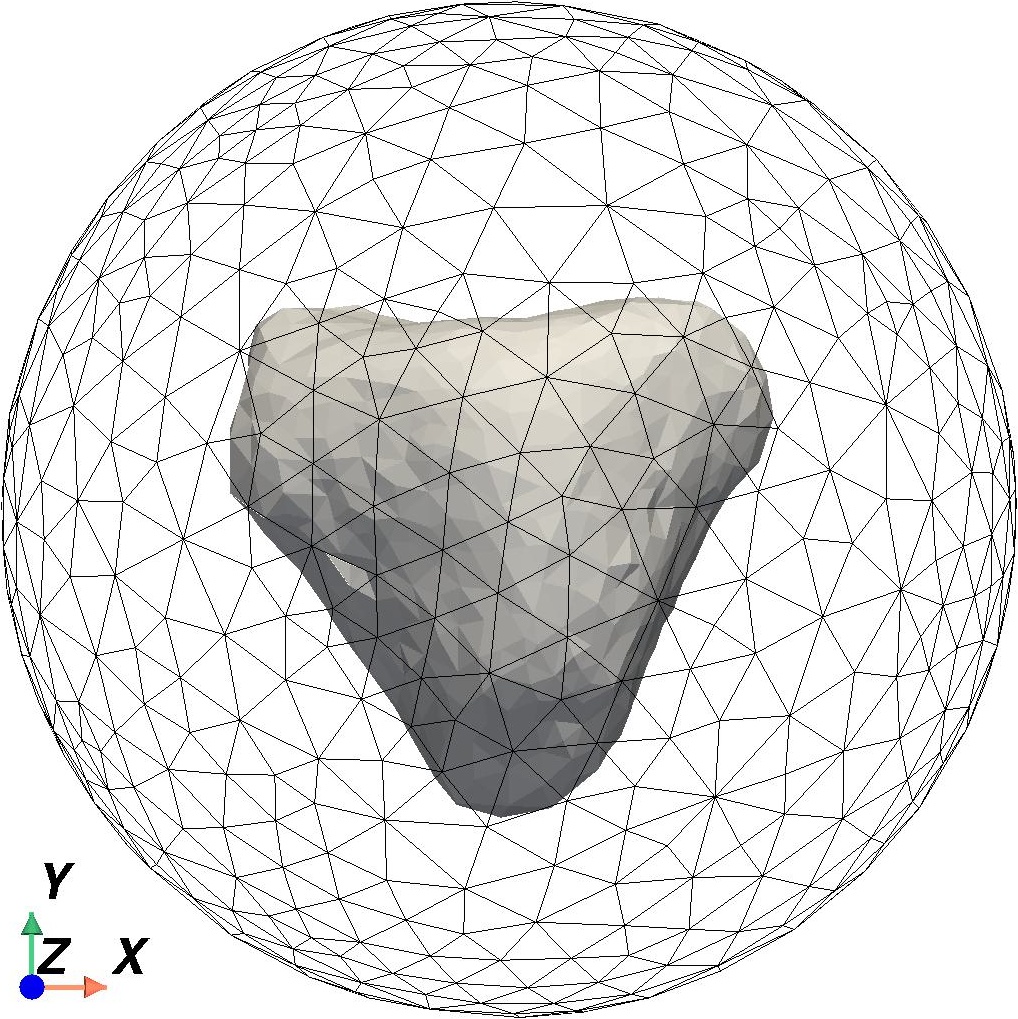}} \quad 
\resizebox{0.235\linewidth}{!}{\includegraphics{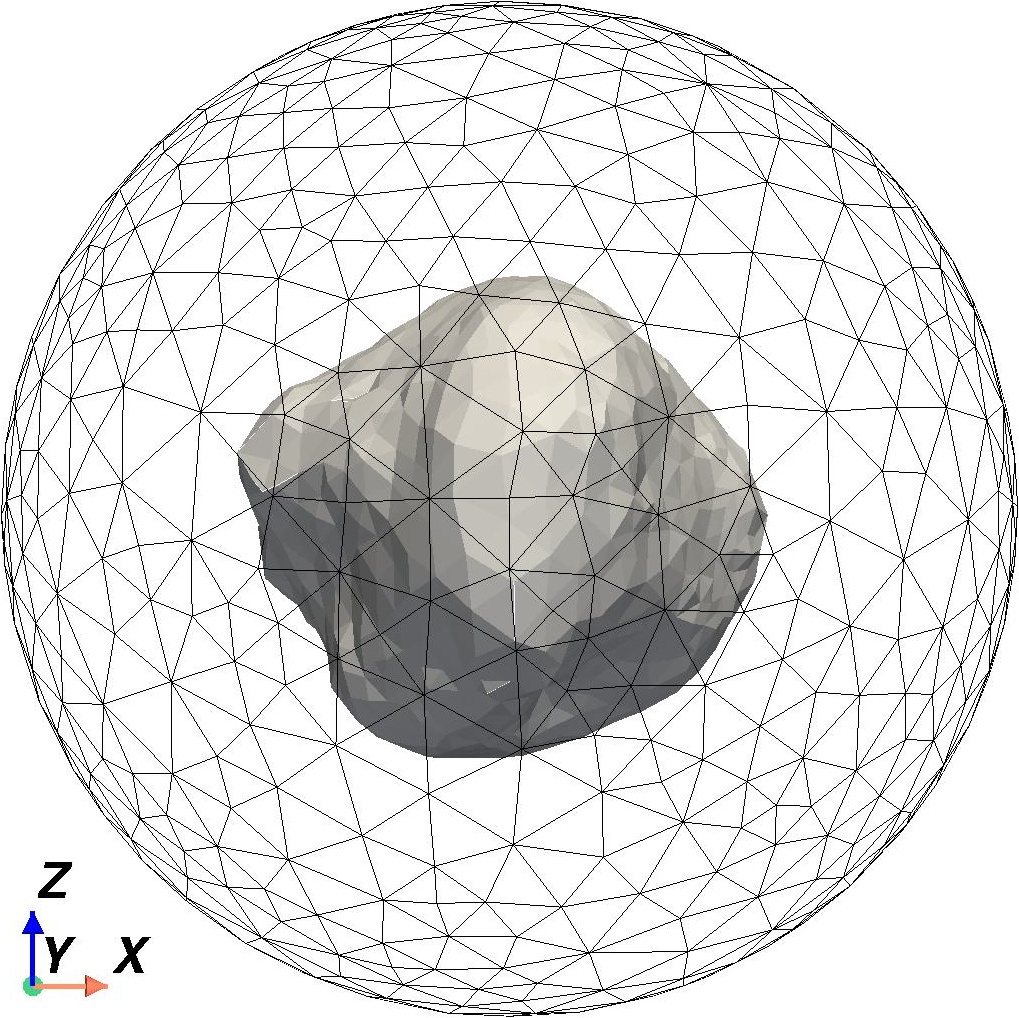}}\\[0.5em]
\resizebox{0.235\linewidth}{!}{\includegraphics{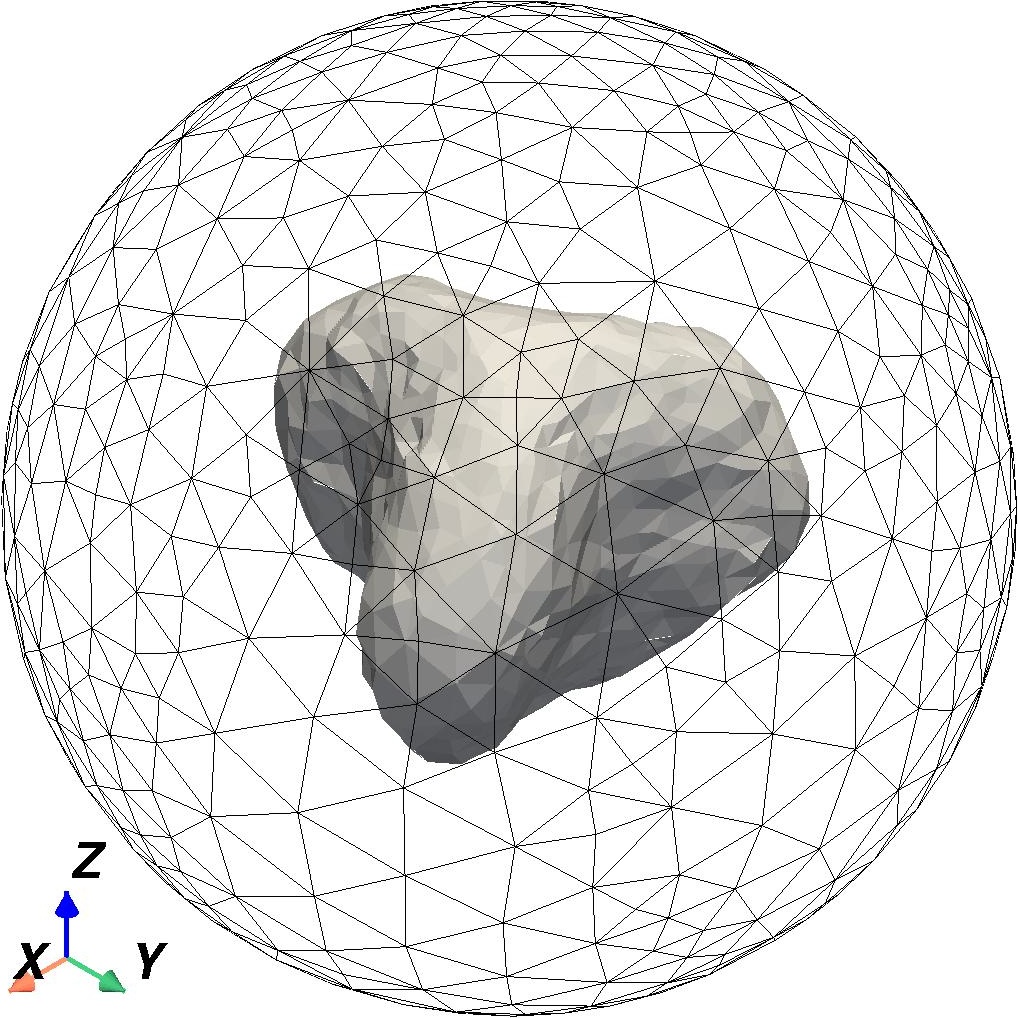}} \quad 
\resizebox{0.235\linewidth}{!}{\includegraphics{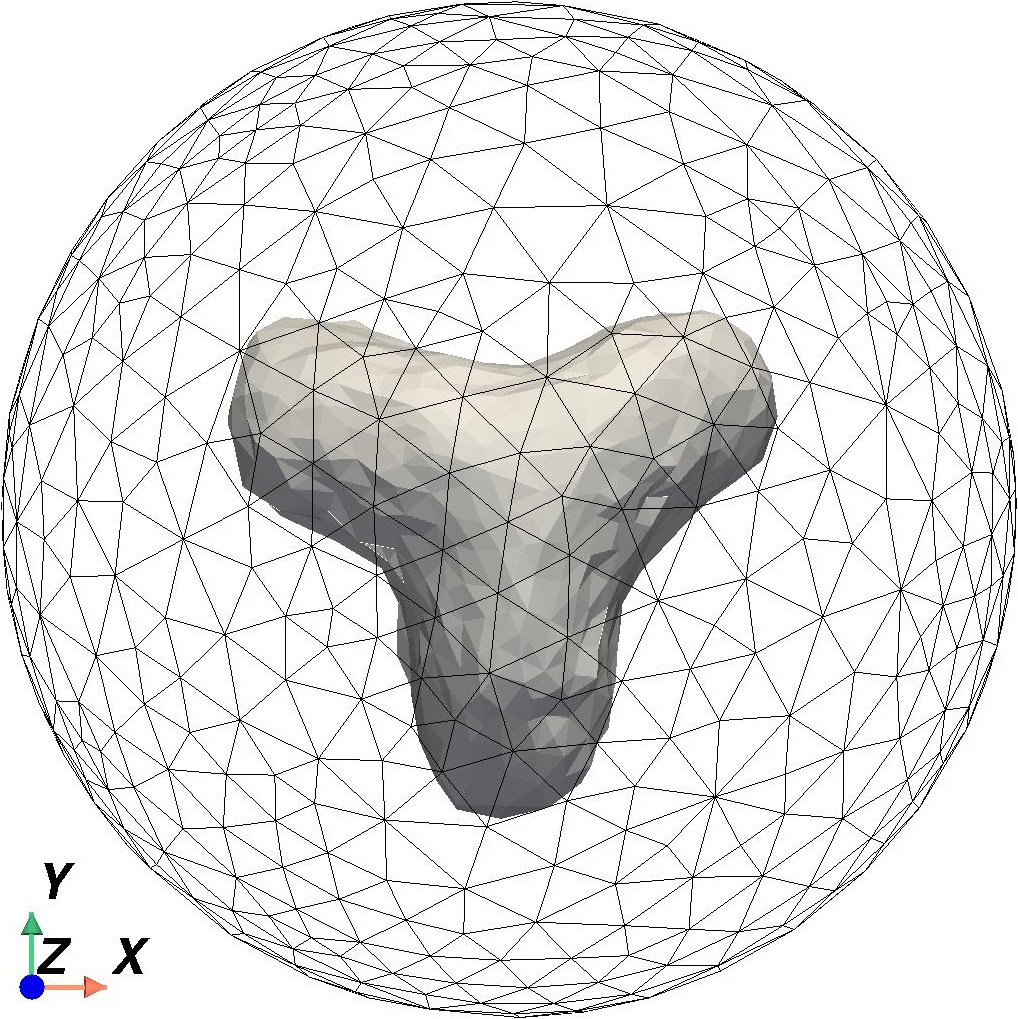}} \quad 
\resizebox{0.235\linewidth}{!}{\includegraphics{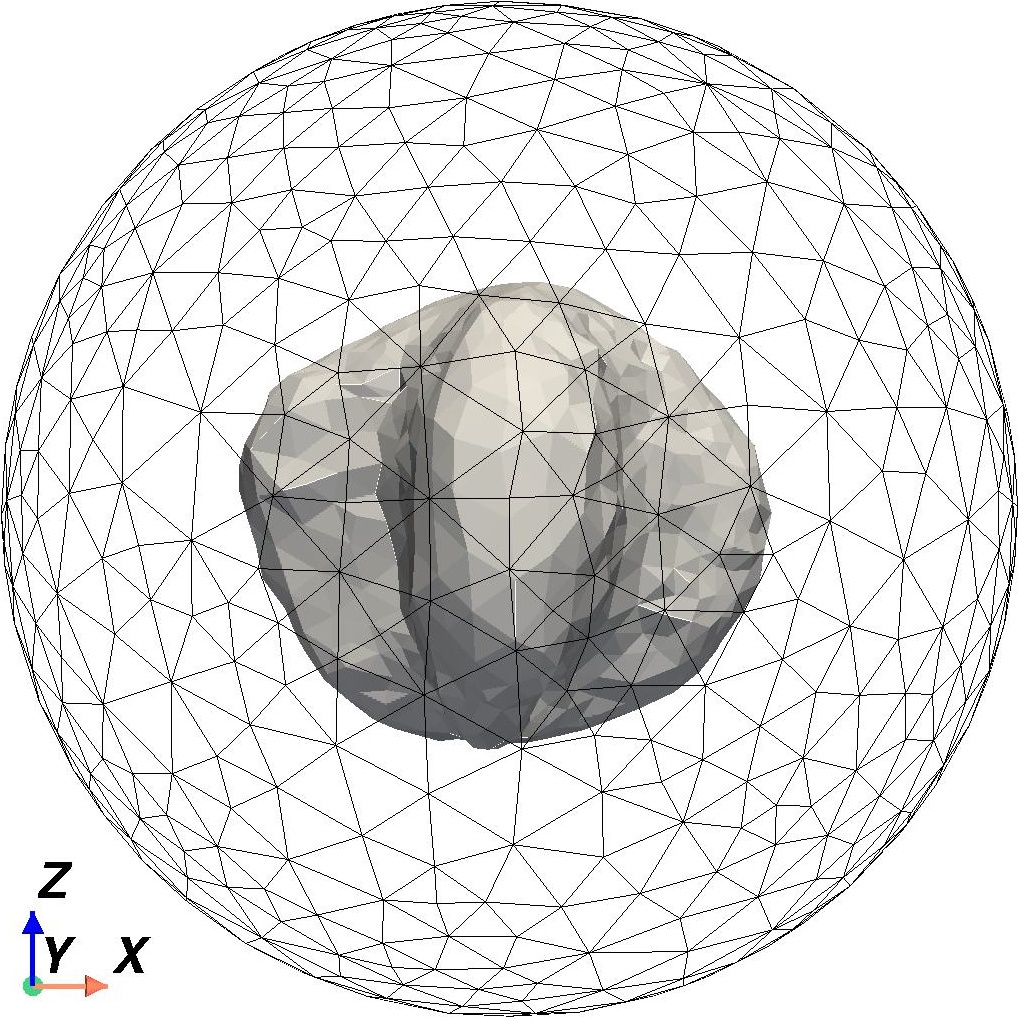}}
\caption{Reconstructed shapes obtained via SO (top row) and ADMM (bottom row) with noisy data at a $30\%$ noise level.}
\label{fig:figure9}
\end{figure}

\begin{figure}[htp!]
\centering 
\resizebox{0.235\linewidth}{!}{\includegraphics{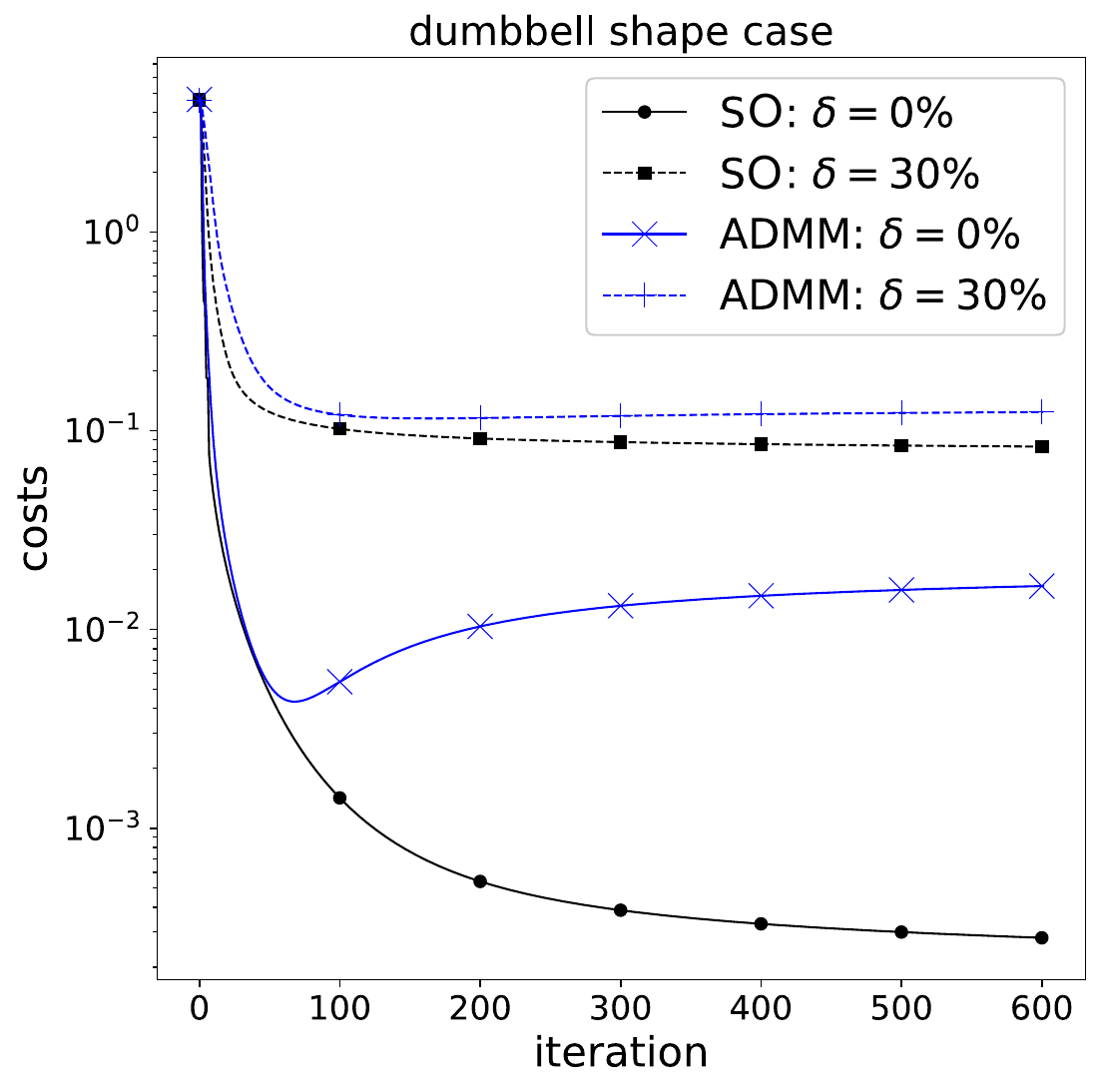}}
\resizebox{0.235\linewidth}{!}{\includegraphics{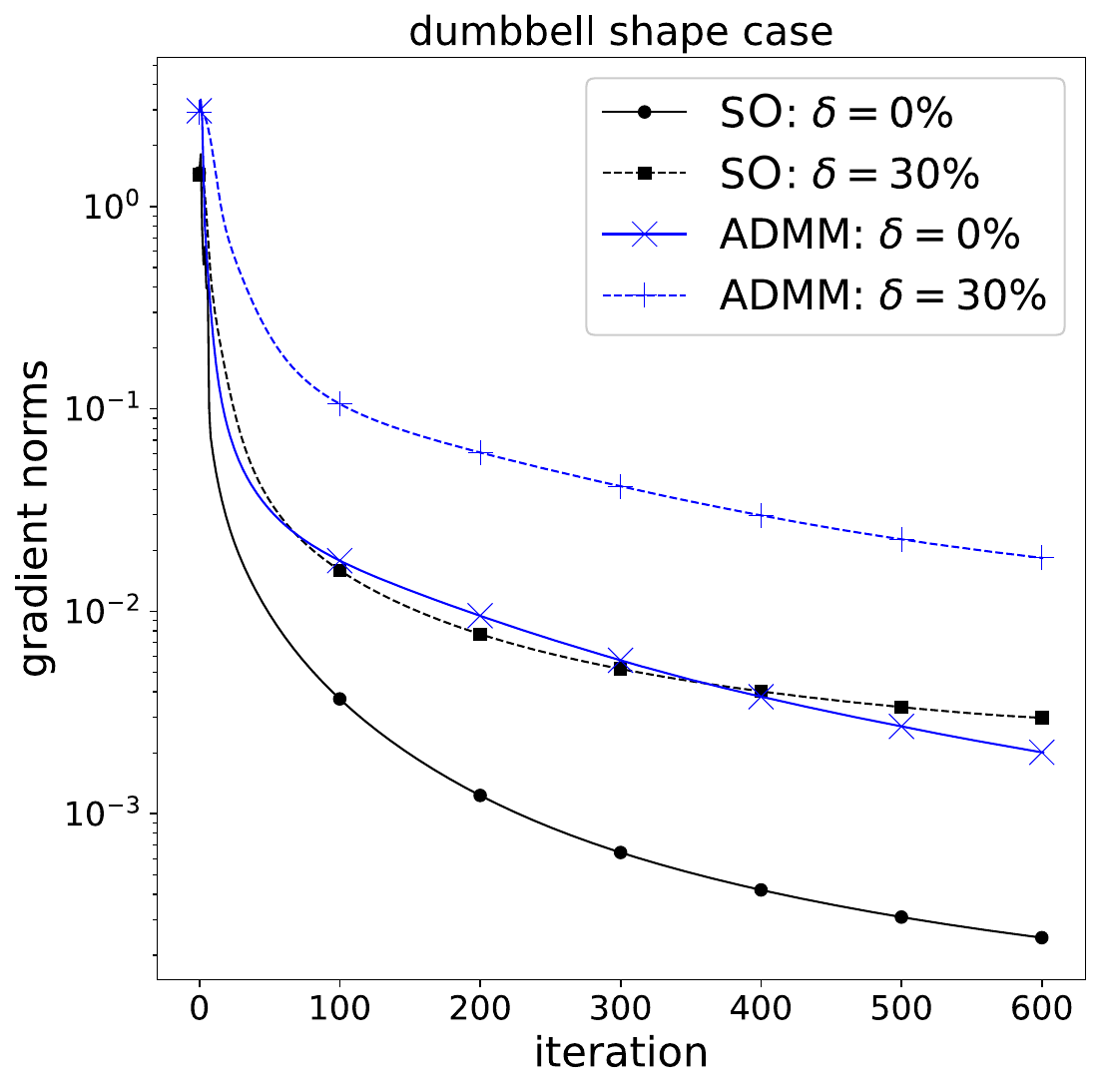}} \quad
\resizebox{0.235\linewidth}{!}{\includegraphics{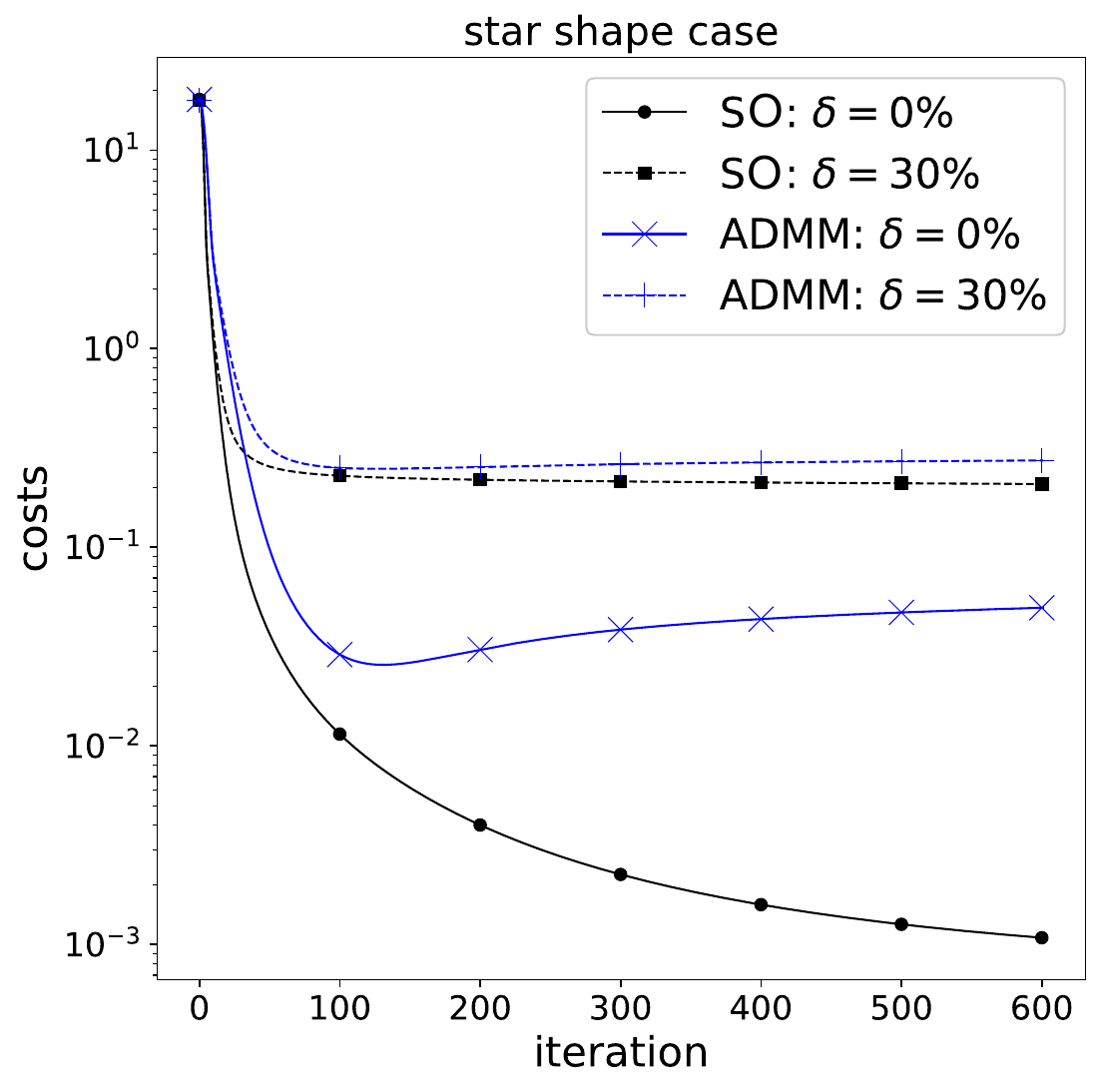}}
\resizebox{0.235\linewidth}{!}{\includegraphics{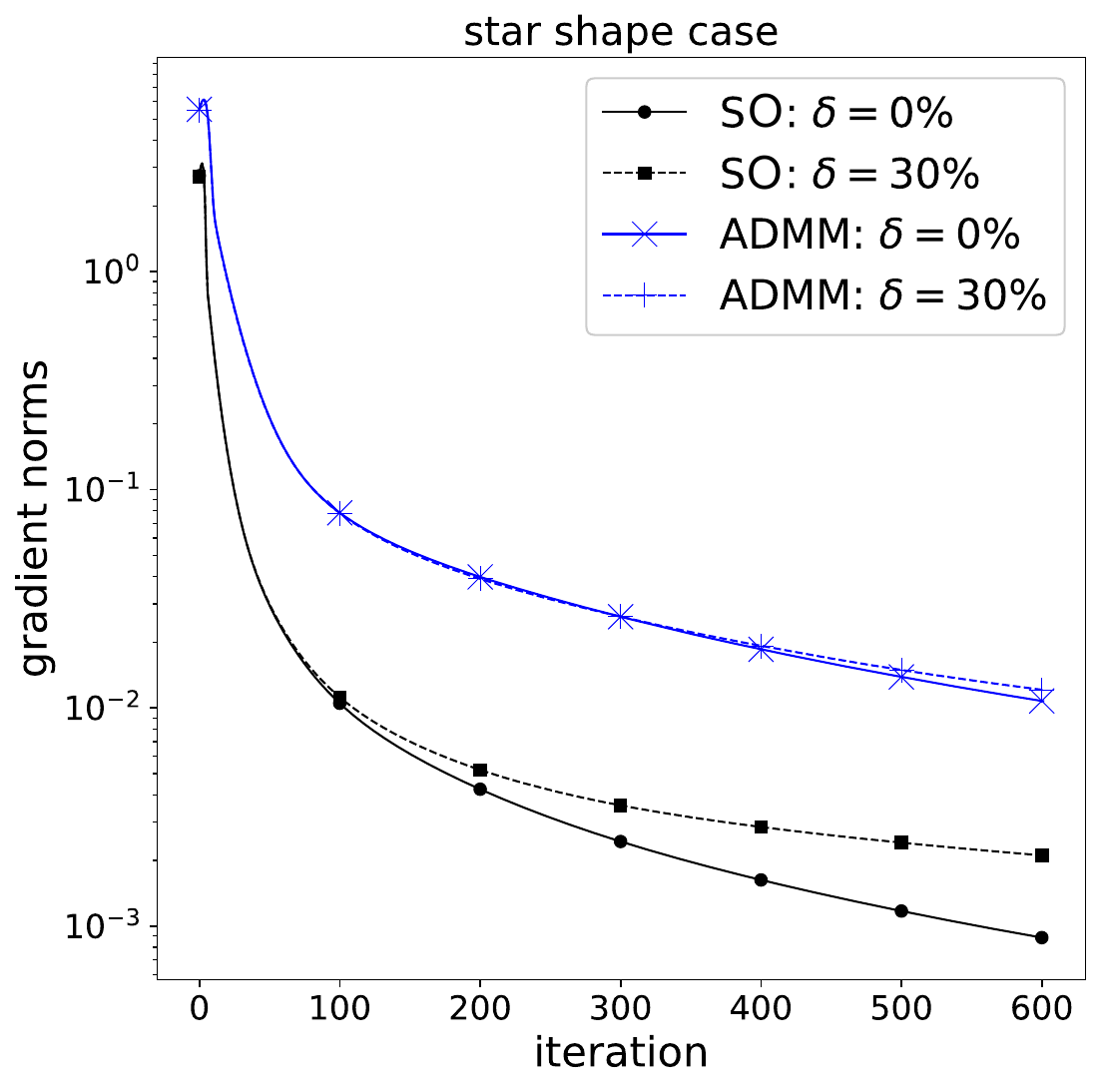}}
\caption{Histories of cost values and gradient norms.}
\label{fig:figure10}
\end{figure}

\subsection{Numerical experiments in 3D via ADMM with space-dependent diffusion matrix} \label{subsec:ADMM_3D_numerics_general_case}

Finally, we explore a broader scenario where instead of a diffusion coefficient, a diffusion matrix dependent on space is utilized. 
We comment here that the expression for the shape derivative remains unchanged; however, \ferj{assumptions in \eqref{eq:assumptions} need to be modified accordingly for technical reasons} and the specific details of the calculation differ. For instance, it should be noted that $\sigma\dn{u}$ should be computed as $(\sigma\nabla{u}) \cdot \nn$.
These differences are omitted here, as the underlying argument closely follows that of Section \ref{sec:shape_derivatives}.

The computational setup to solve the present case mirrors the previous subsection, except for $\sigma(x) = (\sigma_{ij}(x))_{ij} \in L^{\infty}(D)^{3 \times 3}$, $1 \leqslant i, j \leqslant d$, where $\sigma_{ij} = 0$ if $i \neq j$, $\sigma_{11} = 2 - 0.5\cos{(\arctan(x_{2}/x_{1}))}$, $\sigma_{22} = 2 + 0.5\sin{(\arctan(x_{2}/x_{1}))}$, and $\sigma_{33} = 2 + 0.5\sin{(\pi x_{1})} \sin{(\pi x_{2})}$. 
Furthermore, we assume that the data is corrupted with $\delta = 30\%$ noise.

The figures presented in Figure~\ref{fig:figure11} and Figure~\ref{fig:figure12} depict the numerical results of the current experiment. 
As expected, even in the general case, ADMM exhibits superior accuracy over SO, successfully reconstructing unknown obstacles amidst significant data noise. 
Remarkably, as evident in the plots shown in Figure~\ref{fig:figure11} and Figure~\ref{fig:figure12}, ADMM further distinguishes itself by accurately detecting the concave features of these obstacles even under considerable noise levels, thereby emphasizing its pronounced superiority over SO.
%
%
\begin{figure}[htp!]
\centering 
\resizebox{0.235\linewidth}{!}{\includegraphics{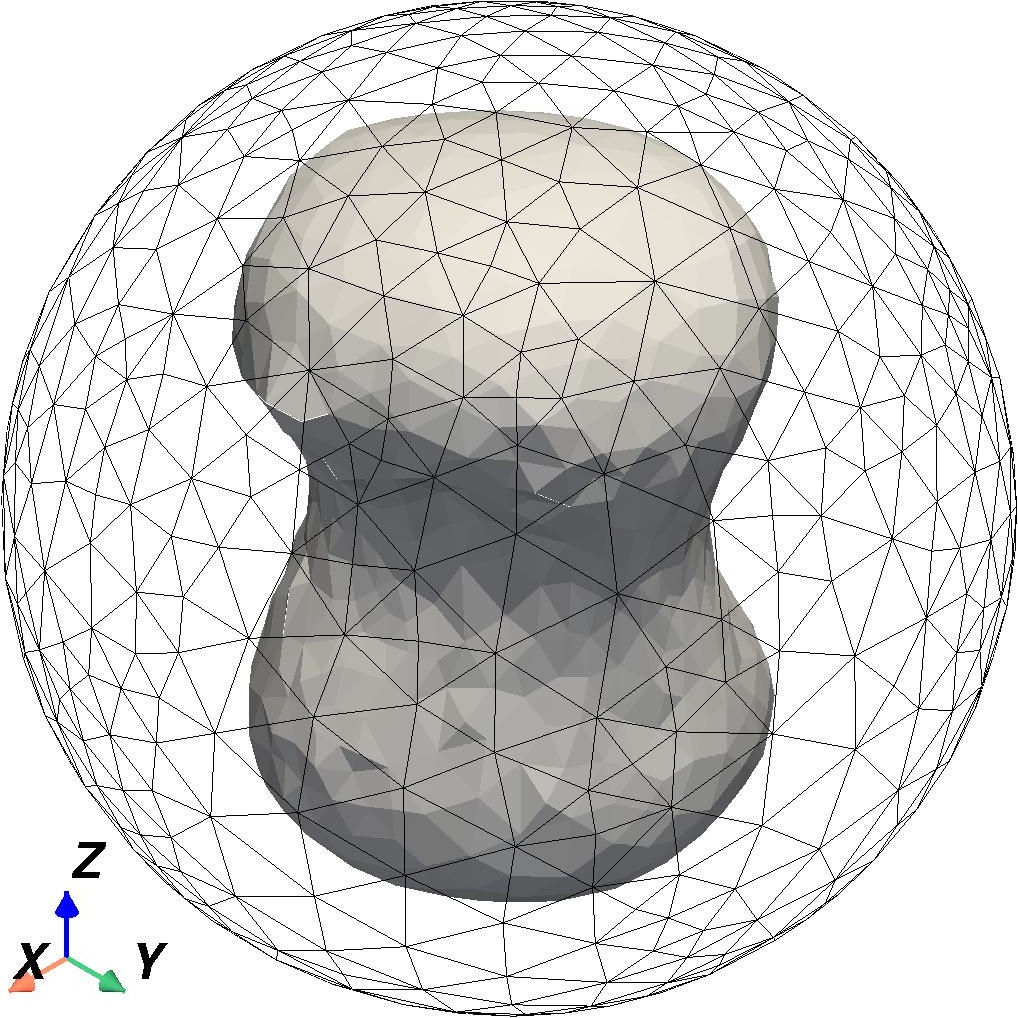}} \quad 
\resizebox{0.235\linewidth}{!}{\includegraphics{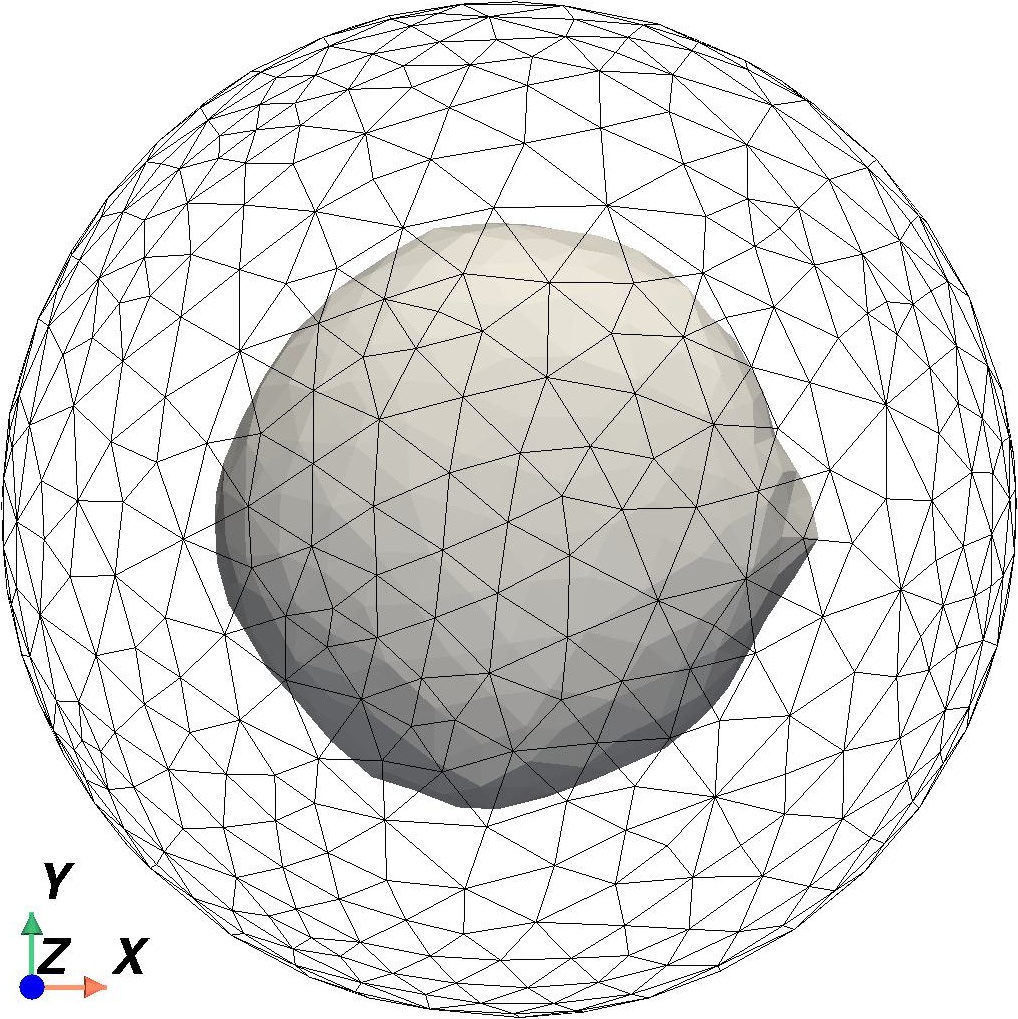}} \quad 
\resizebox{0.235\linewidth}{!}{\includegraphics{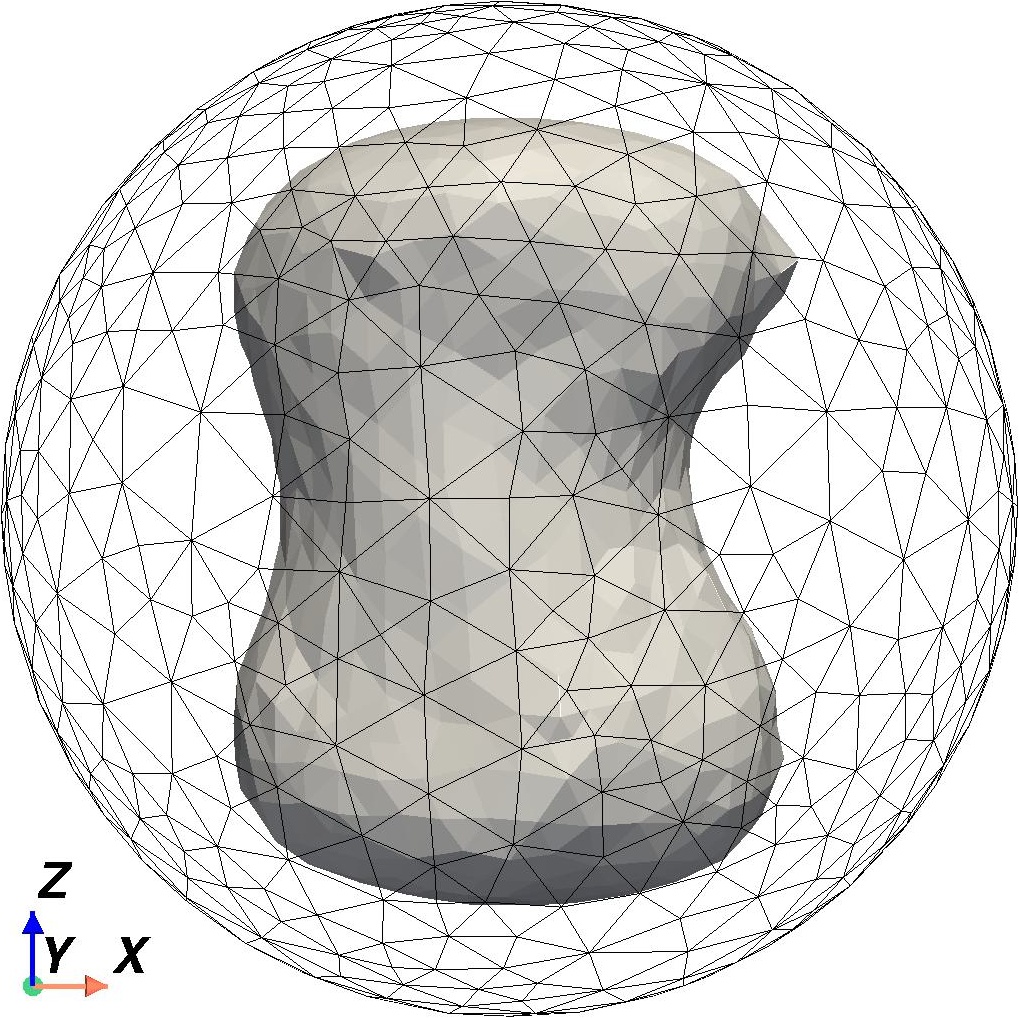}}\\[0.5em]
\resizebox{0.235\linewidth}{!}{\includegraphics{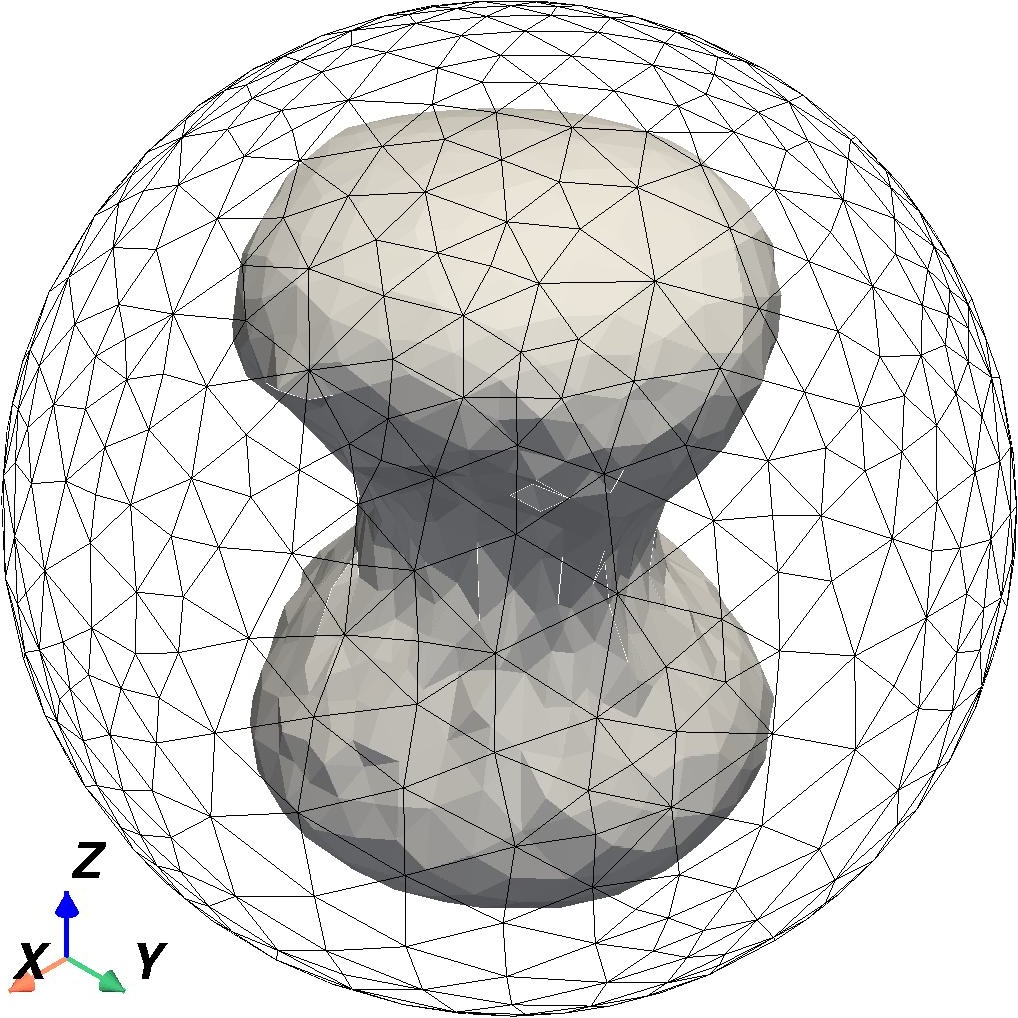}} \quad 
\resizebox{0.235\linewidth}{!}{\includegraphics{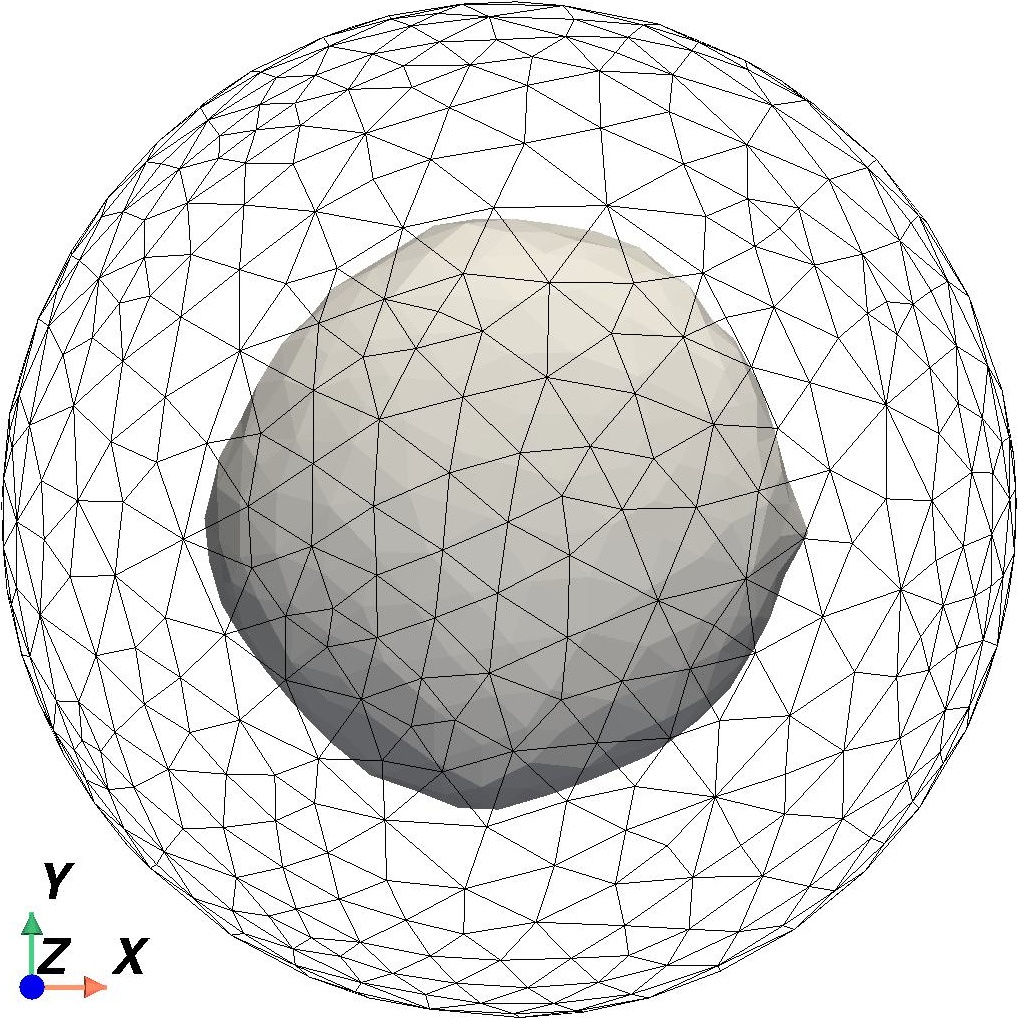}} \quad 
\resizebox{0.235\linewidth}{!}{\includegraphics{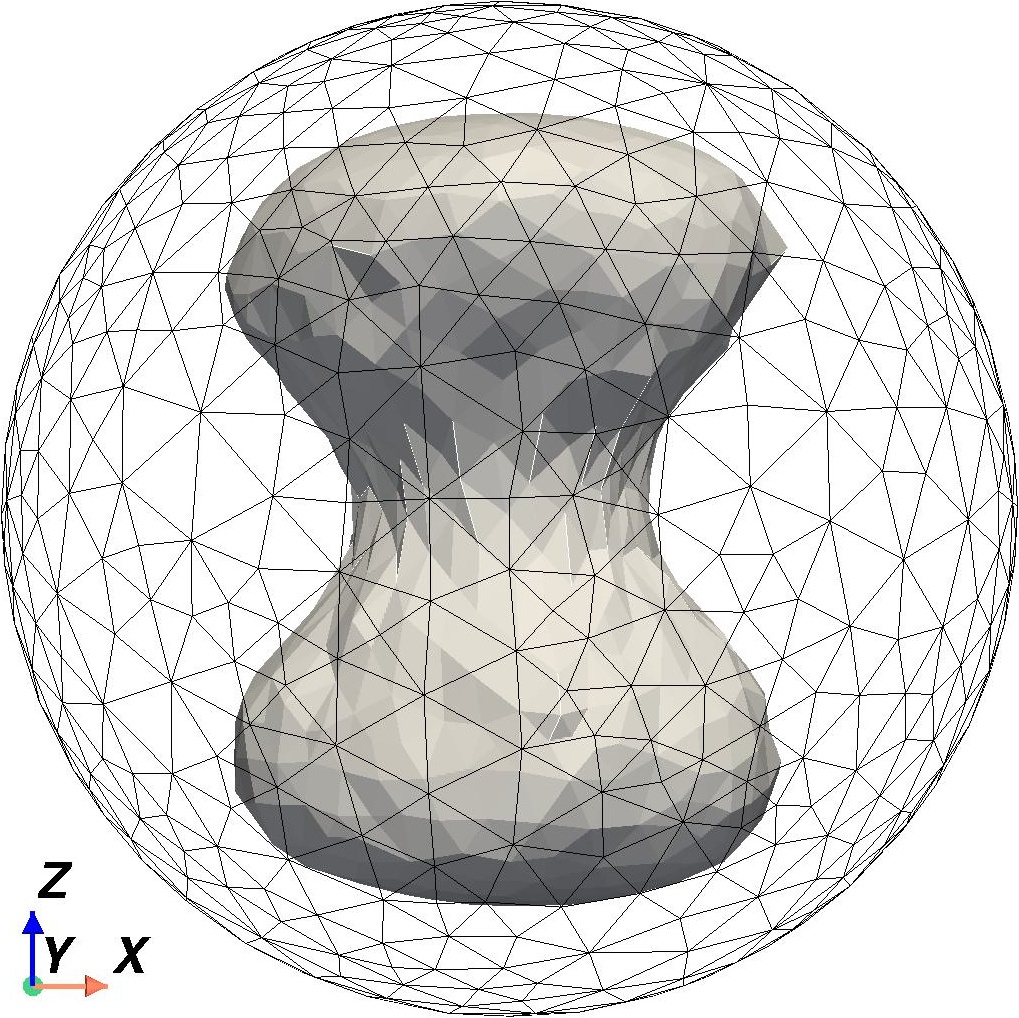}}
\caption{Reconstructed shapes obtained via SO (top row) and ADMM (bottom row) with noisy data at a $30\%$ noise level.}
\label{fig:figure11}
\end{figure}

\begin{figure}[htp!]
\centering 
\resizebox{0.235\linewidth}{!}{\includegraphics{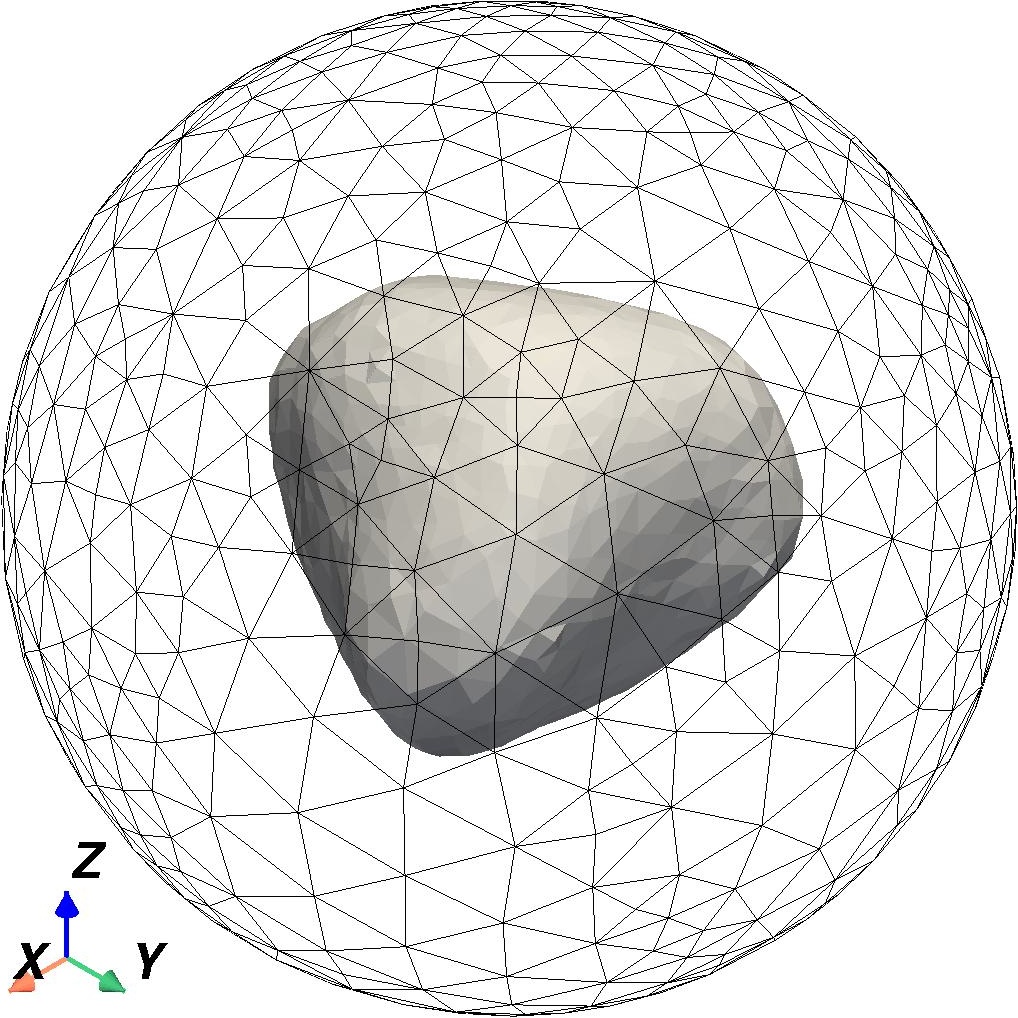}} \quad 
\resizebox{0.235\linewidth}{!}{\includegraphics{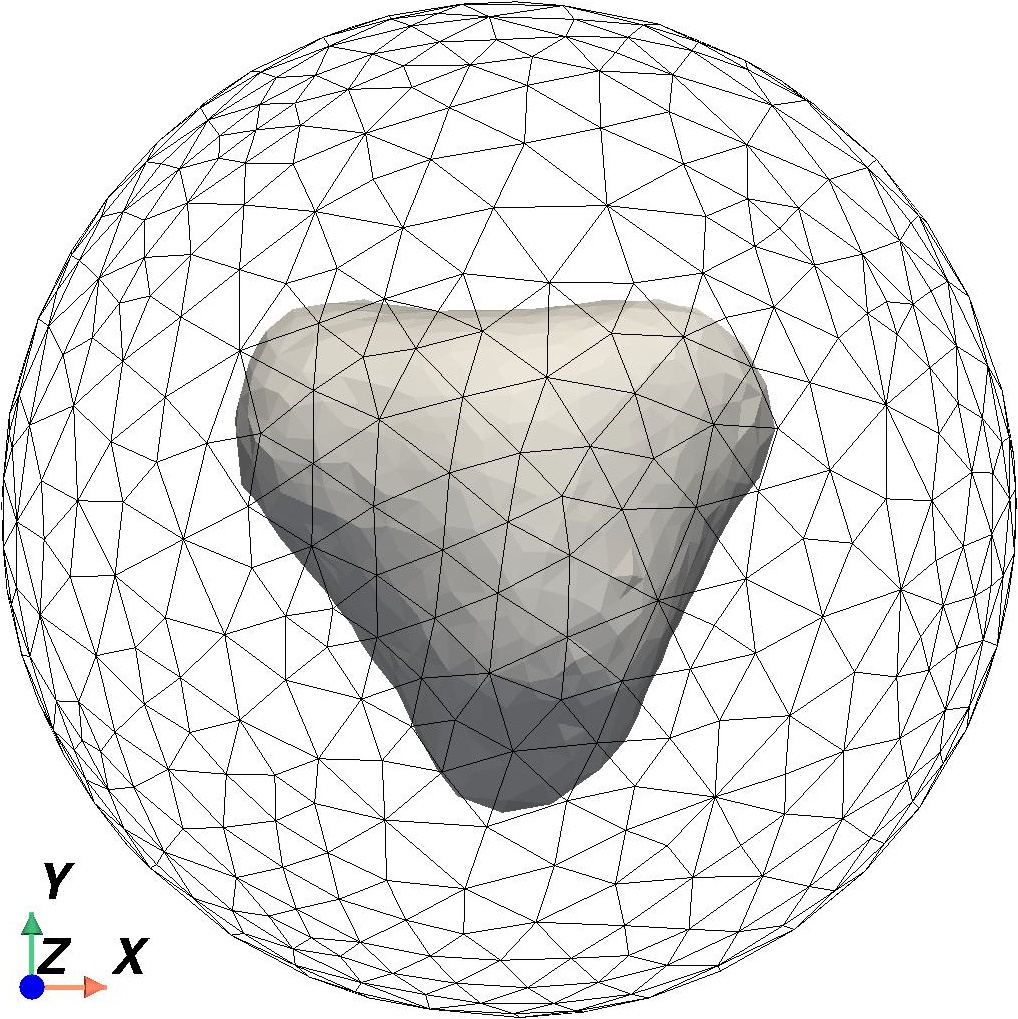}} \quad 
\resizebox{0.235\linewidth}{!}{\includegraphics{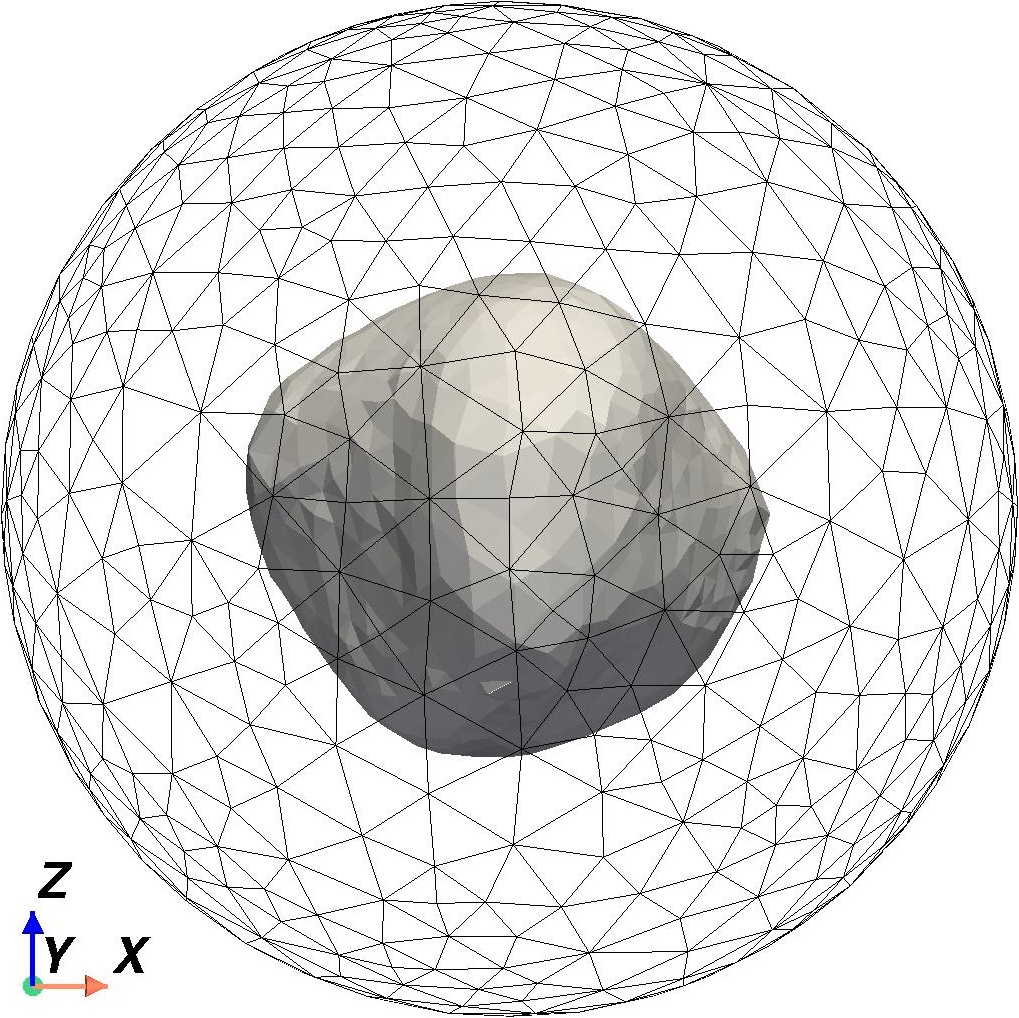}}\\[0.5em]
\resizebox{0.235\linewidth}{!}{\includegraphics{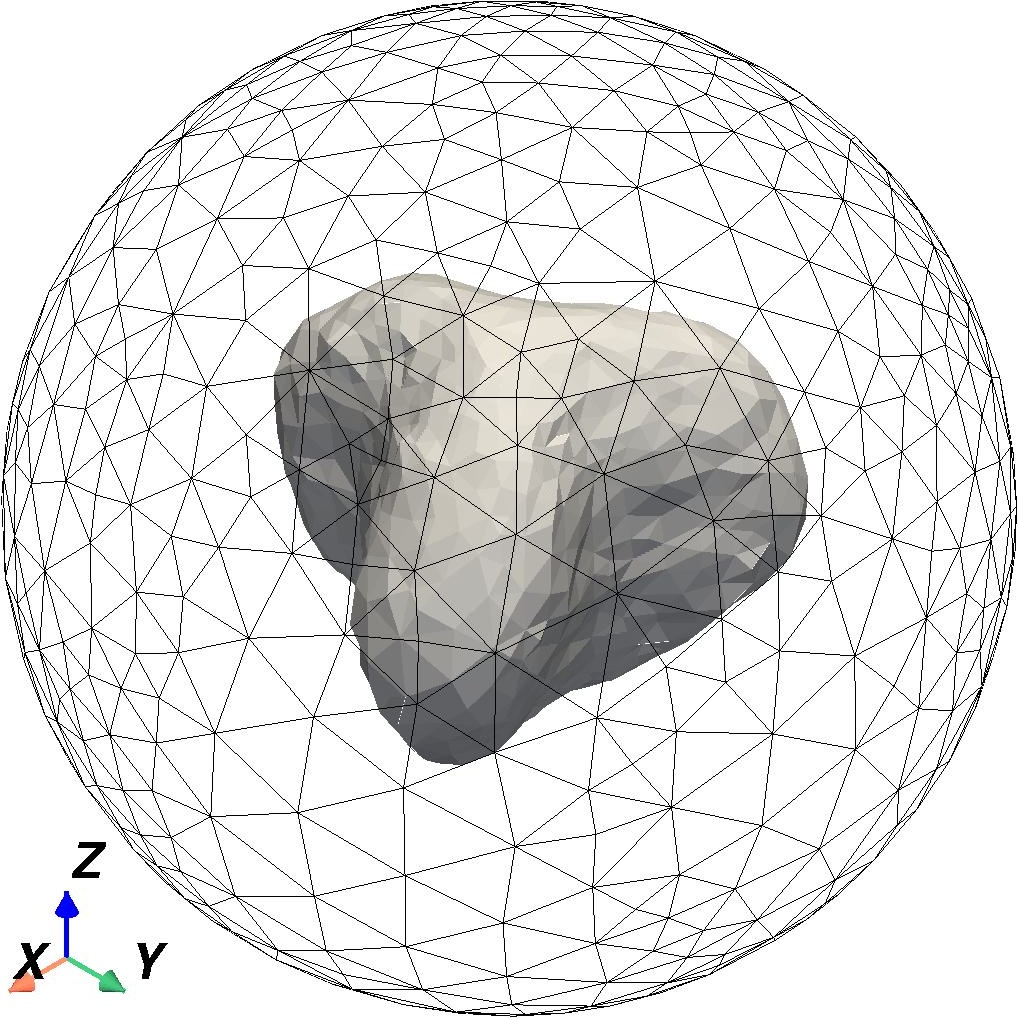}} \quad 
\resizebox{0.235\linewidth}{!}{\includegraphics{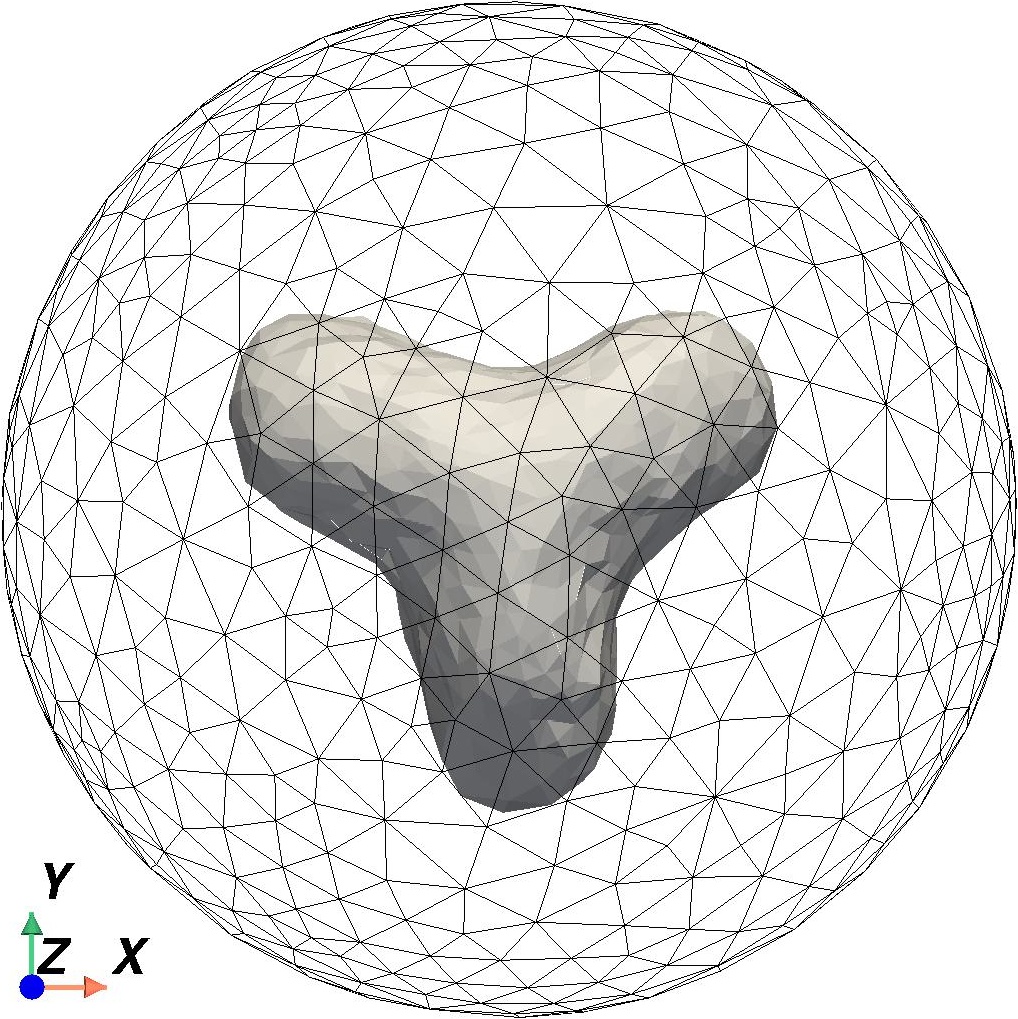}} \quad 
\resizebox{0.235\linewidth}{!}{\includegraphics{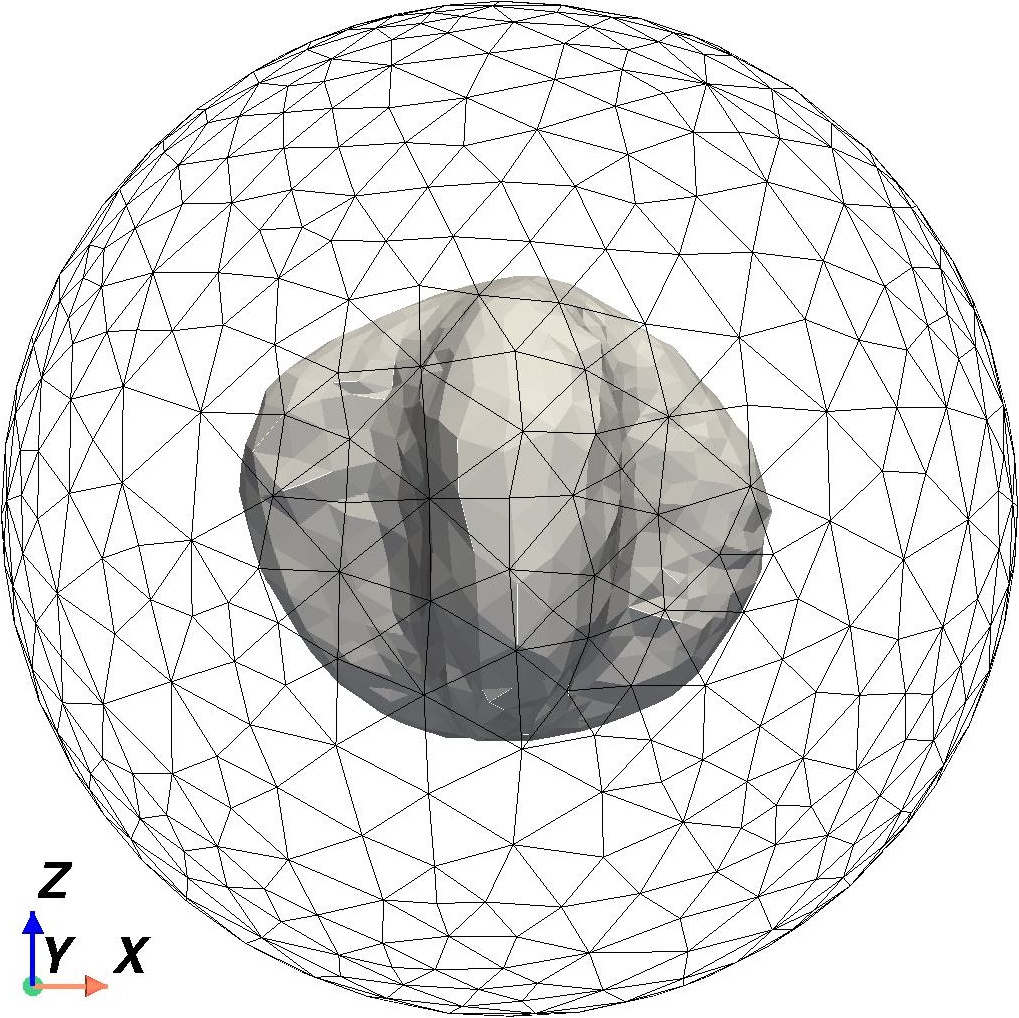}}
\caption{Reconstructed shapes obtained via SO (top row) and ADMM (top row) with noisy data at a $30\%$ noise level.}
\label{fig:figure12}
\end{figure}

\section{Conclusion}\label{sec:conclusion}
In this paper, we investigated a shape inverse problem for the advection–diffusion equation with spatially varying coefficients.
Within a shape optimization framework, we aimed to reconstruct an unknown obstacle from boundary measurements.
We considered two objective functions, $J_D$ and $J_N$, established their differentiability, and derived the corresponding shape gradients using the adjoint method.
Numerical reconstructions were carried out using the alternating direction method of multipliers (ADMM) combined with a Sobolev gradient descent approach in a finite element setting.
The results demonstrated accurate reconstructions of various obstacle shapes, even in the presence of noise.
In particular, ADMM improved the detection of concavities, especially in cases with constant diffusion and spatially varying advection.

\section*{Acknowledgements}
We would like to express our sincere gratitude to the two reviewers for their valuable feedback and insightful suggestions, which greatly contributed to improving the quality of this work. We also thank the Editor for their careful handling of the manuscript.
JFTR is supported by the JSPS Postdoctoral Fellowships for Research in Japan, and partially by the JSPS Grant-in-Aid for Early-Career Scientists Grant Number JP23K13012 and the JST CREST Grant Number JPMJCR2014.
%
%
%
%
%
\section*{Conflict of interest}
The authors declare that they have no conflict of interest.
%
%
%
%
%
%
\bibliographystyle{alpha} 
\bibliography{main}   

\newcommand{\etalchar}[1]{$^{#1}$}
\begin{thebibliography}{DMNV07}

\bibitem[ADK07]{AfraitesDambrineKateb2007}
L.~Afraites, M.~Dambrine, and D.~Kateb.
\newblock Shape methods for the transmission problem with a single measurement.
\newblock {\em Numer. Funct. Anal. Optim.}, 28(5--6):519--551, 2007.

\bibitem[AMN22]{AfraitesMasnaouiNachaoui2022}
L.~Afraites, C.~Masnaoui, and M.~Nachaoui.
\newblock Shape optimization method for an inverse geometric source problem and
  stability at critical shape.
\newblock {\em Discrete Contin. Dyn. Syst. Ser. S}, 15(1):1--21, 2022.

\bibitem[AR25]{AfraitesRabago2025}
L.~Afraites and J.~F.~T. Rabago.
\newblock Shape optimization methods for detecting an unknown boundary with the
  robin condition by a single boundary measurement.
\newblock {\em Discrete Contin. Dyn. Syst. Ser. S}, 18(1):43--76, 2025.

\bibitem[BCD11]{BadraCaubetDambrine2011}
M.~Badra, F.~Caubet, and M.~Dambrine.
\newblock Detecting an obstacle immersed in a fluid by shape optimization
  methods.
\newblock {\em Math. Models Methods Appl. Sci.}, 21:2069--2201, 2011.

\bibitem[BP13]{BacaniPeichl2013}
J.~B. Bacani and G.~H. Peichl.
\newblock On the first-order shape derivative of the {K}ohn-{V}ogelius cost
  functional of the {B}ernoulli problem.
\newblock {\em Abstr. Appl. Anal.}, 2013:19 pp. Article ID 384320, 2013.

\bibitem[CAR25]{CherratAfraitesRabago2025}
E.~Cherrat, L.~Afraites, and J.~F.~T. Rabago.
\newblock Shape reconstruction for advection-diffusion problems by shape
  optimization techniques: The case of constant velocity.
\newblock {\em Discrete Contin. Dyn. Syst. Ser. S}, 18(1):296--320, 2025.

\bibitem[Cau13]{Caubet2013}
F.~Caubet.
\newblock Instability of an inverse problem for the stationary
  {N}avier-{S}tokes equations.
\newblock {\em SIAM J. Control \& Optim.}, 51(4), 2013.

\bibitem[CDK13]{CaubetDambrineKateb2013}
F.~Caubet, M.~Dambrine, and D.~Kateb.
\newblock Shape optimization methods for the inverse obstacle problem with
  generalized impedance boundary conditions.
\newblock {\em Inverse Problems}, 29:Art. 115011 (26pp), 2013.

\bibitem[CDKT13]{CaubetDambrineKatebTimimoun2013}
F.~Caubet, M.~Dambrine, D.~Kateb, and C.~Z. Timimoun.
\newblock A {K}ohn-{V}ogelius formulation to detect an obstacle immersed in a
  fluid.
\newblock {\em Inverse Prob. Imaging}, 7(1):123--157, 2013.

\bibitem[CDLZ22]{CaoDiaoLiuZou2022}
X.~Cao, H.~Diao, H.~Liu, and J.~Zou.
\newblock Two single-measurement uniqueness results for inverse scattering
  problems within polyhedral geometries.
\newblock {\em Inverse Prob. Imaging}, 16(6):1501--1528, 2022.

\bibitem[CKY98]{ChapkoKressYoon1998}
R.~Chapko, R.~Kress, and J.-R. Yoon.
\newblock On the numerical solution of an inverse boundary value problem for
  the heat equation.
\newblock {\em Inverse Problems}, 14(4):853, 1998.

\bibitem[CKY99]{ChapkoKressYoon1999}
R.~Chapko, R.~Kress, and J.-R. Yoon.
\newblock An inverse boundary value problem for the heat equation: the
  {N}eumann condition.
\newblock {\em Invers Problems}, 15(4):1033, 1999.

\bibitem[Dem20]{Dempe2020}
S.~Dempe.
\newblock {\em Bilevel Optimization: Theory, Algorithms, Applications and a
  Bibliography}, volume 161 of {\em Springer Optimization and Its
  Applications}, pages 581--672.
\newblock Springer International Publishing, 2020.

\bibitem[DMNV07]{Doganetal2007}
G.~Do\v{g}an, P.~Morin, R.~H. Nochetto, and M.~Verani.
\newblock Discrete gradient flows for shape optimization and applications.
\newblock {\em Comput. Methods Appl. Mech. Engrg.}, 196:3898--3914, 2007.

\bibitem[DZ11]{DelfourZolesio2011}
M.~C. Delfour and J.-P. Zol\'{e}sio.
\newblock {\em {S}hapes and {G}eometries: {M}etrics, {A}nalysis, {D}ifferential
  {C}alculus, and {O}ptimization}, volume~22 of {\em Adv. Des. Control}.
\newblock SIAM, Philadelphia, 2nd edition, 2011.

\bibitem[Fan22]{Fang2022}
W.~Fang.
\newblock Simultaneous recovery of {R}obin boundary and coefficient for the
  {L}aplace equation by shape derivative.
\newblock {\em J. Comput. Appl. Math.}, 413:Art. 114376 13 pp, 2022.

\bibitem[FNPS21]{Fernandezetal2021}
L.~Fernandez, A.~A. Novotny, R.~Prakash, and J.~Soko\l{}owski.
\newblock Pollution sources reconstruction based on the topological derivative
  method.
\newblock {\em Appl. Math. Optim.}, 84:1493--1525, 2021.

\bibitem[FZ09]{FangZeng2009}
W.~Fang and S.~Zeng.
\newblock Numerical recovery of {R}obin boundary from boundary measurements for
  the {L}aplace equation.
\newblock {\em J. Comput. Appl. Math.}, 224:573--580, 2009.

\bibitem[Gri85]{Grisvard1985}
P.~Grisvard.
\newblock {\em Elliptic Problems in Nonsmooth Domains}.
\newblock Pitman Publishing, Marshfield, Massachusetts, 1985.

\bibitem[GT01]{GilbargTrudinger2001}
D.~Gilbarg and N.~S. Trudinger.
\newblock {\em Elliptic Partial Differential Equations of Second Order}.
\newblock Springer-Verlag, Berlin, Heidelberg, 2001.

\bibitem[Hec12]{Hecht2012}
F.~Hecht.
\newblock New development in {F}ree{F}em++.
\newblock {\em J. Numer. Math.}, 20:251--265, 2012.

\bibitem[Het98]{Hettlich1998}
F.~Hettlich.
\newblock The landweber iteration applied to inverse conductive scattering
  problems.
\newblock {\em Inverse Problems}, 14(4):931, 1998.

\bibitem[HP18]{HenrotPierre2018}
A.~Henrot and M.~Pierre.
\newblock {\em Shape Variation and Optimization: A Geometrical Analysis},
  volume~28 of {\em Tracts in Mathematics}.
\newblock European Mathematical Society, Z\"{u}rich, 2018.

\bibitem[HR98]{HettlichRundell1998}
F.~Hettlich and W.~Rundell.
\newblock The determination of a discontinuity in a conductivity from a single
  boundary measurement.
\newblock {\em Inverse Problems}, 14(1):67, 1998.

\bibitem[HT11]{HarbrechtTausch2011}
H.~Harbrecht and J.~Tausch.
\newblock An efficient numerical method for a shape-identification problem
  arising from the heat equation.
\newblock {\em Inverse Problems}, 27(6):065013, 2011.

\bibitem[HT13]{HarbrechtTausch2013}
H.~Harbrecht and J.~Tausch.
\newblock On the numerical solution of a shape optimization problem for the
  heat equation.
\newblock {\em SIAM J. Sci. Comput.}, 35(1):A104--A121, 2013.

\bibitem[KC98]{ColtonKress2013}
R.~Kress and D.~Colton.
\newblock {\em Inverse Acoustic and Electromagnetic Scattering Theory},
  volume~93 of {\em Applied Mathematical Sciences}.
\newblock Springer, New York, 3rd edition, 1998.

\bibitem[KR01]{KressRundell2001}
R.~Kress and W.~Rundell.
\newblock Inverse scattering for shape and impedance.
\newblock {\em Inverse problems}, 17(4):1075, 2001.

\bibitem[KV84]{KohnVogelius1984}
R.~Kohn and M.~Vogelius.
\newblock Determining conductivity by boundary measurements.
\newblock {\em Commun. Pure Appl. Math.}, 37:289--298, 1984.

\bibitem[MS76]{MuratSimon1976}
F.~Murat and J.~Simon.
\newblock Sur le contr\^{o}le par un domaine g\'{e}om\'{e}trique.
\newblock Research report 76015, Univ. Pierre et Marie Curie, Paris, 1976.

\bibitem[Neu97]{Neuberger1997}
J.~W. Neuberger.
\newblock {\em Sobolev Gradients and Differential Equations}.
\newblock Springer-Verlag, Berlin, 1997.

\bibitem[RA18]{RabagoAzegami2018}
J.~F.~T. Rabago and H.~Azegami.
\newblock Shape optimization approach to defect-shape identification with
  convective boundary condition via partial boundary measurement.
\newblock {\em Japan J. Indust. Appl. Math.}, 31(1):131--176, 2018.

\bibitem[RA19]{RabagoAzegami2019a}
J.~F.~T. Rabago and H.~Azegami.
\newblock An improved shape optimization formulation of the {B}ernoulli problem
  by tracking the {N}eumann data.
\newblock {\em J. Eng. Math.}, 117:1--29, 2019.

\bibitem[RA20]{RabagoAzegami2020}
J.~F.~T. Rabago and H.~Azegami.
\newblock A second-order shape optimization algorithm for solving the exterior
  {B}ernoulli free boundary problem using a new boundary cost functional.
\newblock {\em Comput. Optim. Appl.}, 77(1):251--305, 2020.

\bibitem[RAN25]{RabagoAfraitesNotsu2025}
J.~F.~T. Rabago, L.~Afraites, and H.~Notsu.
\newblock Detecting immersed obstacle in {S}tokes fluid flow using the coupled
  complex boundary method.
\newblock {\em SIAM J. Control \& Optim.}, 63(2):822--851, 2025.

\bibitem[RHA{\etalchar{+}}24]{RabagoHadriAfraitesHendyZaky2024}
J.~F.~T. Rabago, A.~Hadri, L.~Afraites, A.~S. Hendy, and M.~A. Zaky.
\newblock A robust alternating direction numerical scheme in a shape
  optimization setting for solving geometric inverse problems.
\newblock {\em Comput. Math. Appl.}, 175:19--32, September 2024.

\bibitem[Sim80]{Simon1980}
J.~Simon.
\newblock Differentiation with respect to the domain in boundary value.
\newblock {\em Numer. Funct. Anal. Optim.}, 2:649--687, 1980.

\bibitem[SZ92]{SokolowskiZolesio1992}
J.~Soko\l{}owski and J.-P. Zol\'{e}sio.
\newblock {\em {I}ntroduction to {S}hape {O}ptimization: {S}hape {S}ensitivity
  {A}nalysis}.
\newblock Springer Series in Computational Mathematics. Springer-Verlag,
  Berlin, Heidelberg, 1992.

\bibitem[YHG17]{YanHouGao2017}
W.~Yan, J.~Hou, and Z.~Gao.
\newblock Shape identification for convection--diffusion problem based on the
  continuous adjoint method.
\newblock {\em Appl. Math. Lett.}, 64:74--80, 2017.

\bibitem[YM06]{YanMa2006}
W.~Yan and Y.-C. Ma.
\newblock The application of domain derivative for heat conduction with mixed
  condition in shape reconstruction.
\newblock {\em Appl. Math. Comp.}, 181(2):894--902, 2006.

\bibitem[YM08]{YanMa2008}
W.~Yan and Y.-C. Ma.
\newblock Shape reconstruction of an inverse {S}tokes problem.
\newblock {\em J. Comput. Appl. Math.}, 216(2):554--562, 2008.

\bibitem[YSJ14]{YanSuJing2014}
W.~Yan, J.~Su, and F.~Jing.
\newblock Shape reconstruction for unsteady advection-diffusion problems by
  domain derivative method.
\newblock {\em Abstr. Appl. Anal.}, 2014(Art. ID 673108):7 pp., 2014.

\end{thebibliography}
%
%
%
\appendix
\section{Differentiability of the state variables}\label{appx:differentiability_of_the_state}
In these proofs, we streamline notation by omitting the subscript ${}_{N}$. 
Furthermore, we introduce a generic constant $c>0$, which remains independent of $t$ and may assume different values in varying contexts. 
Lemma~\ref{lem:continuity_and_coercivity_of_at} shows the continuity and coercivity of $a^t$. Lemma~\ref{lem:solution_of_transformed_perturbed_problem} describes the solution in the transformed perturbed domain while Lemma~\ref{Lem:IFT} establishes that $u^t$ is of class $C^1$ in a neighborhood of 0. Subsequently, by applying the implicit function theorem, we demonstrate the existence of the material derivative.

\begin{lemma}\label{lem:continuity_and_coercivity_of_at}
	Given the assumptions in \eqref{eq:assumptions}, assume that $\abs{\bb}_{\infty}$ is sufficiently small with   
 $\bb~\circ~T_{t}~=~ \bb^t$ and $\sigma \circ T_{t} = \sigma^{t}$.
	Then, the map $a^{t}: \interval \times \HGamma \times \HGamma \to \mathbb{R}$ given by
	\[
		(t,\varphi,\psi) \quad \longmapsto\quad \intO{\sigma^{t} A_{t} \nabla \varphi \cdot \nabla \psi} +\intO {  \bb^t \cdot C_{t}\nabla \varphi \psi}
	\]
	is a continuous and coercive bilinear form on $\HGamma \times \HGamma$ which satisfies
	\[
		a^{t}(\varphi,\varphi) \geqslant c \|\varphi\|_{\HGamma}^{2},
	\]
	for some positive constant $c:=c(\Omega)$ that is independent of $t$. 
\end{lemma}
\begin{proof}
	Assume \eqref{eq:assumptions} and that $\abs{ \bb^t}_{\infty}$ is sufficiently small -- to be specified later in the proof.
	Then, the following estimate holds
	\begin{align*}
		\abs{a^{t}(\varphi,\psi)} 
		&= \abs{\intO{\sigma^{t} A_{t} \nabla \varphi \cdot \nabla \psi} + \intO {  \bb^t \cdot C_{t}\nabla \varphi \psi}}\\
		&\leqslant c\sup_{t\in \interval} \left( |\sigma^{t}|_{\infty} \abs{A_{t}}_{\infty} + \abs{ \bb^t}_{\infty} \abs{C_{t}}_{\infty} \right) \norm{\varphi}_{\HGamma} \norm{\psi}_{\HGamma}\\
		&\leqslant c\sup_{t\in \interval} \left( \sigma^{t}_{1} \abs{A_{t}}_{\infty} + \abs{ \bb^t}_{\infty} \abs{C_{t}}_{\infty} \right) \norm{\varphi}_{\HGamma} \norm{\psi}_{\HGamma}\\
		&\leqslant c \norm{\varphi}_{\HGamma} \norm{\psi}_{\HGamma},	
	\end{align*}
	which shows the continuity of the given map.~
	On the other hand, for the coercivity of $a^{t}$, we use the boundedness of $A_{t}$ given in \eqref{eq:bounds}
	from which we obtain
	\[
		(\constant_{1}-1) |\xi|^2 \leqslant (A_{t}-\vect{I} )\xi \cdot \xi \leqslant (\constant_{2}-1)|\xi|^2.
	\]
	Then, we have the following estimate
	\begin{align*}
		a^{t}(\varphi,\varphi)
			&= \intO{\sigma^{t} \nabla \varphi \cdot \nabla \varphi} 
				+ \intO{\sigma^{t} \left(  A_{t}-\vect{I}  \right) \nabla \varphi \cdot \nabla \varphi}
				+ \intO { C_{t}^{\top}   \bb^t \cdot \nabla \varphi \varphi}\\
			&\geqslant \fergy{c} \sigma^{t}_0 \left( 1 +\constant_{1}-1\right) \norm{\varphi}_{\HGamma}^{2}
				- \abs{ \bb^t}_{\infty} \sup_{t \in \interval}  \abs{C_{t}}_{\infty}\norm{\varphi}_{\HGamma}^{2}\\
			&\geqslant \left( \fergy{c} \constant_{1} \sigma^{t}_{0} - \abs{ \bb^t}_{\infty} \sup_{t \in \interval}  \abs{C_{t}}_{\infty}\right) \norm{\varphi}_{\HGamma}^{2},
	\end{align*}
	\fergy{for some constant $c > 0$.}
	So, for sufficiently small $\abs{ \bb^t}_{\infty}$, more specifically, if $\abs{ \bb^t}_{\infty}$ is such that
	\[
		\abs{ \bb^t}_{\infty} < \frac{\correction{\constant}\sigma^{t}_{0}}{\sup_{t \in \interval}  \abs{C_{t}}_{\infty}}, 
	\]
	then
	\[
		a^{t}(\varphi,\varphi) \geqslant c \norm{\varphi}_{\HGamma}^{2},
	\]
	for some positive constant $c:=c(\Omega)$ that is independent of $t$. 
\end{proof}
\begin{lemma}\label{lem:solution_of_transformed_perturbed_problem}
For any ${\psi} \in \HGamma$, the function $u^{t} \in \HGamma$ solves the equation
\begin{equation}\label{eq:transformed_perturbed_problem}
\begin{aligned}
	a^{t} (u^{t}, \psi) = \intS { g \psi }.
\end{aligned}
\end{equation}
\end{lemma}
\begin{proof}
Let $u_{t}=u(\Omega_{t})$ the solution of problem $\eqref{eq:state_un}$.
Then, we have
\fergy{
\[
       \intOt {\sigma \nabla{u}_{t} \cdot  \nabla \psi_{t}}
        +\intOt {\bb \cdot \nabla{u}_{t} \psi_{t} }
        = \intS {{g}_{t} \psi_{t}}, \qquad \forall \psi \in \HGamma,
\]} 
where $\psi_{t}=0$ on $\Gamma $ and $\nabla(u_{t})\circ T_{t} =(DT_{t})^{-\top}
\nabla{u}^{t}$ \ferj{(see, e.g., \cite[Eq.~(71)]{BacaniPeichl2013})} with $u^{t} \in \HGamma$ and \fergy{$g_{t} \circ T_{t} =g^{t}=g \in H^{1/2}(\Sigma)$} and 
$\bb \circ T_{t} = \bb^t$ and 
$\sigma \circ T_{t} = \sigma^{t}$.
By the change of variable, the transported function $u^{t}({x})=(u(\Omega_{t}) \circ T_{t})(x)$, $x \in \Omega$, solves the following variational equation
\begin{equation}\label{eq:transported_ut}
\begin{aligned}
   \intO{\sigma^{t} A_{t} \nabla{u}^{t} \cdot \nabla \psi}
    +\intO { \bb^t \cdot C_{t}\nabla{u}^{t} \psi} =
    \intS { g \psi }, \qquad  \forall \psi \in \HGamma,
\end{aligned}
\end{equation}
as desired.
\end{proof}
\begin{lemma}\label{Lem:IFT}
The solution $t\mapsto u^{t}$ of \eqref{eq:transformed_perturbed_problem} is ${{C}}^{1}$ in a neighborhood of $~0$.
\end{lemma}
\begin{proof}
To prove the claim, we will apply the implicit function theorem (IFT).
Upon careful examination of \eqref{eq:transformed_perturbed_problem}, it can be verified that $u^t-u$ represents the unique element in $\HGamma$ satisfying the variational equation
\[
\begin{aligned}
	&\intO{\sigma^{t} A_{t} \nabla (u^{t}-u) \cdot \nabla \psi}
	+\intO { \bb^t \cdot C_{t}\nabla (u^{t}-u) \psi}\\
	&\qquad = \intS { g \psi }
    	- \intO{\sigma^{t} A_{t} \nabla{u} \cdot \nabla \psi}
    	- \intO { \bb^t \cdot C_{t}\nabla{u} \psi}, 
	\quad  \forall \psi \in \HGamma,
\end{aligned}
\]
Using the duality pairing $\langle \cdot,\cdot\rangle$ between $\HGamma$ and its dual space ${V}^{\prime}(\Omega)$, we can define a function $\mathcal{F}: \interval \times \HGamma\longrightarrow {V'}(\Omega)$ by
\[
\begin{aligned}
\langle \mathcal{F}(t,\varphi),\psi\rangle 
	&= \intO{\sigma^{t} A_{t} \nabla (\varphi +u) \cdot \nabla \psi}
		+ \intO { \bb^t \cdot C_{t}\nabla (\varphi + u) \psi} 
		- \intS { g \psi }\\
	&=a^t(\varphi +u,\psi)- \intS { g \psi }, \quad (\varphi, \psi \in \HGamma).
\end{aligned}
\]
\fergy{Above, it suffices to assume relaxed regularities for the data and the domain to establish the boundedness of the map $\mathcal{F}$ through a duality pairing argument.
Because in \eqref{eq:regular_maps} the maps $[t \mapsto A_{t}]$, $[t \mapsto C_{t}]$, $[t \mapsto \sigma^{t}]$, and $[t \mapsto \bb^{t}]$ are ${{C}}^1$ in a neighborhood of $0$, then clearly $\mathcal{F}$ is ${{C}}^{1}$.}
\color{black}
Then, taking $\varphi = u^{t}-u \in \HGamma$, we have
\[
\begin{aligned}
\langle \mathcal{F}(t,u^{t}-u),\psi\rangle&= \intO{\sigma^{t} A_{t} \nabla{u}^{t} \cdot \nabla \psi}
    +\intO { \bb^t \cdot C_{t}\nabla{u}^{t} \psi} -
    \intS { g \psi }\\
    &=a^{t} (u^{t}, \psi) - \intS { g \psi }=0, \quad  \forall \psi \in \HGamma.
\end{aligned}
\]
The next step is to show that there exists a unique function $k$, a mapping $t \mapsto u^t-u$ from a neighborhood of $0$ to 
$\HGamma$ such that $\mathcal{F}(t,k(t))=0$.
To accomplish the task, let us note that $u^{t}-u$ solves uniquely $\mathcal{F}(t,u^{t}-u)=0$ in $\HGamma$. 
In addition, we see that
\fergy{
\[
	\langle \mathcal{F}(0,\varphi),\psi\rangle-\langle \mathcal{F}(0,0),\psi\rangle=\langle D_{\varphi}\mathcal{F}(0,0)\varphi,\psi\rangle=a^{t} (\varphi+ u, \psi)-a^{t} (u, \psi) =a^{t} (\varphi, \psi).
\]
}
By Riesz' representation theorem, with $\HGamma$ being a Hilbert space, we obtain
\[
	\langle D_{\varphi}\mathcal{F}(0,0)\varphi,\psi\rangle=d_{\varphi} \mathcal{F}(0,0)(\varphi,\psi)=a^{t} (\varphi, \psi).
\]
Using Lemma~\ref{lem:continuity_and_coercivity_of_at}, we deduce via Lax-Milgram lemma that $d_{\varphi} \mathcal{F}(0,0)$ is an isomorphism from $V({\Omega})$ to $V'({\Omega})$, and we conclude by IFT that the map $k$ given by $[t \mapsto u^{t}-u]$ is ${{C}}^{1}$ in a neighborhood of $0$.

To finish the proof, we will demonstrate that \eqref{eq:material_un} actually holds.
For this purpose, let us denote by $\dot{u} \in \HGamma$ the derivative of the map $[t \mapsto u^{t}-u] \in {{C}}^{1}([-\varepsilon, \varepsilon];\HGamma)$, $\varepsilon > 0$ sufficiently small, as $t \to 0$.
Differentiating the equation $\mathcal{F}(t,u^{t}-u)=0$ with respect to $t$, leads to
\fergy{
\[
a^{t} (\dot{u}, \psi) - l(u;\psi)
	= \langle D_{\varphi}\mathcal{F}(0,0)\dot{u},\psi\rangle
	+ \langle \dfrac{\partial}{\partial t} \mathcal{F}(0,0),\psi\rangle 
	= 0, \quad \forall\psi \in \HGamma,
\]
}
where $l$ is given by \eqref{eq:right_el}.
\end{proof}
\section{Proof of key identities}\label{appx:proofs_of_key_identities}
Here we provide proofs to the key identities used in this study. 

We start with the proof of identity \eqref{eq:grad_sig_gradu}.
\begin{proof}[Proof of identity \eqref{eq:grad_sig_gradu}]
Let  $\VV=(\theta_{1},\cdots,\theta_{d})^{\top}$ and $\sigma = \sigma(x)$, $x \in \mathbb{R}^{d}$, be differentiable.
For notational convenience, we write $\partial_{i} := \partial/\partial{x_{i}}$.
For example, $\sigma \nabla{u} =(\sigma \partial_{1}{u},\ldots, \sigma \partial_{d}{u})^{\top}$.
Now, by expansion, we have
\[ 
\nabla(\sigma \nabla{u} )
	=\left( \sigma \partial_{j}(\partial_{i}{u})+ \partial_{j}{\sigma}\partial_{i}{u} \right)_{{i, j}} 
	=\sigma \nabla^{2} u + \left(\partial_{j}{\sigma}\partial_{i}{u}\right)_{{i, j}},
	\qquad (1 \leqslant i, j \leqslant d).
\]
Additionally, let us note that
\[
\begin{aligned}
\left(\partial_{j}{\sigma}\partial_{i}{u}\right)_{{i, j}}\VV \cdot \nabla{v}
=\sum_{i=1}^{d}  \sum_{j=1}^{d} \partial_{j}{\sigma}\partial_{i}{u}\theta_{j} {\partial_{j}{v}}
=\sum_{i=1}^{d} \partial_{i}{\sigma}(\nabla{u} \cdot \nabla{v})\theta_{i}
&=\left(\sum_{i=1}^{d} \partial_{i}{\sigma}\theta_{i}\right)(\nabla{u} \cdot \nabla{v}) \\
&=(\nabla \sigma \cdot \VV)(\nabla{u}\cdot \nabla{v})
\end{aligned}
\]
Thus, we have
\[
\begin{aligned}
\nabla(\sigma \nabla{u} )\VV \cdot \nabla{v}
	&=\sigma \nabla^{2} u\VV \cdot \nabla{v} + \left(\partial_{j}{\sigma}\partial_{i}{u}\right)_{{i, j}}\VV \cdot \nabla{v}\\
	&=\sigma \nabla^{2} u \VV \cdot \nabla{v} + (\nabla \sigma \cdot \VV)(\nabla{u}\cdot \nabla{v}).
\end{aligned} 
\]
\end{proof}
\begin{proof}[Proof of the first identity in Lemma~\ref{eq:expressions_j}]
Because $\div {\curl \varphi}=0 $ for all $\varphi \in \HGamma $ and $u=p=0$ on $\Gamma$ while $\VV \in \sfTheta$ (i.e., $\VV = \vect{0}$ on $\Sigma$), then the application of integration-by-parts clearly yields
\[
\begin{aligned}
\intO{\curl(\sigma \nabla{u} \times \VV)\cdot \nabla{p}}
&=-\intO{\div{\curl(\sigma \nabla{u} \times \VV)}p}\\
&\qquad +\intdO{\curl (\sigma \nabla{u} \times \VV)\cdot {p}{\nn}}\\
&=0.
\end{aligned}
\]
The same holds when $u$ and $p$ are interchanged.
\end{proof}
\begin{proof}[Proof of the second identity in Lemma~\ref{eq:expressions_j}]
Let \fergy{$(u,p) \in  [\HGamma \cap H^{2}(\Omega)]^{2}$} and $\VV \in \sfTheta$.
Then, by straightforward computations, we have
\[
\begin{aligned}
 -\intO{{\dd} ( \sigma \nabla{p} \cdot \nabla{u})}
&= -\intG{( \sigma \nabla{p} \cdot \nabla{u}) \Vn } 
+ \intO{\VV \cdot \nabla ( \sigma \nabla{p} \cdot \nabla{u})}\\
&=-\intG{( \sigma \nabla{p} \cdot \nabla{u}) \Vn }
+\intO{\VV \cdot \left[\nabla^{\top} (\sigma \nabla{p})\nabla{u}+\nabla^{2} u (\sigma \nabla{p}) \right]}\\
&=-\intG{( \sigma \nabla{p} \cdot \nabla{u}) \Vn }
+\intO{\nabla (\sigma \nabla{p})\VV \cdot \nabla{u}}
+\intO{\sigma \nabla^{2} u \VV \cdot \nabla{p} }.
\end{aligned}
\]
Because $u = p = 0$ on $\Gamma$, then it immediately follows that
\[ 
    \intO{\nabla(\sigma\nabla{p})\VV \cdot \nabla{u}}
    +\intO{\sigma \nabla^{2} u \VV \cdot \nabla{p}}
    +\intO{ {\dd}(\sigma\nabla{p} \cdot \nabla{u})}
    = \intG{ \sigma \dn{p}\dn{u} \Vn }.
\]
\end{proof}
\begin{proof}[Proof of identity \eqref{eq:vanishing_term}]
Let us reconsider \eqref{eq:important_identity} with $\varphi = p \in \HGamma$, $\psi = \nabla {u} \cdot \VV$, $u \in \HGamma$, $\VV \in \sfTheta$, and $\vect{F} = \bb \in W^{1,\infty}(\Omega)^{d}$.
Then, in particular, $p = 0$ on $\Gamma$ and $\VV = \vect{0}$ on $\Sigma$, and we get
\begin{align*}
\intO{ \left[ {p}  (\nabla {u} \cdot \VV)  \operatorname{div}{\bb} + (\bb \cdot \nabla {p} ) (\nabla {u} \cdot \VV)   + (\bb \cdot \nabla (\nabla {u} \cdot \VV)  ) {p}  \right]}
	&= \intdO{ {p}  (\nabla {u} \cdot \VV)  (\bb \cdot \nn) }
	= 0,\\
\intO{ \left[ {p}  (\nabla {u} \cdot \bb) \dd + (\VV \cdot \nabla {p} ) (\nabla {u} \cdot \bb)   + (\VV \cdot \nabla (\nabla {u} \cdot \bb)  ) {p}  \right]}
	&= \intdO{ {p}  (\nabla {u} \cdot \bb)  (\VV \cdot \nn) } 
	= 0.
\end{align*}
Subtracting the second equation from the first one, we obtain the following sequence of equalities
\[
\begin{aligned}
&\intO{ \left[ (\bb\cdot \nabla{p})(\VV\cdot \nabla{u}) - (\bb \cdot\nabla{u})(\VV \cdot \nabla{p}) \right]}\\
&\quad =
    \intO{p(\nabla{u}\cdot \bb){\dd}}
    - \intO{p(\nabla{u}\cdot \VV)\operatorname{div}{\bb}}
    + \intO{\left[ \VV\cdot \nabla(\nabla{u}\cdot \bb)p - \bb\cdot \nabla(\nabla{u}\cdot \VV)p \right]}\\
&\quad =
    \intO{p(\nabla{u}\cdot \bb){\dd}}
    - \intO{p(\nabla{u}\cdot \VV)\operatorname{div}{\bb}}
    + \intO{\left[ D\bb( \VV \cdot  \nabla{u})p - D\VV (\bb\cdot \nabla{u})p \right]},\\
& \quad = \intO{\curl(\bb\times \VV)\cdot\nabla{u}{p}},
\end{aligned}
\]
where the second equation line follows from the fact that $\nabla^{2}{u}$ is symmetric.
\end{proof}
\end{document}